\documentclass{amsart}
\usepackage[all]{xy}
\usepackage{amsmath,amsfonts,amssymb,amsthm,epsfig,amscd,comment,latexsym,psfrag}
\usepackage{nicefrac,xspace,tikz}
\usetikzlibrary{arrows}

\title{A categorification of twisted Heisenberg algebras}
\author{David Hill}
\author{Joshua Sussan}
\date{\today}

\newtheorem{defined}{Definition}
\newtheorem{prop}{Proposition}
\newtheorem{theorem}{Theorem}
\newtheorem{corollary}{Corollary}
\newtheorem{lemma}{Lemma}

\newtheorem*{thm}{Theorem}
\newtheorem{ex}{Example}

\begin{document}

\maketitle
\baselineskip 14pt

\def\adj{\text{adj}}
\def\tr{\text{tr}}
\def\R{\mathbb R}
\def\Q{\mathbb Q}
\def\Z{\mathbb Z}
\def\N{\mathbb N}
\def\C{\mathbb C}
\def\l{\lbrace}
\def\r{\rbrace}
\def\o{\otimes}
\def\lra{\longrightarrow}
\def\Hom{\mathrm{Hom}}
\def\HOM{\mathrm{HOM}}
\def\RHom{\mathrm{RHom}}
\def\Id{\mathrm{Id}}
\def\mc{\mathcal}
\def\mf{\mathfrak}
\def\Ext{\mathrm{Ext}}
\def\Ind{\mathrm{Ind}}
\def\Res{\mathrm{Res}}
\def\soc{\mathrm{soc}}
\def\hd{\mathrm{hd}}
\def\Seq{\mathrm{Seq}}
\def\shuffle{\,\raise 1pt\hbox{$\scriptscriptstyle\cup{\mskip
               -4mu}\cup$}\,}
\newcommand{\define}{\stackrel{\mbox{\scriptsize{def}}}{=}}


%
\def\drawing#1{\begin{center}\epsfig{file=#1}\end{center}}

 \def\yesnocases#1#2#3#4{\left\{
\begin{array}{ll} #1 & #2 \\ #3 & #4
\end{array} \right. }

\newcommand{\LOT}{H^-}

\begin{abstract}
We categorify a quantized twisted Heisenberg algebra associated to a finite subgroup of $ SL(2,\mathbb{C}) $.
\end{abstract}

\tableofcontents



\section{Introduction}
Categorifications of quantum groups associated to Dynkin diagrams were constructed by Khovanov and Lauda \cite{KL1, KL2} and independently by Rouquier \cite{Ro}.  In particular, there is a categorification of the quantum group associated to an affine Dynkin diagram in terms of its Drinfeld-Jimbo presentation.  A natural question is how to categorify its loop realization.  A first step towards this goal is the program of categorifying Heisenberg algebras which was initiated by Khovanov ~\cite{Kh} who constructed a diagrammatic category $ \mathcal{H} $ whose Grothendieck group contains a certain Heisenberg algebra $ \mathfrak{h}$.
Conjecturally the Grothendieck group is isomorphic to the Heisenberg algebra $ \mathfrak{h}$.

Cautis and Licata defined a diagrammatic graded 2-category $ \mathcal{H}_{\Gamma} $ depending on a finite subgroup $ \Gamma $ of $ SL(2,\C) $.  They proved that the Grothendieck group of $ \mathcal{H}_{\Gamma} $
is isomorphic to a quantized Heisenberg algebra $ \mathfrak{h}_{q,\Gamma} $ associated to the group $ \Gamma $.  Relations in the Heisenberg algebra are lifted to isomorphisms of 1-morphisms in their diagrammatic category.  For example, a relation in the Heisenberg algebra $ q_i p_i = p_i q_i + [2] $ becomes an isomorphism $ Q_i P_i \cong P_i Q_i \oplus \Id \langle1 \rangle \oplus \Id \langle -1 \rangle $ where $ \Id \langle r \rangle $ is the identity 1-morphism shifted up by degree $ r $.
A key step in their identification of the Grothendieck group is a 2-representation of their diagrammatic 2-category on a 2-category of modules for a wreath product associated to symmetric groups and $ \Gamma $.

This suggests that we should be able to extend the Cautis-Licata construction to a categorification of twisted Heisenberg algebras $ \mathfrak{h}_{q,\Gamma}^-$ by replacing the symmetric group with its super version- the Hecke-Clifford algebra.
This is the goal of this paper (in the language of monoidal categories instead of 2-categories).  We rely heavily upon ~\cite{CL} along with the work of Frenkel, Jing, and Wang ~\cite{FJW}, ~\cite{JW}, who utilize the Hecke-Clifford algebra to construct the Fock space representation of $ \mathfrak{h}_{\Gamma}^-$.
An alternate approach to that construction is to use the spin symmetric group ~\cite{FJW}.  A topological construction of twisted Heisenberg algebras was introduced by Wang ~\cite{W}.
We define here a diagrammatic $\Z \times \Z_2$-graded monoidal category $ \mathcal{H}_{\Gamma}^- $.
One relation in the twisted Heisenberg algebra is $ q_i p_i = p_i q_i + 2[2] $.  This relation becomes an isomorphism of objects
$ Q_i P_i \cong P_i Q_i \oplus \Id \langle1 \rangle \oplus \Id \langle -1 \rangle \oplus  \Id \langle1 \rangle \lbrace 1 \rbrace \oplus  \Id \langle-1 \rangle \lbrace 1 \rbrace $.  If the finite group $ \Gamma $ is trivial, this isomorphism appears in the context of induction and restriction of modules for the Hecke-Clifford algebra (see \cite{Klesh} or \cite{CS} for example).
Our main theorem is:

\begin{thm}
There is a isomorphism of algebras
$ \phi \colon \mathfrak{h}_{q,\Gamma}^- \rightarrow K_0(\mathcal{H}_{\Gamma}^-) $.
\end{thm}

As in \cite{CL}, we construct a representation of the diagrammatic category on categories of modules for a wreath product associated to $ \Gamma $ and a Clifford algebra with the symmetric group in order to prove the homomorphism above is an injection.

It would be interesting to find a twisted analogue of the construction of Shan and Vasserot who categorify the action of the Heisenberg algebra on Fock space using a category of modules for cyclotomic rational double affine Hecke algebras ~\cite{SV}.
Another natural problem is to understand a twisted version of  ~\cite{SW} where Stroppel and Webster construct an action of quantum $\hat{\mathfrak{sl}}_e$ on tensor products of fermionic Fock spaces.

\subsection*{Outline of the paper}
The twisted Heisenberg algebras, which are the objects to be categorified, are reviewed in Section ~\ref{heisalgebra}.
In Section ~\ref{preliminaries} we give some background material on Clifford algebras, Hecke-Clifford algebras, wreath products, monoidal categories, and Grothendieck groups.
Section ~\ref{catH} contains the definition of the main diagrammatic category and its Karoubi envelope which categorifes the Heisenberg algebra.  A homomorphism from the Heisenberg algebra $ \mathfrak{h}_{q,\Gamma}^- $ to the Grothendieck group of the appropriate category is constructed in this section although its proof relies heavily on the results of Section ~\ref{auxcat}.  The proof relies on the construction of an auxiliary category where relations between objects are easier to prove.  There is a functor from the auxiliary category to the category defined in Section ~\ref{catH}.  This functor transports the isomorphisms between the objects in the auxiliary category to relations between the objects in Section ~\ref{catH} giving rise to the desired homomorphism described above.
In Section ~\ref{2rep} we consider a category associated to the wreath product attached to the data of Hecke-Clifford algebra of all ranks and the finite group $ \Gamma $.  There is an action
of $ \mathcal{H}_{\Gamma}^- $ on this category which allows us to prove that the homomorphism from $ \mathfrak{h}_{q,\Gamma}^- $ to the Grothendieck group of $ \mathcal{H}_{\Gamma}^- $ is an isomorphism in Section
~\ref{mainresult}.

\subsection*{Acknowledgements}
The authors would like to thank Mikhail Khovanov, Naihuan Jing, and Weiqiang Wang for helpful conversations.  They are also very grateful to Sabin Cautis and Anthony Licata for explaining details of their paper.
This project began during the second author's visit to the Max Planck Institute for Mathematics and is very grateful for its support and excellent working conditions.

J.S. was supported by NSF grant DMS-1407394 and PSC-CUNY Award 67144-00 45.

\section{The twisted Heisenberg algebra $ \mathfrak{h}_{q,\Gamma}^- $}
\label{heisalgebra}
Let $ \mathfrak{g} $ be a finite dimensional complex simple Lie algebra associated to a simply laced Dynkin diagram with Cartan matrix $ (a_{ij}) $.  Let the set of simple roots of $ \mathfrak{g} $ be
$ \lbrace \alpha_1, \ldots, \alpha_r \rbrace $.  Fix a Cartan subalgebra $ \mathfrak{h} $ of $ \mathfrak{g} $ with basis $ \lbrace h_1, \ldots, h_r \rbrace $ dual to the set of simple roots.
Let $ e_i $ and $ f_i $ be Chevalley generators corresponding to the roots $ \alpha_i $ and $ -\alpha_i $ respectively.

To the finite dimensional algebra $ \mathfrak{g} $ we associate to it its affine algebra and extended affine algebra respectively
\begin{equation*}
\hat{\mathfrak{g}} = \mathfrak{g} \otimes \mathbb{C}[t,t^{-1}] \oplus \mathbb{C} c \hspace{.4in} \tilde{\mathfrak{g}} = \hat{\mathfrak{g}} \oplus \mathbb{C} d.
\end{equation*}

By the McKay correspondence, all such $ \widetilde{\mathfrak{g}} $ and hence $ \hat{\mathfrak{g}} $ are in bijection with finite subgroups $ \Gamma $ of $ SL_2(\mathbb{C}) $.  Denote
the corresponding affine algebra by $ \hat{\mathfrak{g}}_{\Gamma} $.

Let $ \theta $ be the involution of $ \mathfrak{g} $ such that
\begin{equation*}
e_i \mapsto f_i \hspace{.5in} f_i \mapsto e_i \hspace{.5in} h_i \mapsto -h_i.
\end{equation*}
Define
\begin{equation*}
\mathfrak{g}_0 = \lbrace g \in \mathfrak{g} | \theta(g) = g \rbrace \hspace{.5in} \mathfrak{g}_1 = \lbrace g \in \mathfrak{g} | \theta(g) = -g \rbrace.
\end{equation*}

Now we define the twisted affine algebra
\begin{equation*}
\hat{\mathfrak{g}}_{\Gamma}[-1] = \bigoplus_{i=0}^1 \mathfrak{g}_i \otimes t^i \mathbb{C}[t^2, t^{-2}] \oplus \mathbb{C}c
\end{equation*}
with relations
\begin{equation*}
[x \otimes t^m, y \otimes t^n] = (m+n)[x,y] + \frac{m}{2} \delta_{m,-n}(x,y)c
\end{equation*}
\begin{equation*}
[c, x \otimes t^n]=0
\end{equation*}
where $ (,) $ is the Killing form on $ \mathfrak{g} $.

For this paper we will assume that the central extension $ c $ always acts by $ 1 $ so we define the twisted Heisenberg algebra as the enveloping algebra
\begin{equation*}
\mathfrak{h}_{\Gamma}^-  = \mathcal{U}((\mathfrak{h} \otimes t \mathbb{C}[t^2, t^{-2}] \oplus \mathbb{C}c)/ (c-1)).
\end{equation*}

Setting $ h_i(\frac{m}{2}) = h_i \otimes t^m $, we see that $ \mathfrak{h}_{\Gamma}^- $ is generated by
\begin{equation*}
\lbrace h_i(\mbox{$\frac{m}{2}$}) | i=1, \ldots, r \text{ and } m \in 2 \mathbb{Z} +1 \rbrace
\end{equation*}
with relations
\begin{equation*}
[h_i(\mbox{$\frac{m}{2}$}), h_j(\mbox{$\frac{n}{2}$})]=\frac{m}{2} a_{ij} \delta_{m,-n}.
\end{equation*}

\begin{prop}
The twisted Heisenberg algebra $\mathfrak{h}_{\Gamma}^- $ is generated by
\begin{equation*}
\lbrace p_i^{(m)}, q_i^{(m)} | i = 1, \ldots, r, \text{ and } m \in \Z_{\geq 0} \rbrace
\end{equation*}
with relations:
\begin{enumerate}
\item $ p_i^{(n)} p_j^{(m)} = p_j^{(m)} p_i^{(n)} $ for all $ i,j $
\item $ q_i^{(n)} q_j^{(m)} = q_j^{(m)} q_i^{(n)} $ for all $ i,j $
\item $ q_i^{(n)} p_i^{(m)} = p_i^{(m)} q_i^{(n)} + \sum_{k \geq 1} (-1)^k 4k p_i^{(m-k)} q_i^{(n-k)} $ for all $ i $
\item $ q_i^{(n)} p_j^{(m)} = p_j^{(m)} q_i^{(n)} + \sum_{k \geq 1} 2 p_j^{(m-k)} q_i^{(n-k)} $ for all $ i, j $
with $ a_{ij}=-1 $
\item $ q_i^{(n)} p_j^{(m)} = p_j^{(m)} q_i^{(n)} $ for all $ i, j $ with $ a_{ij} = 0 $.
\end{enumerate}
\end{prop}

\begin{proof}
Let
\begin{equation*}
\sum_{{m} \in \frac{1}{2}\N} q_i^{(2m)} z_1^{-m} =  \text{exp}(\sum_{m \in \N + \frac{1}{2}} \frac{h_i(m)}{m} z_1^{-m})
\end{equation*}

\begin{equation*}
\sum_{{m} \in \frac{1}{2}\N} p_j^{(2m)} z_2^{m} =  \text{exp}(\sum_{m \in \N + \frac{1}{2}} -\frac{h_j(-m)}{m} z_2^{m}).
\end{equation*}
Now the first, second, and fifth relations are clear.

By ~\cite[Proposition 3.4.1]{FLM},
\begin{equation}
\label{vertexcomm}
\left[\sum_{{m} \in \frac{1}{2}\N} q_i^{\left(2m\right)} z_1^{-m}\right]\left[\sum_{{m} \in \frac{1}{2}\N} p_j^{\left(2m\right)} z_2^{m}\right]=\left[\sum_{{m} \in \frac{1}{2}\N} p_j^{\left(2m\right)} z_2^{m}\right]\left[\sum_{{m} \in \frac{1}{2}\N} q_i^{\left(2m\right)} z_1^{-m}\right]
\left(\frac{1-z_2^{\frac{1}{2}}z_1^{\frac{-1}{2}}}{1+z_2^{\frac{1}{2}}z_1^{\frac{-1}{2}}}\right)^{a_{ij}}
\end{equation}

If $ a_{ij} = -1 $, then
\begin{eqnarray*}
\left(\frac{1-z_2^{\frac{1}{2}}z_1^{\frac{-1}{2}}}{1+z_2^{\frac{1}{2}}z_1^{\frac{-1}{2}}}\right)^{a_{ij}} &=
& \left(\frac{1+z_2^{\frac{1}{2}}z_1^{\frac{-1}{2}}}{1-z_2^{\frac{1}{2}}z_1^{\frac{-1}{2}}}\right)\\
&= &\left(1+z_2^{\frac{1}{2}}z_1^{\frac{-1}{2}}\right)\left(1+ z_2^{\frac{1}{2}}z_1^{\frac{-1}{2}} + z_2^{{1}}z_1^{{-1}}+ z_2^{\frac{3}{2}}z_1^{\frac{-3}{2}} + \cdots \right) \\
& = & \left(1+2z_2^{\frac{1}{2}}z_1^{\frac{-1}{2}} + 2z_2^{{1}}z_1^{{-1}}+ 2z_2^{\frac{3}{2}}z_1^{\frac{-3}{2}} + \cdots \right).
\end{eqnarray*}
Substituting this expansion into \eqref{vertexcomm} gives the fourth relation.

If $ i=j$, then $ a_{ij} = 2 $.  Then \eqref{vertexcomm} gives
\begin{equation}
\label{vertexcomm2}
\left[\sum_{{m} \in \frac{1}{2}\N} q_j^{\left(2m\right)} z_1^{-m}\right]\left[\sum_{{m} \in \frac{1}{2}\N} p_i^{\left(2m\right)} z_2^{m}\right] =  \left[\sum_{{m} \in \frac{1}{2}\N} p_i^{\left(2m\right)} z_2^{m}\right]\left[\sum_{{m} \in \frac{1}{2}\N} q_j^{\left(2m\right)} z_1^{-m}\right]
\left(\frac{1-z_2^{\frac{1}{2}}z_1^{\frac{-1}{2}}}{1+z_2^{\frac{1}{2}}z_1^{\frac{-1}{2}}}\right)^2.
\end{equation}
Now we expand
\begin{eqnarray*}
\left(\frac{1-z_2^{\frac{1}{2}}z_1^{\frac{-1}{2}}}{1+z_2^{\frac{1}{2}}z_1^{\frac{-1}{2}}}\right)^2 &= &
\left(1- 2 z_2^{\frac{1}{2}}z_1^{\frac{-1}{2}} + z_2^{{1}}z_1^{{-1}}\right) \left(1- 2 z_2^{\frac{1}{2}}z_1^{\frac{-1}{2}} + 3 z_2^{{1}}z_1^{{-1}} - 4z_2^{\frac{3}{2}}z_1^{\frac{-3}{2}} + \cdots\right) \\
& = & \left(1- 4 z_2^{\frac{1}{2}}z_1^{\frac{-1}{2}} + 8 z_2^{{1}}z_1^{{-1}} - 12 z_2^{\frac{3}{2}}z_1^{\frac{-3}{2}} + \cdots\right).
\end{eqnarray*}
Substituting this series into \eqref{vertexcomm2} gives the third relation.
\end{proof}

\begin{ex}
We have the following formulas:
$ p_i^{(1)} = -2h_i(\frac{-1}{2}) $,
$ q_i^{(1)} = 2h_i(\frac{1}{2}) $,
$ p_i^{(2)} = 2h_i(\frac{-1}{2}) h_i(\frac{-1}{2}) $, and
$ q_i^{(2)} =  2h_i(\frac{1}{2})h_i(\frac{1}{2}) $.
\end{ex}

For a natural number $ k $, let $ [k] = q^{k-1} + q^{k-3} + \cdots + q^{1-k} $.
Now we come to the main object of study: the quantized twisted Heisenberg algebra.

\begin{defined}
\label{qheis}
Let $ \mathfrak{h}_{q,\Gamma}^- $ be the $ \mathbb{C}[q,q^{-1}] $ algebra generated by
\begin{equation*}
\lbrace p_i^{(m)}, q_i^{(m)} | i = 1, \ldots, r, \text{ and } m \in \Z_{\geq 0} \rbrace
\end{equation*}
with relations
\begin{enumerate}
\item $ p_i^{(n)} p_j^{(m)} = p_j^{(m)} p_i^{(n)} $ for all $ i,j $
\item $ q_i^{(n)} q_j^{(m)} = q_j^{(m)} q_i^{(n)} $ for all $ i,j $
\item $ q_i^{(n)} p_i^{(m)} = p_i^{(m)} q_i^{(n)} + \sum_{k \geq 1} 2([k+1]+[k-1]) p_i^{(m-k)} q_i^{(n-k)} $ for all $ i $
\item $ q_i^{(n)} p_j^{(m)} = p_j^{(m)} q_i^{(n)} + \sum_{k \geq 1} 2 p_j^{(m-k)} q_i^{(n-k)} $ for all $ i, j $
with $ a_{ij}=-1 $
\item $ q_i^{(n)} p_j^{(m)} = p_j^{(m)} q_i^{(n)} $ for all $ i, j $ with $ a_{ij} = 0 $.
\end{enumerate}
\end{defined}

\begin{prop}
The algebra $ \mathfrak{h}_{q,\Gamma}^- $ specializes to  $ \mathfrak{h}_{\Gamma}^- $.  That is, there is an isomorphism of $ \mathbb{C}$-algebras
$ f \colon \mathfrak{h}_{q=-1,\Gamma}^- \rightarrow \mathfrak{h}_{\Gamma}^- $.
\end{prop}

\begin{proof}
The only thing to note is that when $ q=-1 $, the expression $ 2([k+1]+[k-1]) $ is $ (-1)^k 4k $.
\end{proof}
See \cite{J2} and \cite{J} for more details on the relationship between quantized twisted Heisenberg algebras and quantized twisted affine algebras.

\section{Preliminaries}
\label{preliminaries}
\subsection{The Hecke-Clifford algebra $ \mathbb{S}_n $}
Let $ Cl_n $ be the Clifford algebra generated by $ c_i $ for $ i=1, \ldots, n $ with relations 
$$ c_i^2 = 1, \;\;\;\mbox{ and }\;\;\; c_i c_j = -c_j c_i \;\;\;\mbox{ for }\;\;\; i \neq j .$$

Let $ {S}_n $ be the symmetric group of permutations of $ n $ elements generated by transpositions $ s_i $ for $ i=1, \ldots, n-1 $.
Then, $ S_n $ acts on $ Cl_n $ by permutations on the subscripts of the Clifford generators.
When $S_n$ acts on a set $X$, we will denote the action of $ w \in S_n $ on an element $ x \in X $ by $ w.x $.
Then the Hecke-Clifford algebra $ \mathbb{S}_n $ is defined as the semi-direct product:
\begin{equation*}
\mathbb{S}_n = Cl_n \rtimes \C[S_n]
\end{equation*}

The Hecke-Clifford algebra is a $ \Z_2$-graded superalgebra where the symmetric group lies in degree zero and each Clifford generator $ c_i $ has degree one.  We denote the degree of an element $ x \in \mathbb{S}_n $ by $ ||x|| $. Throughout this paper, we will adopt the conventions on super representation theory described in \cite[Part II]{Klesh}.

The super representation theory of $ \mathbb{S}_n $ is equivalent to the projective representation theory of $ S_n $.  Thus the finite dimensional irreducible super representations of $ \mathbb{S}_n $ are indexed by strict partitions of $ n $. Nazarov gave a construction of (quasi-)idempotents in $ \mathbb{S}_n $ indexed by strict partitions of $ n $ ~\cite{N}.  His methods parallel Cherednik's construction of idempotents in
$ \mathbb{C}[S_n] $ using the affine Hecke algebra.
A different construction of (quasi-)idempotents $ \psi_{\lambda} $ was given later by Hecke-Clifford ~\cite{Se} which are the ones we use here.  For the special case that $ \lambda = (n) $, the element $ \psi_{\lambda} $ has a familiar expression:

\begin{equation}
\label{psi}
\psi_{(n)} = \frac{1}{n!} \sum_{w \in S_n} w.
\end{equation}


We will need some well known properties of $ \psi_{(n)} $ which follow directly from the definition of $ \psi_{(n)}$.

\begin{prop}
The element $ \psi_{(n)} $ is an idempotent:
$ \psi_{(n)} \psi_{(n)} = \psi_{(n)} $.
\end{prop}


\begin{prop}
\label{absorbformulas}
There are equalities
\begin{equation*}
s_k \psi_{(n)} =  \psi_{(n)}
\end{equation*}
\begin{equation*}
\psi_{(n)} s_k = \psi_{(n)}.
\end{equation*}
\end{prop}


For $ 0 \leq i \leq m $, let $ \rho_i \colon \mathbb{S}_n \rightarrow \mathbb{S}_{n+m} $ be a homomorphism of algebras given by
$ s_r \mapsto s_{r+i} $ and $ c_r \mapsto c_{r+i} $.
Then we set $ \psi^i_{\lambda} = \rho_i(\psi_{\lambda}) $.

We will make use of the next two propositions when manipulating graphical depictions of morphisms in the monoidal category to be defined later.
Their proofs are easy.

\begin{prop}
\label{idempotentslide}
There are equalities of elements in $ \mathbb{S}_{n+m} $:
\begin{enumerate}
\item $ (s_{i+1} \cdots s_{i+n}) \psi^i_{(n)} = \psi^{i+1}_{(n)}(s_{i+1} \cdots s_{i+n}) $
\item $ (s_{i+n} \cdots s_{i+1}) \psi^{i+1}_{(n)} = \psi^{i}_{(n)}(s_{i+n} \cdots s_{i+1}) $.
\end{enumerate}
\end{prop}

\begin{prop}
\label{idempotentabsorb}
There are equalities of elements in $ \mathbb{S}_{n+m} $:
\begin{equation*}
\psi_{(m+n)} \psi^i_{(n)} = \psi_{(m+n)} = \psi^i_{(n)} \psi_{(m+n)}.
\end{equation*}
\end{prop}


\subsection{Monoidal categories and Grothendieck groups}
Let $ \mathcal{C} $ be a $ \Z \times \Z_2 $-graded monoidal category with tensor product $ \otimes $.  We will usually write $P_1 \otimes P_2$ simply as $ P_1 P_2$.
If $ P $ is an object in $ \mathcal{C} $, then let the $ (i,j) $ graded piece of $ P $ be $ P_{i,j} $.
Define the shifted object $ P \langle r \rangle \lbrace s \rbrace $ by $ (P \langle r \rangle \lbrace s \rbrace)_{i,j} = P_{i-r,j-s} $.

The Grothendieck group $ K_0(\mathcal{C}) $ is the free abelian group generated by symbols $ [P] $ where $ P $ is a projective object of $ \mathcal{C} $ with relations
$ [P \oplus Q] = [P] + [Q] $ and $ [P] = [P \lbrace 1 \rbrace] $.
The Grothendieck group becomes a ring via the tensor product: $ [P][Q] = [P \otimes Q] $.
The Grothendieck group in fact has the structure of a $ \Z[q,q^{-1}]$-algebra by $ q^i[P] = [P \langle i \rangle] $.

\section{The category $ \mathcal{H}_{\Gamma}^- $}
\label{catH}

\subsection{The algebra $ B_{}^{\Gamma}$}
We recall the algebra $ B_{}^{\Gamma}$ defined in \cite{CL}.
Let $ V = \C^2 $ with fixed basis $ \lbrace v_1, v_2 \rbrace $ and $ \omega = v_1 \wedge v_2 $.  Let $ \Gamma $ a finite subgroup of $ SL(2,\C) $ with identity element $1$.  Then $ \Gamma $ acts on $ V $ and hence on the exterior algebra $ \Lambda^*(V) $ so we may define the group algebra of the semi-direct product
$ B^{\Gamma} := \Lambda^*(V) \rtimes \C[\Gamma] $.
This is a $ \Z$-graded algebra where the degree of $ (v, \gamma) $ is the degree of $ v $ in the exterior algebra $ \Lambda^*(V) $.  Denote the degree of an element $ x $ by $ |x| $.

Define a $ \C$-linear, supersymmetric, non-degenerate trace $ \tr \colon B^{\Gamma} \rightarrow \C $ by
$ \tr((f, \gamma)) = \delta_{f, \omega} \delta_{\gamma,1} 1. $
The associated non-degenerate bilinear form $ \langle, \rangle \colon B^{\Gamma} \times B^{\Gamma} \rightarrow \C $ is defined by
$ \langle a, b \rangle = \tr(ab)$.
If $ \mathcal{B} $ is a basis of $ B^{\Gamma} $, denote the dual basis with respect to $ \langle, \rangle $ by $ \check{\mathcal{B}} $ and if $ b \in \mathcal{B} $ let its dual be $ \check{b} $.

\subsection{The algebra $ B_{n}^{\Gamma} $}\label{BGn}

The symmetric group acts on the $n$-fold tensor product $ B^{\Gamma} \otimes \cdots \otimes B^{\Gamma} $ by the formula
\begin{equation*}
s_i.(b_1 \otimes \cdots \otimes b_i \otimes b_{i+1} \otimes \cdots \otimes b_n) = (-1)^{|b_i||b_{i+1}|}(b_1 \otimes \cdots \otimes b_{i+1} \otimes b_{i} \otimes \cdots \otimes b_n).
\end{equation*}
Consider the (super)algebra $( B^\Gamma  )^{\otimes n} \otimes Cl_n$, where $(B^\Gamma)^{\otimes n}\otimes 1$ has $\Z_2$-degree 0. We may extend the action of $S_n$ to this algebra, and form the semi-direct product
\begin{equation*}
{B}_n^{\Gamma} = [(B^{\Gamma})^{ \otimes n} \otimes Cl_n] \rtimes S_n.
\end{equation*}
Observe that this algebra contains the Hecke-Clifford algebra as a subalgebra. In particular, $B_n^\Gamma$ comes equipped with an action $\mathbb{S}_n$, where
\begin{equation*}
c_i.(b_1 \otimes \cdots \otimes b_n) c  w = (b_1 \otimes \cdots \otimes b_n) c_ic  w,
\end{equation*}
where $b_1,\ldots,b_n\in B^\Gamma$, $c\in Cl_n$, $w\in S_n$, and we have written $(b_1\otimes\cdots\otimes b_n) c w : = (b_1 \otimes \cdots \otimes b_n)\otimes c \otimes w$. This algebra is $ \Z \times \Z_2 $-graded. We denote the $ \Z $-degree of an element $ x $ is denoted by $ |x| $ and the $ \Z_2$-degree of an element $ x $ is denoted by $ ||x|| $. That is, for $b\in (B^\Gamma)^n$, $c\in Cl_n$ and $w\in S_n$, $|bcw|=|b|$, and $||bcw||=||c||$.

The finite group $\Gamma$ has $r$ conjugacy classes and, for a strict partition $ \lambda $ of $ n $ and $ 1 \leq i\leq r$, there exists an idempotent $ e_{i,1,\lambda} $ of $ {B}_n^{\Gamma} $.  See ~\cite[Section 3.1.1]{CL} for more details when the
Hecke-Clifford algebra is replaced by the group algebra of the symmetric group and refer to ~\cite[Section 1.2]{JW} for the necessary modifications in the Hecke-Clifford algebra case.

\subsection{The category $ \mathcal{H}_{\Gamma}'^{-} $}
We define an additive, $ \C$-linear, $ \Z \times \Z_2$-graded monoidal category $ \mathcal{H}_{\Gamma}'^{-} $ as follows.
The objects of $ \mathcal{H}_{\Gamma}'^{-} $ are generated by $ {P} $ and $ {Q} $.
A general object is a formal sum of finite composition of $ {P}$'s and $ {Q}$'s with graded shifts.
The identity object is denoted by $ \bf{1} $.

The space of morphisms between two objects is a $ \Z \times \Z_2$-graded $\C$-algebra generated by planar diagrams modulo certain relations.
Using the letters $ X_i $ and $ Y_i $ to denote either $ P $ or $ Q $, the set of morphisms $ \Hom_{\mathcal{H}_{\Gamma}'^{-}}(X_{1} \cdots X_r, Y_1 \cdots Y_s) $ is generated
by diagrams which are unions of oriented, compact, 1-manifolds immersed into $ \R \times [0,1] $ which carry solid dots labeled by elements of $ B^{\Gamma} $ and hollow dots which are not labeled
such that the boundary of each diagram intersects $ \R \times \lbrace 0 \rbrace $ at points labeled $ X_1, \ldots, X_r $ and intersects $ \R \times \lbrace 1 \rbrace $ at points labeled $ Y_1, \ldots, Y_s $ from left to right.
A strand connects a point labeled by $ P $ or $ Q $ in $ \R \times \lbrace 0 \rbrace $ to a point labeled by $ P $ or $ Q $ respectively in $ \R \times \lbrace 1 \rbrace $.  If the boundary of a strand is contained entirely in $ \R \times \lbrace 0 \rbrace $ or $ \R \times \lbrace 1 \rbrace $ then the boundary points are labeled by different objects.
Furthermore, we make the convention that a strand is oriented upwards near a boundary point labeled by $ P $ and oriented downwards near a boundary point labeled by $ Q $.
These diagrams are taken modulo isotopies which do not change the relative position of dots.

The diagrams in ~\eqref{idpidq} on the left and right denote $ \text{id} \colon {P} \rightarrow {P} $ and $ \text{id} \colon {Q} \rightarrow {Q} $ respectively.
\begin{equation}
\label{idpidq}
\begin{tikzpicture}
\draw (0,0) -- (0,1)[->][thick];
\draw (4,0) -- (4,1)[<-][thick];
\end{tikzpicture}
\end{equation}

The relations that we impose on the diagrams are as follows:

\begin{equation}
\label{H1}
\begin{tikzpicture}[>=stealth]
\draw (0,0) -- (1,1)[->][thick];
\draw (1,0) -- (0,1)[->][thick];
\filldraw [black] (.25,.25) circle (2pt);
\draw (0,.3) node{$b$};
\draw (1.5,.5) node{=};
\draw (2,0) -- (3,1)[->][thick];
\draw (3,0) -- (2,1)[->][thick];
\filldraw [black] (2.75, .75) circle (2pt);
\draw (2.95,.65) node{$b$};

\end{tikzpicture}
\end{equation}

\begin{equation}
\label{H2}
\begin{tikzpicture}[>=stealth]
\draw (0,0) -- (1,1)[->][thick];
\draw (1,0) -- (0,1)[->][thick];
\filldraw [black] (.25,.75) circle (2pt);
\draw (0,.7) node{$b$};
\draw (1.5,.5) node{=};
\draw (2,0) -- (3,1)[->][thick];
\draw (3,0) -- (2,1)[->][thick];
\filldraw [black] (2.75, .25) circle (2pt);
\draw (2.95,.35) node{$b$};

\end{tikzpicture}
\end{equation}

\begin{equation}
\label{H3}
\begin{tikzpicture}[>=stealth]
\draw (0,0) -- (0,-.5) [->][thick];
\filldraw [black] (0,-.25) circle (2pt);
\draw (0,-.25) node [anchor=east] [black] {$b$};
\draw (1,0) -- (1,-.5) [thick];
\draw (0,0) arc (180:0:.5) [thick];
\draw (1.5,0) node {=};
\draw (2,0) -- (2,-.5)[->] [thick];
\filldraw [black] (3,-.25) circle (2pt);
\draw (3,-.25) node [anchor=west] [black] {$b$};
\draw (3,0) -- (3,-.5) [thick];
\draw (2,0) arc (180:0:.5) [thick];

\draw (5,0) -- (5,-.5) [thick];
\filldraw [black] (5,-.25) circle (2pt);
\draw (5,-.25) node [anchor=east] [black] {$b$};
\draw (6,0) -- (6,-.5)[->] [thick];
\draw (5,0) arc (180:0:.5) [thick];
\draw (6.5,0) node {=};
\draw (7,0) -- (7,-.5) [thick];
\filldraw [black] (8,-.25) circle (2pt);
\draw (8,-.25) node [anchor=west] [black] {$b$};
\draw (8,0) -- (8,-.5) [->][thick];
\draw (7,0) arc (180:0:.5) [thick];
\end{tikzpicture}
\end{equation}

\begin{equation}
\label{H4}
\begin{tikzpicture}[>=stealth]
\draw (0,0) -- (0,.5) [thick];
\filldraw [black] (0,.25) circle (2pt);
\draw (0,.25) node [anchor=east] [black] {$b$};
\draw (1,0) -- (1,.5) [->][thick];
\draw (0,0) arc (180:360:.5) [thick];
\draw (1.5,0) node {=};
\draw (2,0) -- (2,.5)[thick];
\filldraw [black] (3,.25) circle (2pt);
\draw (3,.25) node [anchor=west] [black] {$b$};
\draw (3,0) -- (3,.5) [->][thick];
\draw (2,0) arc (180:360:.5) [thick];

\draw (5,0) -- (5,.5)[->] [thick];
\filldraw [black] (5,.25) circle (2pt);
\draw (5,.25) node [anchor=east] [black] {$b$};
\draw (6,0) -- (6,.5) [thick];
\draw (5,0) arc (180:360:.5) [thick];
\draw (6.5,0) node {=};
\draw (7,0) -- (7,.5)[->] [thick];
\filldraw [black] (8,.25) circle (2pt);
\draw (8,.25) node [anchor=west] [black] {$b$};
\draw (8,0) -- (8,.5) [thick];
\draw (7,0) arc (180:360:.5) [thick];
\end{tikzpicture}.
\end{equation}

\begin{equation}
\label{H5}
\begin{tikzpicture}[>=stealth]
\draw (-1.5,0) -- (-1.5,-2)[<-][thick];
\draw (-1,-1) node{$=$};
\draw (0,-1) node{$(-1)^{|bb'|}$};
\draw (1,0) -- (1,-2)[<-][thick];
\filldraw [black] (-1.5, -1) circle (2pt);
\draw (-2.,-1) node{$b'b$};
\filldraw [black] (1, -.5) circle (2pt);
\filldraw [black] (1, -1.5) circle (2pt);
\draw (1.2,-.5) node{$b$};
\draw (1.2,-1.5) node{$b'$};

\draw (3,0) -- (3,-2)[->][thick];
\draw (3.5,-1) node{$=$};
\draw (4.5,0) -- (4.5,-2)[->][thick];
\filldraw [black] (3, -1) circle (2pt);
\draw (2.5,-1) node{$bb'$};
\filldraw [black] (4.5, -.5) circle (2pt);
\filldraw [black] (4.5, -1.5) circle (2pt);
\draw (4.7,-.5) node{$b$};
\draw (4.7,-1.5) node{$b'$};
\end{tikzpicture}
\end{equation}

\begin{equation}
\label{H6}
\begin{tikzpicture}
\draw (0,0) -- (0,2)[->][thick];
\draw (0,-.5) node{};
\draw (1,-.5) node{};
\draw (1,0) -- (1,2)[->][thick];
\filldraw [black] (0, 1.5) circle (2pt);
\draw (-.5, 1.5) node{$b_1$};
\filldraw [black] (1, .5) circle (2pt);
\draw (1.5, .5) node{$b_2$};
\draw (.5,1) node{$\cdots$};
\draw (2.5,1) node{$=$};
\draw (4,1) node{$(-1)^{|b_1||b_2|}$};
\draw (5,0) -- (5,2)[->][thick];
\draw (6,0) -- (6,2)[->][thick];
\filldraw [black] (5, .5) circle (2pt);
\draw (4.5, .5) node{$b_1$};
\filldraw [black] (6, 1.5) circle (2pt);
\draw (6.5, 1.5) node{$b_2$};
\draw (5.5,1) node{$\cdots$};
\end{tikzpicture}
\end{equation}



\begin{equation}
\label{H7}
\begin{tikzpicture}[>=stealth]
\draw (0,0) -- (2,2)[->][thick];\draw (2,0) -- (0,2)[->][thick];
\draw (1,0) .. controls (0,1) .. (1,2)[->][thick];
\draw (2.5,1) node {=};
\draw (3,0) -- (5,2)[->][thick];
\draw (5,0) -- (3,2)[->][thick];
\draw (4,0) .. controls (5,1) .. (4,2)[->][thick];
\end{tikzpicture}
\end{equation}

\begin{equation}
\label{H8}
\begin{tikzpicture}
\draw (0,0) .. controls (1,1) .. (0,2)[->][thick];
\draw (1,0) .. controls (0,1) .. (1,2)[->][thick];
\draw (1.5,1) node{$=$};
\draw (2.5,0) -- (2.5,2) [->][thick];
\draw (3.5,0) -- (3.5,2) [->][thick];
\end{tikzpicture}
\end{equation}

\begin{equation}
\label{H9}
\begin{tikzpicture}
\draw (0,0) .. controls (1,1) .. (0,2)[->][thick];
\draw (1,0) .. controls (0,1) .. (1,2)[<-][thick];
\draw (1.5,1) node{$=$};
\draw (2.5,0) -- (2.5,2) [->][thick];
\draw (3.5,0) -- (3.5,2) [<-][thick];
\end{tikzpicture}
\end{equation}

\begin{equation}
\label{H10}
\begin{tikzpicture}[>=stealth]
\draw (0,0) .. controls (1,1) .. (0,2)[<-][thick];
\draw (1,0) .. controls (0,1) .. (1,2)[->] [thick];
\draw (1.5,1) node {=};
\draw (2,0) --(2,2)[<-][thick];
\draw (3,0) -- (3,2)[->][thick];
\draw (3.8,1) node{$-\sum_{b \in \mathcal{B}}$};
\draw (4,1.75) arc (180:360:.5) [thick];
\draw (4,2) -- (4,1.75) [thick];
\draw (5,2) -- (5,1.75) [thick][<-];
\draw (5,.25) arc (0:180:.5) [thick];
\filldraw [black] (4.5,1.25) circle (2pt);
\draw (4.5,1.25) node [anchor=south] {$b^\vee$};
\filldraw [black] (4.5,0.75) circle (2pt);
\draw (4.5,.75) node [anchor=north] {$b$};
\draw (5,0) -- (5,.25) [thick];
\draw (4,0) -- (4,.25) [thick][<-];

\draw (5.8,1) node{$+\sum_{b \in \mathcal{B}}$};
\draw (6,1.75) arc (180:360:.5) [thick];
\draw (6,2) -- (6,1.75) [thick];
\draw (7,2) -- (7,1.75) [thick][<-];
\draw (7,.25) arc (0:180:.5) [thick];
\filldraw [black] (6.5,1.25) circle (2pt);
\draw (6.5,1.25) node [anchor=south] {$b^\vee$};
\filldraw [black] (6.5,0.75) circle (2pt);
\draw (6.5,.75) node [anchor=north] {$b$};
\draw (7,0) -- (7,.25) [thick];
\draw (6,0) -- (6,.25) [thick][<-];
\draw [black] (7,.1) circle (2pt);
\draw [black] (6,1.75) circle (2pt);


\end{tikzpicture}
\end{equation}

\begin{equation}
\label{H11}
\begin{tikzpicture}[>=stealth]
\draw [shift={+(0,0)}](0,0) arc (180:360:0.5cm) [thick];
\draw [shift={+(0,0)}][->](1,0) arc (0:180:0.5cm) [thick];
\filldraw [shift={+(1,0)}][black](0,0) circle (2pt);
\draw [shift={+(0,0)}](1,0) node [anchor=east] {$b$};
\draw [shift={+(0,0)}](1.75,0) node{$= \tr(b).$};
\end{tikzpicture}
\end{equation}

\begin{equation}
\label{H12}
\begin{tikzpicture}[>=stealth]
\draw  [shift={+(5,0)}](0,0) .. controls (0,.5) and (.7,.5) .. (.9,0) [thick];
\draw  [shift={+(5,0)}](0,0) .. controls (0,-.5) and (.7,-.5) .. (.9,0) [thick];
\draw  [shift={+(5,0)}](1,-1) .. controls (1,-.5) .. (.9,0) [thick];
\draw  [shift={+(5,0)}](.9,0) .. controls (1,.5) .. (1,1) [->] [thick];
\draw  [shift={+(5,0)}](1.5,0) node {$=$};
\draw  [shift={+(5,0)}](2,0) node {$0.$};
\end{tikzpicture}
\end{equation}

\begin{equation}
\label{H13}
\begin{tikzpicture}[>=stealth]
\draw (0,0) -- (1,1)[->][thick];
\draw (1,0) -- (0,1)[->][thick];
\draw [black] (.25,.25) circle (2pt);
\draw (1.5,.5) node{=};
\draw (2,0) -- (3,1)[->][thick];
\draw (3,0) -- (2,1)[->][thick];
\draw [black] (2.75, .75) circle (2pt);

\end{tikzpicture}
\end{equation}

\begin{equation}
\label{H14}
\begin{tikzpicture}[>=stealth]
\draw (0,0) -- (1,1)[->][thick];
\draw (1,0) -- (0,1)[->][thick];
\draw [black] (.25,.75) circle (2pt);
\draw (1.5,.5) node{=};
\draw (2,0) -- (3,1)[->][thick];
\draw (3,0) -- (2,1)[->][thick];
\draw [black] (2.75, .25) circle (2pt);

\end{tikzpicture}
\end{equation}

\begin{equation}
\label{H15}
\begin{tikzpicture}[>=stealth]
\draw (0,0) -- (0,-.5) [->][thick];
\draw [black] (0,-.25) circle (2pt);
\draw (1,0) -- (1,-.5) [thick];
\draw (0,0) arc (180:0:.5) [thick];
\draw (1.5,0) node {=};
\draw (1.8,0) node{$-$};
\draw (2,0) -- (2,-.5)[->] [thick];
\draw [black] (3,-.25) circle (2pt);
\draw (3,0) -- (3,-.5) [thick];
\draw (2,0) arc (180:0:.5) [thick];

\draw (5,0) -- (5,-.5) [thick];
\draw [black] (5,-.25) circle (2pt);
\draw (6,0) -- (6,-.5)[->] [thick];
\draw (5,0) arc (180:0:.5) [thick];
\draw (6.5,0) node {=};
\draw (7,0) -- (7,-.5) [thick];
\draw [black] (8,-.25) circle (2pt);
\draw (8,0) -- (8,-.5) [->][thick];
\draw (7,0) arc (180:0:.5) [thick];
\end{tikzpicture}
\end{equation}

\begin{equation}
\label{H16}
\begin{tikzpicture}[>=stealth]
\draw (0,0) -- (0,.5) [thick];
\draw [black] (0,.25) circle (2pt);
\draw (1,0) -- (1,.5) [->][thick];
\draw (0,0) arc (180:360:.5) [thick];
\draw (1.5,0) node {=};
\draw (2,0) -- (2,.5)[thick];
\draw [black] (3,.25) circle (2pt);
\draw (3,0) -- (3,.5) [->][thick];
\draw (2,0) arc (180:360:.5) [thick];

\draw (5,0) -- (5,.5)[->] [thick];
\draw [black] (5,.25) circle (2pt);
\draw (6,0) -- (6,.5) [thick];
\draw (5,0) arc (180:360:.5) [thick];
\draw (6.5,0) node {=};
\draw (6.8,0) node{$-$};
\draw (7,0) -- (7,.5)[->] [thick];
\draw [black] (8,.25) circle (2pt);
\draw (8,0) -- (8,.5) [thick];
\draw (7,0) arc (180:360:.5) [thick];
\end{tikzpicture}.
\end{equation}

\begin{equation}
\label{H17}
\begin{tikzpicture}[>=stealth]
\draw (0,0) -- (0,-2)[->][thick];
\draw (.5,-1) node{$=$};
\draw (1,0) -- (1,-2)[->][thick];
\draw [black] (1, -.5) circle (2pt);
\draw [black] (1, -1.5) circle (2pt);

\draw (3,0) -- (3,-2)[<-][thick];
\draw (3.5,-1) node{$=$};
\draw (4,0) -- (4,-2)[<-][thick];
\draw (2.8,-1) node{$-$};
\draw [black] (4, -.5) circle (2pt);
\draw [black] (4, -1.5) circle (2pt);

\end{tikzpicture}
\end{equation}

\begin{equation}
\label{H18}
\begin{tikzpicture}[>=stealth]
\draw (0,0) -- (0,-2)[<-][thick];
\draw (.5,-1) node{$=$};
\draw (1,0) -- (1,-2)[<-][thick];
\draw (-.2,-.5) node{$b$};
\filldraw [black] (0, -.5) circle (2pt);
\draw [black] (0, -1.5) circle (2pt);

\draw [black] (1, -.5) circle (2pt);
\filldraw [black] (1, -1.5) circle (2pt);
\draw (1.2,-1.5) node{$b$};
\end{tikzpicture}
\end{equation}

\begin{equation}
\label{H19}
\begin{tikzpicture}
\draw (0,0) -- (0,2)[->][thick];
\draw (0,-.5) node{};
\draw (1,-.5) node{};
\draw (1,0) -- (1,2)[->][thick];
\draw [black] (0, 1.5) circle (2pt);
\draw [black] (1, .5) circle (2pt);
\draw (.5,1) node{$\cdots$};
\draw (1.5,1) node{$=$};
\draw (2,1) node{$-$};
\draw (2.3,0) -- (2.3,2)[->][thick];
\draw (3.3,0) -- (3.3,2)[->][thick];
\draw [black] (2.3, .5) circle (2pt);
\draw [black] (3.3, 1.5) circle (2pt);
\draw (2.8,1) node{$\cdots$};
\end{tikzpicture}
\end{equation}

\begin{equation}
\label{H20}
\begin{tikzpicture}[>=stealth]
\draw [shift={+(0,0)}](0,0) arc (180:360:0.5cm) [thick];
\draw [shift={+(0,0)}][->](1,0) arc (0:180:0.5cm) [thick];
\filldraw [shift={+(1,0)}][black](0,0) circle (2pt);
\draw [black](.83,.35) circle (2pt);
\draw [shift={+(0,0)}](1,0) node [anchor=east] {$b$};
\draw [shift={+(0,0)}](1.75,0) node{$= 0.$};

\end{tikzpicture}
\end{equation}

There is a $\Z \times \Z_2$-grading on the generating diagrams which make the relations above homogenous.  We set all crossings to have degree $ (0,0)$.  Counterclockwise oriented arcs have degree $ (-1,0) $ while clockwise oriented arcs have degree $ (1,0)$ .  A solid dot labeled by $ b $ has degree $ (|b|,0)$ while a hollow dot has degree $ (0,1) $.  We summarize the degrees in ~\eqref{degrees} where orientation is omitted from the diagram when the bidegree is independent of it.

\begin{equation}
\label{degrees}
\begin{tikzpicture}
\draw (-1,0) -- (13,0)[][very thick];
\draw (0,.5) node{degree};
\draw (0, -.5) node{generator};
\draw (1,1) -- (1,-2)[][very thick];
\draw (1.5,.5) node{$(0,1)$};
\draw (1.5,-.5) -- (1.5,-1.5)[thick];
\draw[black](1.5,-1) circle (2pt);
\draw (2,1) -- (2, -2)[very thick];
\draw (3,-.5) -- (3,-1.5)[thick];
\draw (3.2, -1) node{$ b$};
\draw (4,1) -- (4, -2)[very thick];
\draw (3, .5) node{$(|b|,0)$};
\filldraw[black](3,-1) circle (2pt);
\draw (4.5,-1.5) -- (5.5,-.5)[thick];
\draw (5.5, -1.5) -- (4.5, -.5)[thick];
\draw (5,.5) node{$(0, 0)$};
\draw (-1,1) -- (-1,-2)[very thick];
\draw (6,1) -- (6,-2)[very thick];
\draw (-1,1) -- (13,1)[][very thick];
\draw (-1,-2) -- (13,-2)[][very thick];
\draw (13,1) -- (13,-2)[][very thick];
\draw (11.5,1) -- (11.5,-2)[][very thick];
\draw (9.75,1) -- (9.75,-2)[][very thick];
\draw (8,1) -- (8,-2)[][very thick];
\draw (7,.5) node{$(-1, 0)$};
\draw (8.75,.5) node{$(1, 0)$};
\draw (10.5,.5) node{$(1, 0)$};
\draw (12.25,.5) node{$(-1, 0)$};

\draw (6.5,-.5) arc (180:360:.5)[->] [thick];
\draw (8.25,-.5) arc (180:360:.5)[<-] [thick];
\draw (10,-.95) arc (180:0:.5)[->] [thick];
\draw (11.75,-.95) arc (180:0:.5)[<-] [thick];

\end{tikzpicture}
\end{equation}

If $ P $ is an object concentrated in degree $ (0,0) $, then $ P \langle a \rangle \lbrace b \rbrace $ is concentrated in degree $ (a,b) $.

\subsection{The Karoubi envelope $ \mathcal{H}_{\Gamma}^-$}
Recall the idempotents discussed in $\S$\ref{BGn}. We let $ \mathcal{H}_{\Gamma}^- $ be the Karoubi envelope of $ \mathcal{H}_{\Gamma}'^-$.
In the Karoubi envelope we have direct summands of the objects $ P^n $ and $ Q^n $:
\begin{eqnarray}
{P_i^{(n)}} :=
(P^n, e_{i,1,(n)})
&&
{Q_i^{(n)}} :=
(Q^n, e_{i,1,(n)}).
\end{eqnarray}
One of the main results of this paper is:

\begin{theorem}
\label{relationsinH}
There are isomorphisms of objects in $\mathcal{H}_{\Gamma}^- $:
\begin{enumerate}
\item $ P_i^{(n)} P_j^{(m)} \cong P_j^{(m)} P_i^{(n)} $ for all $ i,j $
\item $ Q_i^{(n)} Q_j^{(m)} \cong Q_j^{(m)} Q_i^{(n)}  $ for all $ i,j $
\item $  Q_i^{(n)} P_i^{(m)} \cong $
\begin{align*}
P_i^{(m)} Q_i^{(n)} \oplus \bigoplus_{k \geq 1} [&\bigoplus_{l=0}^k (P_i^{(m-k)} Q_i^{(n-k)} \langle k-2l \rangle \oplus P_i^{(m-k)} Q_i^{(n-k)} \langle k-2l \rangle \lbrace 1 \rbrace) \oplus\\
&\bigoplus_{l=0}^{k-2} (P_i^{(m-k)} Q_i^{(n-k)} \langle k-2l-2 \rangle \oplus P_i^{(m-k)} Q_i^{(n-k)} \langle k-2l-2 \rangle \lbrace 1 \rbrace)]
\end{align*}
for all $ i $
\item $ Q_i^{(n)} P_j^{(m)} \cong P_j^{(m)} Q_i^{(n)} \oplus \oplus_{k \geq 1} (P_j^{(m-k)} Q_i^{(n-k)} \oplus P_j^{(m-k)} Q_i^{(n-k)} \lbrace 1 \rbrace) $ for all $ i, j $
with $ a_{ij}=-1 $
\item $ Q_i^{(n)} P_j^{(m)} \cong P_j^{(m)} Q_i^{(n)} $ for all $ i, j $ with $ a_{ij} = 0 $.
\end{enumerate}
\end{theorem}

\begin{proof}
This follows from Propositions ~\ref{relationsinK1}, ~\ref{relationsinK2},  ~\ref{relationsinK4}, ~\ref{relationsinK3}, and ~\ref{relationsinK5} given in Section ~\ref{auxcat} utilizing a functor $ \eta $ defined in that section.
\end{proof}


\begin{prop}
There is a homomorphism of algebras
$ \phi \colon \mathfrak{h}_{q,\Gamma}^- \rightarrow K_0(\mathcal{H}_{\Gamma}^-) $.
\end{prop}

\begin{proof}
This follows directly from Theorem ~\ref{relationsinH}.
\end{proof}
We will later prove that this is an isomorphism.






\section{Proof of categorical Heisenberg relations}
\label{auxcat}
\subsection{Auxiliary category $\mathcal{K}_{\Gamma}'^{-} $}
The objects of the $\Z \times \Z_2$-graded monoidal category $\mathcal{K}_{\Gamma}'^{-} $ are direct sums of finite sequences of elements $ \widetilde{P}_i $ or $ \widetilde{Q}_i $ for $ i=1, \ldots, r $ with graded shifts.
The identity object is denoted by $ \bf{1} $.
The space of morphisms between indecomposable objects consists of oriented compact one-manifolds immersed into $ \mathbb{R} \times [0,1] $ compatible with the labeling of the boundary points.  These morphisms are taken up to isotopy fixing the boundary and some local relations which we will describe shortly.
If $ i $ and $ j $ are adjacent in the Dynkin diagram then there is a morphism between $ \widetilde{P}_i $ and $ \widetilde{P}_j $ which we represent by a strand containing a solid dot.  Oriented strands which are bounded by the same element may also carry a solid dot.  Strands may carry a hollow dot.  We summarize these basic morphisms in ~\eqref{auxstrands}.
\begin{equation}
\label{auxstrands}
\begin{tikzpicture}[>=stealth]

\draw (-4,0) -- (-4,1)[->][thick];
\draw (-4,-.4) node{$i$};
\draw (-4,1.4) node{$i$};

\draw (-2,0) -- (-2,1)[->][thick];
\draw (-2,-.4) node{$i$};
\draw (-2,1.4) node{$i$};
\filldraw [black](-2,.5) circle (2pt);

\draw (0,0) -- (0,1)[->][thick];
\draw (0,-.4) node{$i$};
\draw (0,1.4) node{$j$};
\filldraw [black](0,.5) circle (2pt);

\draw (2,0) -- (2,1)[<-][thick];
\draw (2,-.4) node{$i$};
\draw (2,1.4) node{$j$};
\filldraw [black](2,.5) circle (2pt);

\draw (4,0) -- (4,1)[->][thick];
\draw (4,-.4) node{$i$};
\draw (4,1.4) node{$i$};
\draw [black](4,.5) circle (2pt);

\draw (6,0) -- (6,1)[<-][thick];
\draw (6,-.4) node{$i$};
\draw (6,1.4) node{$i$};
\draw [black](6,.5) circle (2pt);
\end{tikzpicture}
\end{equation}

Solid dots move freely along the arcs.  Hollow dots move freely along arcs as well except along clockwise cups or counterclockwise caps as indicated in the relations.
Hollows dots on different strands anti-commute past each other.  Degree one solid dots on different strands anti-commute as well.  All other dots freely move past each other.
Let upward and downward pointing strands with dots labeled by $ i $ on both boundaries be identity morphisms of objects  $ {\widetilde{P}}_i $ and $ {\widetilde{Q}}_i $ respectively.
We will also need the quantities $ \epsilon_{ij} $ for nodes $ i $ and $ j $.  Fix an orientation on the Dynkin diagram.  If $ i $ and $ j $ are not connected by an edge, then set
$ \epsilon_{ij} = 0 $.  If an oriented edge has tail $ i $ and head $ j $, then set $ \epsilon_{ij} = 1 $.  We also set $ \epsilon_{ij} = -\epsilon_{ji} $.

The relations are:

\begin{equation}
\label{K1}
\begin{tikzpicture}[>=stealth]
\draw (-3,0) -- (-3,2)[->][thick];
\draw (-3,-.4) node{$i$};
\draw (-3,2.4) node{$k$};
\draw (-3.2,1) node{$j$};
\filldraw [black](-3,.5) circle (2pt);
\filldraw [black](-3,1.5) circle (2pt);
\draw (-2.5,1) node{$=$};
\draw (-1.8,1) node{$ \delta_{ik} \epsilon_{ij}$};
\draw (-1,0) -- (-1,2)[->][thick];
\filldraw [black](-3,1) circle (2pt);
\draw (-1,-.4) node{$i$};
\draw (-1,2.4) node{$k$};

\draw (1,0) -- (1,2)[<-][thick];
\draw (1,-.4) node{$i$};
\draw (1,2.4) node{$k$};
\draw (.8,1) node{$j$};
\filldraw [black](1,.5) circle (2pt);
\filldraw [black](1,1.5) circle (2pt);
\draw (1.5,1) node{$=$};
\draw (2.2,1) node{$ \delta_{ik} \epsilon_{ij}$};
\draw (3,0) -- (3,2)[<-][thick];
\filldraw [black](3,1) circle (2pt);
\draw (3,-.4) node{$i$};
\draw (3,2.4) node{$k$};

\draw (5,0) -- (5,2)[<-][thick];
\draw (5,-.4) node{$i$};
\draw (5,2.4) node{$i$};
\draw (4.8,1) node{$i$};
\draw [black](5,.5) circle (2pt);
\draw [black](5,1.5) circle (2pt);
\draw (5.5,1) node{$=$};
\draw (6,0) -- (6,2)[<-][thick];
\draw (6,-.4) node{$i$};
\draw (6,2.4) node{$i$};

\draw (8,0) -- (8,2)[->][thick];
\draw (8,-.4) node{$i$};
\draw (8,2.4) node{$i$};
\draw (7.8,1) node{$i$};
\draw [black](8,.5) circle (2pt);
\draw [black](8,1.5) circle (2pt);
\draw (8.5,1) node{$=\;$};
\draw (8.8,1) node{$-$};
\draw (9,0) -- (9,2)[->][thick];
\draw (9,-.4) node{$i$};
\draw (9,2.4) node{$i$};

\end{tikzpicture}
\end{equation}

\begin{equation}
\label{K2}
\begin{tikzpicture}[>=stealth]
\draw (0,0) -- (0,2)[->][thick];
\filldraw [black](0,1) circle (2pt);
\draw [black](0,.5) circle (2pt);
\draw (0,-.2) node{$i$};
\draw (0,2.2) node{$j$};
\draw (1,1) node{$=$};
\draw (2,0) -- (2,2)[->][thick];
\filldraw [black](2,1) circle (2pt);
\draw [black](2,1.5) circle (2pt);
\draw (2,-.2) node{$i$};
\draw (2,2.2) node{$j$};

\draw (5,0) -- (5,2)[<-][thick];
\filldraw [black](5,1) circle (2pt);
\draw [black](5,.5) circle (2pt);
\draw (5,-.2) node{$i$};
\draw (5,2.2) node{$j$};
\draw (6,1) node{$=$};
\draw (7,0) -- (7,2)[<-][thick];
\filldraw [black](7,1) circle (2pt);
\draw [black](7,1.5) circle (2pt);
\draw (7,-.2) node{$i$};
\draw (7,2.2) node{$j$};
\end{tikzpicture}
\end{equation}

\begin{equation}
\label{K3}
\begin{tikzpicture}[>=stealth]
\draw (0,0) -- (2,2)[->][thick];
\draw (2,0) -- (0,2)[->][thick];
\draw (1,0) .. controls (0,1) .. (1,2)[->][thick];
\draw (0,-.4) node{$i$};
\draw (1,-.4) node{$j$};
\draw (2,-.4) node{$k$};
\draw (2.5,1) node {=};
\draw (3,0) -- (5,2)[->][thick];
\draw (5,0) -- (3,2)[->][thick];
\draw (4,0) .. controls (5,1) .. (4,2)[->][thick];
\draw (3,-.4) node{$i$};
\draw (4,-.4) node{$j$};
\draw (5,-.4) node{$k$};
\end{tikzpicture}
\end{equation}

\begin{equation}
\label{K4}
\begin{tikzpicture}
\draw (0,0) .. controls (1,1) .. (0,2)[->][thick];
\draw (1,0) .. controls (0,1) .. (1,2)[->][thick];
\draw (0,-.4) node{$i$};
\draw (1,-.4) node{$j$};
\draw (1.5,1) node{$=$};
\draw (2.5,0) -- (2.5,2) [->][thick];
\draw (3.5,0) -- (3.5,2) [->][thick];
\draw (2.5,-.4) node{$i$};
\draw (3.5,-.4) node{$j$};
\end{tikzpicture}
\end{equation}

\begin{equation}
\label{K5}
\begin{tikzpicture}
\draw (0,0) .. controls (1,1) .. (0,2)[->][thick];
\draw (1,0) .. controls (0,1) .. (1,2)[<-][thick];
\draw (0,-.4) node{$i$};
\draw (1,-.4) node{$j$};
\draw (1.5,1) node{$=$};
\draw (2.5,0) -- (2.5,2) [->][thick];
\draw (3.5,0) -- (3.5,2) [<-][thick];
\draw (2.5,-.4) node{$i$};
\draw (3.5,-.4) node{$j$};
\end{tikzpicture}
\end{equation}

\begin{equation}
\label{K6}
\begin{tikzpicture}[>=stealth]
\draw (0,0) .. controls (1,1) .. (0,2)[<-][thick];
\draw (1,0) .. controls (0,1) .. (1,2)[->] [thick];
\draw (0,-.4) node{$i$};
\draw (1,-.4) node{$j$};
\draw (1.5,1) node {=};
\draw (2,0) --(2,2)[<-][thick];
\draw (3,0) -- (3,2)[->][thick];
\draw (2,-.4) node{$i$};
\draw (3,-.4) node{$j$};

\draw (3.8,1) node{$- \epsilon_{ij}$};

\draw (4,1.75) arc (180:360:.5) [thick];
\draw (4,2) -- (4,1.75) [thick];
\draw (5,2) -- (5,1.75) [thick][<-];
\draw (5,.25) arc (0:180:.5) [thick];
\filldraw [black] (4.5,1.25) circle (2pt);
\filldraw [black] (4.5,0.75) circle (2pt);
\draw (5,0) -- (5,.25) [thick];
\draw (4,0) -- (4,.25) [thick][<-];
\draw (4,-.4) node{$i$};
\draw (5,-.4) node{$j$};
\draw (4,2.4) node{$i$};
\draw (5,2.4) node{$j$};

\draw (5.8,1) node{$+ \epsilon_{ij}$};

\draw (6,1.75) arc (180:360:.5) [thick];
\draw (6,2) -- (6,1.75) [thick];
\draw (7,2) -- (7,1.75) [thick][<-];
\draw (7,.25) arc (0:180:.5) [thick];
\filldraw [black] (6.5,1.25) circle (2pt);
\filldraw [black] (6.5,0.75) circle (2pt);
\draw (7,0) -- (7,.25) [thick];
\draw (6,0) -- (6,.25) [thick][<-];
\draw (6,-.4) node{$i$};
\draw (7,-.4) node{$j$};
\draw (6,2.4) node{$i$};
\draw (7,2.4) node{$j$};
\draw [black] (7,.1) circle (2pt);
\draw [black] (6,1.9) circle (2pt);

\draw (8.5,1) node {for $i \neq j$};

\end{tikzpicture}
\end{equation}

\begin{equation}
\label{K7}
\begin{tikzpicture}[>=stealth]
\draw [shift={+(0,0)}](0,0) arc (180:360:0.5cm) [thick];
\draw [shift={+(0,0)}][->](1,0) arc (0:180:0.5cm) [thick];
\filldraw [shift={+(1,0)}][black](0,0) circle (2pt);
\draw (.5,.8) node{$i$};
\draw (.5,-.8) node{$i$};

\draw [shift={+(0,0)}](1.75,0) node{$= 1.$};

\draw  [shift={+(5,0)}](0,0) .. controls (0,.5) and (.7,.5) .. (.9,0) [thick];
\draw  [shift={+(5,0)}](0,0) .. controls (0,-.5) and (.7,-.5) .. (.9,0) [thick];
\draw  [shift={+(5,0)}](1,-1) .. controls (1,-.5) .. (.9,0) [thick];
\draw  [shift={+(5,0)}](.9,0) .. controls (1,.5) .. (1,1) [->] [thick];
\draw (6,-1.4) node{$i$};

\draw  [shift={+(5,0)}](1.5,0) node {$=$};
\draw  [shift={+(5,0)}](2,0) node {$0.$};
\end{tikzpicture}
\end{equation}

\begin{equation}
\label{K8}
\begin{tikzpicture}[>=stealth]
\draw (0,0) .. controls (1,1) .. (0,2)[<-][thick];
\draw (1,0) .. controls (0,1) .. (1,2)[->] [thick];
\draw (0,-.4) node{$i$};
\draw (1,-.4) node{$i$};
\draw (1.5,1) node {=};
\draw (2,0) --(2,2)[<-][thick];
\draw (3,0) -- (3,2)[->][thick];
\draw (2,-.4) node{$i$};
\draw (3,-.4) node{$i$};

\draw (3.8,1) node{$-$};
\draw (4,1.75) arc (180:360:.5) [thick];
\draw (4,2) -- (4,1.75) [thick];
\draw (5,2) -- (5,1.75) [thick][<-];
\draw (5,.25) arc (0:180:.5) [thick];
\filldraw [black] (4.5,1.25) circle (2pt);
\draw (5,0) -- (5,.25) [thick];
\draw (4,0) -- (4,.25) [thick][<-];
\draw (4,-.4) node{$i$};
\draw (5,-.4) node{$i$};
\draw (4,2.4) node{$i$};
\draw (5,2.4) node{$i$};

\draw (5.8,1) node{$-$};
\draw (6,1.75) arc (180:360:.5) [thick];
\draw (6,2) -- (6,1.75) [thick];
\draw (7,2) -- (7,1.75) [thick][<-];
\draw (7,.25) arc (0:180:.5) [thick];
\filldraw [black] (6.5,0.75) circle (2pt);
\draw (7,0) -- (7,.25) [thick];
\draw (6,0) -- (6,.25) [thick][<-];
\draw (6,-.4) node{$i$};
\draw (7,-.4) node{$i$};
\draw (6,2.4) node{$i$};
\draw (7,2.4) node{$i$};

\draw (7.8,1) node{$+$};
\draw (8,1.75) arc (180:360:.5) [thick];
\draw (8,2) -- (8,1.75) [thick];
\draw (9,2) -- (9,1.75) [thick][<-];
\draw (9,.25) arc (0:180:.5) [thick];
\filldraw [black] (8.5,1.25) circle (2pt);
\draw (9,0) -- (9,.25) [thick];
\draw (8,0) -- (8,.25) [thick][<-];
\draw (8,-.4) node{$i$};
\draw (9,-.4) node{$i$};
\draw (8,2.4) node{$i$};
\draw (9,2.4) node{$i$};
\draw [black] (9,.1) circle (2pt);
\draw [black] (8,1.9) circle (2pt);

\draw (9.8,1) node{$+$};
\draw (10,1.75) arc (180:360:.5) [thick];
\draw (10,2) -- (10,1.75) [thick];
\draw (11,2) -- (11,1.75) [thick][<-];
\draw (11,.25) arc (0:180:.5) [thick];
\filldraw [black] (10.5,.75) circle (2pt);
\draw (11,0) -- (11,.25) [thick];
\draw (10,0) -- (10,.25) [thick][<-];
\draw (10,-.4) node{$i$};
\draw (11,-.4) node{$i$};
\draw (10,2.4) node{$i$};
\draw (11,2.4) node{$i$};
\draw [black] (11,.1) circle (2pt);
\draw [black] (10,1.9) circle (2pt);

\end{tikzpicture}
\end{equation}

\begin{equation}
\label{K9}
\begin{tikzpicture}[>=stealth]
\draw [shift={+(0,0)}](0,0) arc (180:360:0.5cm) [thick];
\draw [shift={+(0,0)}][->](1,0) arc (0:180:0.5cm) [thick];
\filldraw [shift={+(1,0)}][black](0,0) circle (2pt);
\draw [black](.5,.5) circle (2pt);
\draw [shift={+(0,0)}](1.75,0) node{$= 0.$};
\draw (.25,.75) node{$ i$};
\draw (.5,-.75) node{$ i$};

\draw [shift={+(0,0)}](4,0) arc (180:360:0.5cm) [thick];
\draw [shift={+(0,0)}][->](5,0) arc (0:180:0.5cm) [thick];
\draw [black](4.5,.5) circle (2pt);
\draw [shift={+(0,0)}](5.75,0) node{$= 0.$};
\draw (4.5,-.75) node{$ i$};

\draw [shift={+(0,0)}](8,0) arc (180:360:0.5cm) [thick];
\draw [shift={+(0,0)}][->](9,0) arc (0:180:0.5cm) [thick];
\draw [shift={+(0,0)}](9.75,0) node{$= 0.$};
\draw (8.5,-.75) node{$ i$};
\end{tikzpicture}
\end{equation}

\begin{equation}
\label{K13}
\begin{tikzpicture}[>=stealth]
\draw (0,0) -- (1,1)[->][thick];
\draw (1,0) -- (0,1)[->][thick];
\draw [black] (.25,.25) circle (2pt);
\draw (1.5,.5) node{=};
\draw (2,0) -- (3,1)[->][thick];
\draw (3,0) -- (2,1)[->][thick];
\draw [black] (2.75, .75) circle (2pt);

\draw (0,-.25) node{$ i$};
\draw (1,-.25) node{$ i$};
\draw (2,-.25) node{$ i$};
\draw (3,-.25) node{$ i$};

\end{tikzpicture}
\end{equation}

\begin{equation}
\label{K14}
\begin{tikzpicture}[>=stealth]
\draw (0,0) -- (1,1)[->][thick];
\draw (1,0) -- (0,1)[->][thick];
\draw [black] (.25,.75) circle (2pt);
\draw (1.5,.5) node{=};
\draw (2,0) -- (3,1)[->][thick];
\draw (3,0) -- (2,1)[->][thick];
\draw [black] (2.75, .25) circle (2pt);

\draw (0,-.25) node{$ i$};
\draw (1,-.25) node{$ i$};
\draw (2,-.25) node{$ i$};
\draw (3,-.25) node{$ i$};
\end{tikzpicture}
\end{equation}

\begin{equation}
\label{K15}
\begin{tikzpicture}[>=stealth]
\draw (0,0) -- (0,.5) [->][thick];
\draw [black] (0,.25) circle (2pt);
\draw (1,0) -- (1,.5) [thick];
\draw (0,0) arc (180:360:.5) [thick];
\draw (1.5,0) node {=};
\draw (1.8,0) node {$-$};
\draw (2,0) -- (2,.5)[->] [thick];
\draw [black] (3,.25) circle (2pt);
\draw (3,0) -- (3,.5) [thick];
\draw (2,0) arc (180:360:.5) [thick];

\draw (5,0) -- (5,-.5) [thick];
\draw [black] (5,-.25) circle (2pt);
\draw (6,0) -- (6,-.5)[->] [thick];
\draw (5,0) arc (180:0:.5) [thick];
\draw (6.5,0) node {=};
\draw (7,0) -- (7,-.5) [thick];
\draw [black] (8,-.25) circle (2pt);
\draw (8,0) -- (8,-.5) [->][thick];
\draw (7,0) arc (180:0:.5) [thick];
\draw (0,.75) node{$ i$};
\draw (1,.75) node{$ i$};
\draw (2,.75) node{$ i$};
\draw (3,.75) node{$ i$};
\draw (5,-.75) node{$ i$};
\draw (6,-.75) node{$ i$};
\draw (7,-.75) node{$ i$};
\draw (8,-.75) node{$ i$};

\end{tikzpicture}
\end{equation}

\begin{equation}
\label{K16}
\begin{tikzpicture}[>=stealth]
\draw (0,0) -- (0,.5) [thick];
\draw [black] (0,.25) circle (2pt);
\draw (1,0) -- (1,.5) [->][thick];
\draw (0,0) arc (180:360:.5) [thick];
\draw (1.5,0) node {=};
\draw (2,0) -- (2,.5)[thick];
\draw [black] (3,.25) circle (2pt);
\draw (3,0) -- (3,.5) [->][thick];
\draw (2,0) arc (180:360:.5) [thick];

\draw (5,0) -- (5,-.5)[->] [thick];
\draw [black] (5,-.25) circle (2pt);
\draw (6,0) -- (6,-.5) [thick];
\draw (5,0) arc (180:0:.5) [thick];
\draw (6.5,0) node {=};
\draw (6.8,0) node {$-$};
\draw (7,0) -- (7,-.5)[->] [thick];
\draw [black] (8,-.25) circle (2pt);
\draw (8,0) -- (8,-.5) [thick];
\draw (7,0) arc (180:0:.5) [thick];
\draw (0,.75) node{$ i$};
\draw (1,.75) node{$ i$};
\draw (2,.75) node{$ i$};
\draw (3,.75) node{$ i$};
\draw (5,-.75) node{$ i$};
\draw (6,-.75) node{$ i$};
\draw (7,-.75) node{$ i$};
\draw (8,-.75) node{$ i$};
\end{tikzpicture}.
\end{equation}

\begin{equation}
\begin{tikzpicture}
\draw (0,0) -- (0,2)[->][thick];
\draw (0,-.5) node{};
\draw (1,-.5) node{};
\draw (1,0) -- (1,2)[->][thick];
\draw [black] (0, 1.5) circle (2pt);
\draw [black] (1, .5) circle (2pt);
\draw (.5,1) node{$\cdots$};
\draw (1.5,1) node{$=$};
\draw (2,1) node{$-$};
\draw (2.3,0) -- (2.3,2)[->][thick];
\draw (3.3,0) -- (3.3,2)[->][thick];
\draw [black] (2.3, .5) circle (2pt);
\draw [black] (3.3, 1.5) circle (2pt);
\draw (2.8,1) node{$\cdots$};
\draw (0,-.25) node{$ i$};
\draw (1,-.25) node{$ i$};
\draw (2.3,-.25) node{$ i$};
\draw (3.3,-.25) node{$ i$};
\end{tikzpicture}
\end{equation}

There is a $\Z \times \Z_2$-grading on the generating diagrams which make the relations above homogenous.  We set all crossings to have degree $ (0,0)$.  Counterclockwise oriented arcs have degree $ (-1,0) $ while clockwise oriented arcs have degree $ (1,0)$ .  A solid dot between two nodes labeled $ i $ has degree $ (2,0)$ while a solid dot between the two adjacent nodes has degree $ (1,0) $.  Hollow dots have degree $ (0,1) $.  We summarize the degrees in ~\eqref{degrees} where orientation is omitted from the diagram when the bidegree is independent of it.

\begin{equation}
\label{degrees}
\begin{tikzpicture}
\draw (-1,0) -- (13,0)[][very thick];
\draw (0,.5) node{degree};
\draw (0, -.5) node{generator};
\draw (1,1) -- (1,-2)[][very thick];
\draw (1.5,.5) node{$(0,1)$};
\draw (1.5,-.5) -- (1.5,-1.5)[thick];
\draw (1.5,-1.7) node{$i$};
\draw (1.5,-.3) node{$i$};

\draw[black](1.5,-1) circle (2pt);
\draw (2,1) -- (2, -2)[very thick];
\draw (3,-.5) -- (3,-1.5)[thick];
\draw (3,-1.7) node{$i$};
\draw (3,-.3) node{$j$};

\draw (4,1) -- (4, -2)[very thick];
\draw (3, .5) node{$(|a_{ij}|,0)$};
\filldraw[black](3,-1) circle (2pt);
\draw (4.5,-1.5) -- (5.5,-.5)[thick];
\draw (5.5, -1.5) -- (4.5, -.5)[thick];
\draw (4.5,-1.7) node{$i$};
\draw (5.5,-1.7) node{$j$};

\draw (5,.5) node{$(0, 0)$};
\draw (-1,1) -- (-1,-2)[very thick];
\draw (6,1) -- (6,-2)[very thick];
\draw (-1,1) -- (13,1)[][very thick];
\draw (-1,-2) -- (13,-2)[][very thick];
\draw (13,1) -- (13,-2)[][very thick];
\draw (11.5,1) -- (11.5,-2)[][very thick];
\draw (9.75,1) -- (9.75,-2)[][very thick];
\draw (8,1) -- (8,-2)[][very thick];
\draw (7,.5) node{$(-1, 0)$};
\draw (8.75,.5) node{$(1, 0)$};
\draw (10.5,.5) node{$(1, 0)$};
\draw (12.25,.5) node{$(-1, 0)$};

\draw (6.5,-.5) arc (180:360:.5)[->] [thick];
\draw (6.5,-.3) node{$i$};
\draw (7.5,-.3) node{$i$};
\draw (8.25,-.5) arc (180:360:.5)[<-] [thick];
\draw (8.25,-.3) node{$i$};
\draw (9.25,-.3) node{$i$};
\draw (10,-.95) arc (180:0:.5)[->] [thick];
\draw (10,-1.2) node{$i$};
\draw (11,-1.2) node{$i$};
\draw (11.75,-.95) arc (180:0:.5)[<-] [thick];
\draw (11.75,-1.2) node{$i$};
\draw (12.75,-1.2) node{$i$};

\end{tikzpicture}
\end{equation}

The next example is a categorification of a basic twisted Heisenberg relation and follows directly from the relations above.

\begin{ex}
There is an isomorphism of objects:
\begin{equation*}
{\widetilde{Q}_i} {\widetilde{P}_i} \cong {\widetilde{P}_i} {\widetilde{Q}_i} \oplus \Id \langle -1 \rangle \oplus \Id \langle 1 \rangle \oplus \Id \langle -1 \rangle \lbrace 1 \rbrace \oplus \Id \langle 1 \rangle \lbrace 1 \rbrace.
\end{equation*}

Let
\begin{equation*}
f =
\begin{pmatrix}
f_0 \\
f_1 \\
f_2 \\
f_3 \\
f_4 \\
f_5
\end{pmatrix}
\colon
{\widetilde{Q}_i} {\widetilde{P}_i} \rightarrow {\widetilde{P}_i} {\widetilde{Q}_i} \oplus \Id \langle -1 \rangle \oplus \Id \langle 1 \rangle \oplus \Id \langle -1 \rangle \lbrace 1 \rbrace \oplus \Id \langle 1 \rangle \lbrace 1 \rbrace
\end{equation*}
\begin{equation*}
g =
\begin{pmatrix}
g_0 & g_1 & g_2 & g_3 & g_4 & g_5
\end{pmatrix}
\colon {\widetilde{P}_i} {\widetilde{Q}_i} \oplus \Id \langle -1 \rangle \oplus \Id \langle 1 \rangle \oplus \Id \langle -1 \rangle \lbrace 1 \rbrace \oplus \Id \langle 1 \rangle \lbrace 1 \rbrace
\rightarrow {\widetilde{Q}_i} {\widetilde{P}_i}
\end{equation*}
where
\begin{equation*}
\begin{tikzpicture}[>=stealth]
\draw (-.5,.5) node{$f_0=$};
\draw (0,0) --(1,1)[<-][thick];
\draw (1,0) --(0,1)[->][thick];
\draw (2,.5) node{$f_1=$};
\draw (3.5,.25) arc (0:180:0.5cm)[->] [thick];

\draw (4.5,.5) node{$f_2=$};
\draw (6,.25) arc (0:180:0.5cm)[->] [thick];
\filldraw [black] (5.5,.75) circle (2pt);

\draw (7,.5) node{$f_3=$};
\draw (8.5,.25) arc (0:180:0.5cm)[->] [thick];
\draw [black] (8.5,.35) circle (2pt);

\draw (9.5,.5) node{$f_4=$};
\draw (11,.25) arc (0:180:0.5cm)[->] [thick];
\draw [black] (11,.35) circle (2pt);
\filldraw [black] (10.5,.75) circle (2pt);
\end{tikzpicture}
\end{equation*}

\begin{equation*}
\begin{tikzpicture}[>=stealth]
\draw (-.5,.5) node{$g_0=$};
\draw (0,0) --(1,1)[->][thick];
\draw (1,0) --(0,1)[<-][thick];

\draw (2,.5) node{$g_1=$};
\draw (3.5,.75) arc (360:180:0.5cm)[<-] [thick];
\filldraw [black] (3.0,.25) circle (2pt);

\draw (4.5,.5) node{$g_2=$};
\draw (6,.75) arc (360:180:0.5cm)[<-] [thick];

\draw (7,.5) node{$g_3=-$};
\draw (8.7,.75) arc (360:180:0.5cm)[<-] [thick];
\draw [black] (7.7,.65) circle (2pt);
\filldraw [black] (8.2,.25) circle (2pt);

\draw (9.5,.5) node{$g_4=-$};
\draw (11.2,.75) arc (360:180:0.5cm)[<-] [thick];
\draw [black] (10.2,.65) circle (2pt);

\end{tikzpicture}
\end{equation*}
and all strands are labeled by $i$.

Relation \eqref{K8} implies that $ g_0 f_0 + g_1 f_1 + g_2 f_2 + g_3 f_3 + g_4 f_4 $ is the identity map on
$ {\widetilde{Q}_i} {\widetilde{P}_i} $.

Relation \eqref{K5} implies that $ f_0 g_0 $ is the identity map.
The left hand side of relation \eqref{K7} implies that $ f_1 g_1 $ and $ f_2 g_2$  are identity maps.
Relation \eqref{K1} and the left hand side of \eqref{K7} imply that $ f_3 g_3 $ and $ f_4 g_4 $ are identity maps.

For $ k,l=1,2,3,4 $ and $ k \neq l $, $ f_k g_l = 0 $ by relation \eqref{K9}.
For $ l=1,2,3,4$ the equations $ f_0 g_l = 0 $ follow from the right hand side of \eqref{K7}.
This proves the decomposition in the example.

\end{ex}

\subsection{Karoubi envelope $\mathcal{K}_{\Gamma}^{-} $}
Denote the Karoubi envelope of $\mathcal{K}_{\Gamma}'^{-} $ by $\mathcal{K}_{\Gamma}^{-} $.
For a strict partition $ \lambda $ of $ n$, recall that we let $ e_{i, 1, \lambda} $ be a minimal idempotent in $ \text{End}(\widetilde{P}_i^n)$ or $ \text{End}(\widetilde{Q}_i^n)$.
In the Karoubi envelope we have direct summands of the objects $ \widetilde{P}_i^n $ and $ \widetilde{Q}_i^n $:

\begin{eqnarray}
{\widetilde{P}_i^{(n)}} :=
(\widetilde{P}_i^n, e_{i,1,(n)})
&&
{\widetilde{Q}_i^{(n)}} :=
(\widetilde{Q}_i^n, e_{i,1,(n)})
\end{eqnarray}

We now introduce some graphical notation for morphisms in the Karoubi envelope.
Let the identity morphisms of $ \widetilde{P}_i^{(n)} $ and $ \widetilde{Q}_i^{(n)} $ be denoted graphically by the left and right hand sides of ~\eqref{idPkar} respectively.

\begin{equation}
\label{idPkar}
\begin{tikzpicture}[>=stealth]
\draw (0,0) rectangle (1,.5);
\draw (.5,.25) node {$i^n$};
\draw (0,1.5) rectangle (1,2);
\draw (.5,1.75) node {$i^n$};
\draw (0,.5) -- (0,1.5) [->][thick];
\draw (.25,.5) -- (.25,1.5) [->][thick];
\draw (.66,1) node {$\hdots$};
\draw (1,.5) -- (1,1.5) [->][thick];
\draw  (1.5,1) node{$=$};
\draw (2,0) rectangle (3,.5);
\draw (2.5,.25) node {$i^n$};
\draw (2,1.5) rectangle (3,2);
\draw (2.5,1.75) node {$i^n$};
\draw (2.5,.5) -- (2.5,1.5) [->][thick];

\draw (6,0) rectangle (7,.5);
\draw (6.5,.25) node {$i^n$};
\draw (6,1.5) rectangle (7,2);
\draw (6.5,1.75) node {$i^n$};
\draw (6,.5) -- (6,1.5) [<-][thick];
\draw (6.25,.5) -- (6.25,1.5) [<-][thick];
\draw (6.66,1) node {$\hdots$};
\draw (7,.5) -- (7,1.5) [<-][thick];
\draw  (7.5,1) node{$=$};
\draw (8,0) rectangle (9,.5);
\draw (8.5,.25) node {$i^n$};
\draw (8,1.5) rectangle (9,2);
\draw (8.5,1.75) node {$i^n$};
\draw (8.5,.5) -- (8.5,1.5) [<-][thick];
\end{tikzpicture}
\end{equation}

\subsection{The functor $ \eta \colon \mathcal{K}_{\Gamma}^{-} \rightarrow \mathcal{H}_{\Gamma}^{-} $}
We have an analogue of ~\cite[Proposition 3]{CL}.
\begin{prop}
There is a functor $ \eta \colon \mathcal{K}_{\Gamma}^{-} \rightarrow \mathcal{H}_{\Gamma}^{-} $
which maps the object $ \widetilde{P}_i $ to $ (P,e_{i,1,(1)}) $ and the object $ \widetilde{Q}_i $ to $ (Q,e_{i,1,(1)}) $.
The functor $ \eta $ maps morphisms as follows:

\begin{equation*}
\begin{tikzpicture}[>=stealth]
\draw (0,-.2) node {$i$};
\draw (0,1.2) node {$i$};
\draw (0,0) -- (0,1) [->][thick];
\filldraw [black] (0,.5) circle (2pt);

\draw (.3,.5) node {$\leadsto$};

\draw (1.2,.5) node {$\frac{|\Gamma|}{\text{dim}(V)}$};

\draw (1.9,0) -- (1.9,1) [->][thick];
\filldraw [black] (1.9,.5) circle (2pt);
\draw (2.2,.5) node {$\omega$};
\draw (4.5,.5) node {$\colon (P,e_{i,1,(1)}) \rightarrow (P,e_{i,1,(1)}) \langle 2 \rangle $};

\draw (8,-.2) node {$i$};
\draw (8,1.2) node {$i$};
\draw (8,0) -- (8,1) [->][thick];
\draw [black] (8,.5) circle (2pt);

\draw (8.3,.5) node {$\leadsto$};


\draw (8.9,0) -- (8.9,1) [->][thick];
\draw [black] (8.9,.5) circle (2pt);
\draw (11.5,.5) node {$\colon (P,e_{i,1,(1)}) \rightarrow (P,e_{i,1,(1)}) \lbrace 1 \rbrace $};

\end{tikzpicture}
\end{equation*}
\begin{equation*}
\begin{tikzpicture}[>=stealth]
\draw (0,-.2) node {$i$};
\draw (0,1.2) node {$j$};
\draw (0,0) -- (0,1) [->][thick];
\filldraw [black] (0,.5) circle (2pt);

\draw (.3,.5) node {$\leadsto$};

\draw (1.2,.5) node {$\frac{|\Gamma|}{\text{dim}(V)}$};

\draw (1.9,0) -- (1.9,1) [->][thick];
\filldraw [black] (1.9,.5) circle (2pt);
\draw (2.2,.5) node {$v_1$};
\draw (4.6,.5) node {$\colon (P,e_{i,1,(1)}) \rightarrow (P,e_{j,1,(1)}) \langle 1 \rangle $};
\draw (8,.5) node {if $ \epsilon_{ij}=1$};

\end{tikzpicture}
\end{equation*}
\begin{equation*}
\begin{tikzpicture}[>=stealth]
\draw (0,-.2) node {$i$};
\draw (0,1.2) node {$j$};
\draw (0,0) -- (0,1) [->][thick];
\filldraw [black] (0,.5) circle (2pt);

\draw (.3,.5) node {$\leadsto$};

\draw (1.2,.5) node {$\frac{|\Gamma|}{\text{dim}(V)}$};

\draw (1.9,0) -- (1.9,1) [->][thick];
\filldraw [black] (1.9,.5) circle (2pt);
\draw (2.2,.5) node {$v_2$};
\draw (4.6,.5) node {$\colon (P,e_{i,1,(1)}) \rightarrow (P,e_{j,1,(1)}) \langle 1 \rangle$};
\draw (8,.5) node {if $ \epsilon_{ij}=-1$};
\end{tikzpicture}
\end{equation*}
\begin{equation*}
\begin{tikzpicture}[>=stealth]
\draw (0,-.2) node {$i$};
\draw (1,-.2) node {$j$};
\draw (0,1.2) node {$j$};
\draw (1,1.2) node {$i$};
\draw (0,0) -- (1,1) [->][thick];
\draw (1,0) -- (0,1) [->][thick];
\draw (1.3,.5) node {$\leadsto$};

\draw (2,0) -- (3,1) [->][thick];
\draw (3,0) -- (2,1) [->][thick];
\draw (7,.5) node {$\colon (P,e_{i,1,(1)}) (P,e_{j,1,(1)}) \rightarrow (P,e_{j,1,(1)}) (P,e_{i,1,(1)}) $};

\end{tikzpicture}
\end{equation*}
\begin{equation*}
\begin{tikzpicture}[>=stealth]
\draw (0,-.2) node {$i$};
\draw (1,-.2) node {$i$};
\draw (0,0) .. controls (.5,.75) .. (1,0)[<-][thick];

\draw (1.3,.5) node {$\leadsto$};

\draw (2,0) .. controls (2.5,.75) .. (3,0)[<-][thick];

\draw (5.5,.5) node {$\colon (Q,e_{i,1,(1)}) (P,e_{i,1,(1)}) \rightarrow {\bf 1} \langle -1 \rangle$};

\end{tikzpicture}
\end{equation*}
\begin{equation*}
\begin{tikzpicture}[>=stealth]
\draw (0,-.2) node {$i$};
\draw (1,-.2) node {$i$};
\draw (0,0) .. controls (.5,.75) .. (1,0)[->][thick];

\draw (1.3,.5) node {$\leadsto$};

\draw (2,0) .. controls (2.5,.75) .. (3,0)[->][thick];

\draw (5.5,.5) node {$\colon (P,e_{i,1,(1)}) (Q,e_{i,1,(1)}) \rightarrow {\bf 1} \langle 1 \rangle $};

\end{tikzpicture}
\end{equation*}
\begin{equation*}
\begin{tikzpicture}[>=stealth]
\draw (0,1.2) node {$i$};
\draw (1,1.2) node {$i$};
\draw (0,1) .. controls (.5,.25) .. (1,1)[<-][thick];

\draw (1.3,.5) node {$\leadsto$};

\draw (2,1) .. controls (2.5,.25) .. (3,1)[<-][thick];

\draw (5.5,.5) node {$\colon {\bf 1} \rightarrow (P,e_{i,1,(1)}) (Q,e_{i,1,(1)}) \langle 1 \rangle $};

\end{tikzpicture}
\end{equation*}
\begin{equation*}
\begin{tikzpicture}[>=stealth]
\draw (0,1.2) node {$i$};
\draw (1,1.2) node {$i$};
\draw (0,1) .. controls (.5,.25) .. (1,1)[->][thick];

\draw (1.3,.5) node {$\leadsto$};

\draw (2,1) .. controls (2.5,.25) .. (3,1)[->][thick];

\draw (5.5,.5) node {$\colon {\bf 1} \rightarrow (Q,e_{i,1,(1)}) (P,e_{i,1,(1)}) \langle -1 \rangle $.};

\end{tikzpicture}
\end{equation*}
\end{prop}

\begin{proof}
This follows from the proof of ~\cite[Proposition 3]{CL} with trivial modifications.
\end{proof}

As in ~\cite[Proposition 4]{CL} we have the following isomorphism of Grothendieck groups which will be proved later:
\begin{prop}
The functor $ \eta $ induces an isomorphism $ K_0(\eta) \colon K_0(\mathcal{K}_{\Gamma}^{-}) \rightarrow K_0(\mathcal{H}_{\Gamma}^{-}) $.
\end{prop}

\subsection{Isomorphism of objects}
By convention a diagram such as the one in ~\eqref{parcccups} means that there are $k$ parallel arcs.
We view a dot (hollow or solid) on an upper cup as to the upper right of a dot on a lower cup.
\begin{equation}
\label{parcccups}
\begin{tikzpicture}[>=stealth]
\draw (6.5,2) .. controls (7.25,1.25) .. (8,2)[->][thick];
\draw (7.25,1.7) node {$\scriptstyle{k}$};
\draw [black] (6.75,1.75) circle (2pt);
\filldraw [black] (7.25,1.4) circle (2pt);
\end{tikzpicture}
\end{equation}
By convention a diagram such as the one in ~\eqref{parcccaps} means that there are $k$ parallel arcs.
We view a dot (hollow or solid) on a lower cap as to the lower left of a dot on an upper cap.
\begin{equation}
\label{parcccaps}
\begin{tikzpicture}[>=stealth]
\draw (6.5,.5) .. controls (7.25,1.25) .. (8,.5)[<-][thick];
\draw (7.25,.85) node {$\scriptstyle{k}$};
\draw [black] (7.75,.75) circle (2pt);
\filldraw [black] (7.25,1.1) circle (2pt);
\end{tikzpicture}
\end{equation}

\begin{prop}
\label{graphicalidempotentslide}
For any orientations of the strands, there are graphical local relations
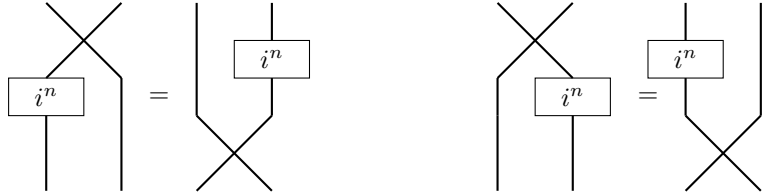
\begin{equation*}
\label{idPkar}
\begin{tikzpicture}[>=stealth]
\draw (0,0) rectangle (1, .5);
\draw (.5,.25) node {$i^n$};
\draw (.5,-1) -- (.5,0)[thick];
\draw (.5,.5) -- (1.5,1.5)[thick];
\draw (.5,1.5) -- (1.5,.5)[thick];
\draw (1.5,.5) -- (1.5,-1)[thick];
\draw (2,.25) node {$=$};

\draw(2.5,-1) -- (3.5,0)[thick];
\draw(3.5,-1) -- (2.5,0)[thick];
\draw(2.5,0) -- (2.5,1.5)[thick];
\draw (3.5,0) -- (3.5,.5)[thick];
\draw (3.5,1) -- (3.5,1.5)[thick];
\draw (3,.5) rectangle (4,1);
\draw (3.5,.75) node {$i^n$};

\draw (7,0) rectangle (8, .5);
\draw (7.5,.25) node {$i^n$};
\draw (6.5,-1) -- (6.5,0)[thick];
\draw (6.5,.5) -- (7.5,1.5)[thick];
\draw (6.5,1.5) -- (7.5,.5)[thick];
\draw (7.5,0) -- (7.5,-1)[thick];
\draw (6.5,.5) -- (6.5,0)[thick];
\draw (8.5,.25) node {$=$};

\draw(9,-1) -- (10,0)[thick];
\draw(10,-1) -- (9,0)[thick];
\draw(9,0) -- (9,.5)[thick];
\draw(9,1) -- (9,1.5)[thick];
\draw (10,0) -- (10,1.5)[thick];
\draw (8.5,.5) rectangle (9.5,1);
\draw (9,.75) node {$i^n$};
\end{tikzpicture}
\end{equation*}
\end{prop}

\begin{proof}
This is simply a graphical depiction of Proposition ~\ref{idempotentslide}.
\end{proof}

\begin{prop}
\label{graphicalidempotentabsorb}
For any orientations of the strands, and with $ m \leq n $, there are graphical local relations
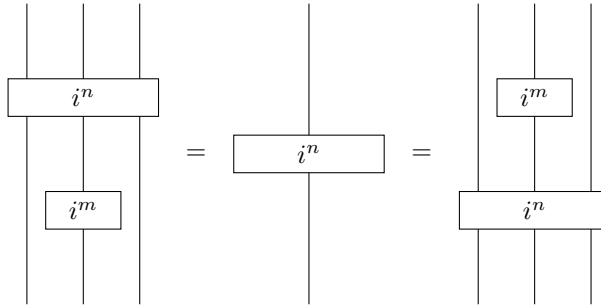
\begin{equation*}
\label{idPkar}
\begin{tikzpicture}[>=stealth]
\draw (0,0) rectangle (2, .5);
\draw (.5,-1.5) rectangle (1.5, -1);
\draw (.25, -2.5) -- (.25, 0);
\draw (1.75, -2.5) -- (1.75, 0);
\draw (.25, .5) -- (.25, 1.5);
\draw (1.75, .5) -- (1.75, 1.5);
\draw (1,-2.5) -- (1,-1.5);
\draw (1,-1) -- (1,0);
\draw (1,1.5) -- (1,.5);
\draw (1,.25) node{$i^n$};
\draw (1,-1.25) node{$i^m$};
\draw (2.5,-.5) node{$=$};

\draw (3,-.75) rectangle (5,-.25);
\draw (4,-.5) node{$i^n$};
\draw (4,-2.5) -- (4,-.75);
\draw (4,1.5) -- (4,-.25);

\draw (5.5, -.5) node{$=$};

\draw (6,-1.5) rectangle (8, -1);
\draw (6.5,0) rectangle (7.5, .5);
\draw (6.25, -2.5) -- (6.25, -1.5);
\draw (6.25, -1) -- (6.25, 1.5);
\draw (7.75, -2.5) -- (7.75, -1.5);
\draw (7.75, -1) -- (7.75, 1.5);

\draw (7,-2.5) -- (7,-1.5);
\draw (7,-1) -- (7,0);
\draw (7,1.5) -- (7,.5);
\draw (7,-1.25) node{$i^n$};
\draw (7,.25) node{$i^m$};

\end{tikzpicture}
\end{equation*}
where each strand is actually a set of parallel copies of strands.
\end{prop}

\begin{proof}
This is just a graphical depiction of Proposition ~\ref{idempotentabsorb}.
\end{proof}

\begin{prop}
\label{relationsinK1}
For all $ i,j $,
there is an isomorphism of objects in $\mathcal{K}_{\Gamma}^- $:
\begin{equation*}
\widetilde{P}_i^{(n)} \widetilde{P}_j^{(m)} \cong \widetilde{P}_j^{(m)} \widetilde{P}_i^{(n)}.
\end{equation*}
\end{prop}

\begin{proof}
Let $ f \colon \widetilde{P}_i^{(n)} \widetilde{P}_j^{(m)} \rightarrow \widetilde{P}_j^{(m)} \widetilde{P}_i^{(n)} $ and
$ g \colon \widetilde{P}_j^{(m)} \widetilde{P}_i^{(n)} \rightarrow \widetilde{P}_i^{(n)} \widetilde{P}_j^{(m)} $ be the maps given by
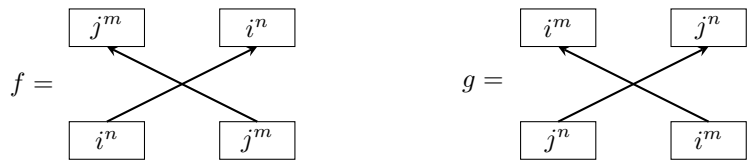
\begin{equation*}
\label{idPkar}
\begin{tikzpicture}[>=stealth]
\draw (-.5,1) node {$f =$};
\draw (0,0) rectangle (1,.5);
\draw (.5,.25) node {$i^n$};
\draw (0,1.5) rectangle (1,2);
\draw (.5,1.75) node {$j^m$};
\draw (.5,.5) -- (2.5,1.5) [->][thick];
\draw (2.5,.5) -- (.5,1.5) [->][thick];
\draw (2,0) rectangle (3,.5);
\draw (2.5,.25) node {$j^m$};
\draw (2,1.5) rectangle (3,2);
\draw (2.5,1.75) node {$i^n$};

\draw (5.5,1) node {$g =$};
\draw (6,0) rectangle (7,.5);
\draw (6.5,.25) node {$j^n$};
\draw (6,1.5) rectangle (7,2);
\draw (6.5,1.75) node {$i^m$};
\draw (6.5,.5) -- (8.5,1.5) [->][thick];
\draw (8.5,.5) -- (6.5,1.5) [->][thick];
\draw (8,0) rectangle (9,.5);
\draw (8.5,.25) node {$i^m$};
\draw (8,1.5) rectangle (9,2);
\draw (8.5,1.75) node {$j^n$};
\end{tikzpicture}
\end{equation*}
where the diagonal arrows are shorthand notation for parallel copies of the arrows where the number of copies is determined by the size of the rectangles in which the arrows begin or terminate.

Then we get for the composition $ gf$:
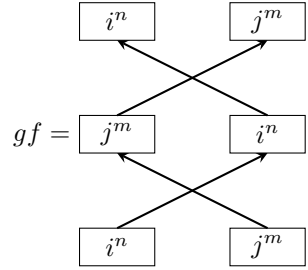
\begin{equation*}
\label{idPkar}
\begin{tikzpicture}[>=stealth]
\draw (-.5,1.75) node {$gf =$};
\draw (0,0) rectangle (1,.5);
\draw (.5,.25) node {$i^n$};
\draw (0,1.5) rectangle (1,2);
\draw (.5,1.75) node {$j^m$};
\draw (.5,.5) -- (2.5,1.5) [->][thick];
\draw (2.5,.5) -- (.5,1.5) [->][thick];
\draw (2,0) rectangle (3,.5);
\draw (2.5,.25) node {$j^m$};
\draw (2,1.5) rectangle (3,2);
\draw (2.5,1.75) node {$i^n$};

\draw (.5,2) -- (2.5,3) [->][thick];
\draw (2.5,2) -- (.5,3) [->][thick];
\draw (0,3) rectangle (1,3.5);
\draw (.5,3.25) node {$i^n$};
\draw (2,3) rectangle (3,3.5);
\draw (2.5,3.25) node {$j^m$};
\end{tikzpicture}
\end{equation*}

By Proposition ~\ref{graphicalidempotentslide}, we may slide the idempotents in the middle of the diagram to the top and absorb them into the top idempotents to get
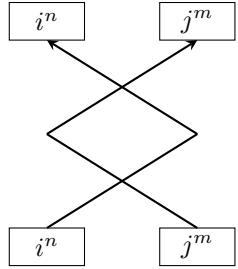
\begin{equation*}
\label{idPkar}
\begin{tikzpicture}[>=stealth]
\draw (0,0) rectangle (1,.5);
\draw (.5,.25) node {$i^n$};
\draw (0,3) rectangle (1,3.5);
\draw (.5,3.25) node {$i^n$};
\draw (.5,.5) -- (2.5,1.75) [][thick];
\draw (2.5,.5) -- (.5,1.75) [][thick];
\draw (2,0) rectangle (3,.5);
\draw (2.5,.25) node {$j^m$};
\draw (2,3) rectangle (3,3.5);
\draw (2.5,3.25) node {$j^m$};
\draw (.5,1.75) -- (2.5,3) [->][thick];
\draw (2.5,1.75) -- (.5,3) [->][thick];
\end{tikzpicture}
\end{equation*}

Now by repeated use of \eqref{K4} we may separate all of the strands to get
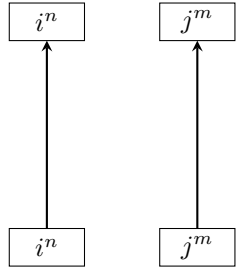
\begin{equation*}
\label{idPkar}
\begin{tikzpicture}[>=stealth]
\draw (0,0) rectangle (1,.5);
\draw (.5,.25) node {$i^n$};
\draw (0,3) rectangle (1,3.5);
\draw (.5,3.25) node {$i^n$};
\draw (.5,.5) -- (.5,3) [->][thick];
\draw (2.5,.5) -- (2.5,3) [->][thick];
\draw (2,0) rectangle (3,.5);
\draw (2.5,.25) node {$j^m$};
\draw (2,3) rectangle (3,3.5);
\draw (2.5,3.25) node {$j^m$};
\end{tikzpicture}
\end{equation*}
which is the identity.  Similarly, $ fg $ is the identity as well.
\end{proof}

\begin{prop}
\label{relationsinK2}
For all $ i,j $,
there is an isomorphism of objects in $\mathcal{K}_{\Gamma}^- $:
\begin{equation*}
\widetilde{Q}_i^{(n)} \widetilde{Q}_j^{(m)} \cong \widetilde{Q}_j^{(m)} \widetilde{Q}_i^{(n)}.
\end{equation*}
\end{prop}

\begin{proof}
The proof of this is nearly identical to that of Proposition ~\ref{relationsinK1}.
\end{proof}

Before we prove the next categorical relation we begin with some prepatory lemmas.

\begin{lemma}
\label{crossinglemma1}

Suppose $ i $ and $ j $ are adjacent nodes and $a+1=n$ and $ b+1=m$. Then there are equalities

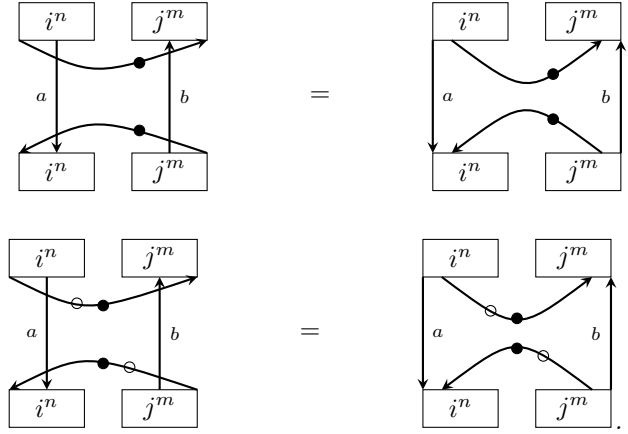
\begin{equation*}
\begin{tikzpicture}[>=stealth]
\draw (0,0) rectangle (1,.5);
\draw (.5,.25) node {$i^{n}$};
\draw (0,2) rectangle (1,2.5);
\draw (.5,2.25) node {$i^{n}$};
\draw (.5,.5) -- (.5,2) [<-][thick];
\draw (1.5,0) rectangle (2.5,.5);
\draw (2,.25) node {$j^{m}$};
\draw (1.5,2) rectangle (2.5,2.5);
\draw (2,2.25) node {$j^{m}$};
\draw (2,.5) -- (2,2) [->][thick];
\draw (0,.5) .. controls (1,1) .. (2.5,.5)[<-][thick];
\draw (0,2) .. controls (1,1.5) .. (2.5,2)[->][thick];
\filldraw [black] (1.6,1.7) circle (2pt);
\filldraw [black] (1.6,.8) circle (2pt);
\draw (.3,1.25) node {$\scriptstyle{a}$};
\draw (2.2,1.25) node {$\scriptstyle{b}$};


\draw [shift={+(-1,0)}](5,1.25) node {$=$};

\draw [shift={+(-2,0)}](7.5,0) rectangle (8.5,.5);
\draw [shift={+(-2,0)}](8,.25) node {$i^{n}$};
\draw [shift={+(-2,0)}](7.5,2) rectangle (8.5,2.5);
\draw [shift={+(-2,0)}](8,2.25) node {$i^{n}$};
\draw [shift={+(-2,0)}](7.5,.5) -- (7.5,2) [<-][thick];
\draw [shift={+(-2,0)}](9,0) rectangle (10,.5);
\draw [shift={+(-2,0)}](9.5,.25) node {$j^{m}$};
\draw [shift={+(-2,0)}](9.5,2.25) node {$j^{m}$};
\draw [shift={+(-2,0)}](9,2) rectangle (10,2.5);
\draw [shift={+(-2,0)}](10,.5) -- (10,2) [->][thick];
\draw [shift={+(-2,0)}](7.75,2) .. controls (8.75,1.25) .. (9.75,2)[->][thick];
\draw [shift={+(-2,0)}](7.75,.5) .. controls (8.75,1.25) .. (9.75,.5)[<-][thick];
\draw [shift={+(-2,0)}](7.7,1.25) node {$\scriptstyle{a}$};
\draw [shift={+(-2,0)}](9.8,1.25) node {$\scriptstyle{b}$};
\filldraw [shift={+(-2,0)}][black] (9.1,1.55) circle (2pt);
\filldraw [shift={+(-2,0)}][black] (9.1,.95) circle (2pt);

\end{tikzpicture}
\end{equation*}

\begin{equation*}
\begin{tikzpicture}[>=stealth]

\draw [shift={+(-3.5,0)}](3.5,0) rectangle (4.5,.5);
\draw [shift={+(-3.5,0)}](4,.25) node {$i^{n}$};
\draw [shift={+(-3.5,0)}](3.5,2) rectangle (4.5,2.5);
\draw [shift={+(-3.5,0)}](4,2.25) node {$i^{n}$};
\draw [shift={+(-3.5,0)}](4,.5) -- (4,2) [<-][thick];
\draw [shift={+(-3.5,0)}](5,0) rectangle (6,.5);
\draw [shift={+(-3.5,0)}](5.5,.25) node {$j^{m}$};
\draw [shift={+(-3.5,0)}](5,2) rectangle (6,2.5);
\draw [shift={+(-3.5,0)}](5.5,2.25) node {$j^{m}$};
\draw [shift={+(-3.5,0)}](5.5,.5) -- (5.5,2) [->][thick];
\draw [shift={+(-3.5,0)}](3.5,.5) .. controls (4.5,1) .. (6,.5)[<-][thick];
\draw [shift={+(-3.5,0)}](3.5,2) .. controls (4.5,1.5) .. (6,2)[->][thick];
\draw [shift={+(-3.5,0)}][black] (4.4,1.65) circle (2pt);
\draw [shift={+(-3.5,0)}][black] (5.1,.8) circle (2pt);
\filldraw [shift={+(-3.5,0)}][black] (4.75,.85) circle (2pt);
\filldraw [shift={+(-3.5,0)}][black] (4.75,1.62) circle (2pt);
\draw [shift={+(-3.5,0)}](3.8,1.25) node {$\scriptstyle{a}$};
\draw [shift={+(-3.5,0)}](5.7,1.25) node {$\scriptstyle{b}$};

\draw (4,1.25) node {$=$};

\draw [shift={+(-6,0)}](11.5,0) rectangle (12.5,.5);
\draw [shift={+(-6,0)}](12,.25) node {$i^{n}$};
\draw [shift={+(-6,0)}](11.5,2) rectangle (12.5,2.5);
\draw [shift={+(-6,0)}](12,2.25) node {$i^{n}$};
\draw [shift={+(-6,0)}](11.5,.5) -- (11.5,2) [<-][thick];
\draw [shift={+(-6,0)}](13,0) rectangle (14,.5);
\draw [shift={+(-6,0)}](13.5,.25) node {$j^{m}$};
\draw [shift={+(-6,0)}](13.5,2.25) node {$j^{m}$};
\draw [shift={+(-6,0)}](13,2) rectangle (14,2.5);
\draw [shift={+(-6,0)}](14,.5) -- (14,2) [->][thick];
\draw [shift={+(-6,0)}](11.75,2) .. controls (12.75,1.25) .. (13.75,2)[->][thick];
\draw [shift={+(-6,0)}](11.75,.5) .. controls (12.75,1.25) .. (13.75,.5)[<-][thick];
\draw [shift={+(-6,0)}](11.7,1.25) node {$\scriptstyle{a}$};
\draw [shift={+(-6,0)}](13.8,1.25) node {$\scriptstyle{b}$};
\draw [black] [shift={+(-6,0)}](12.4,1.55) circle (2pt);
\draw [black] [shift={+(-6,0)}](13.1,.95) circle (2pt);

\filldraw [black] [shift={+(-6,0)}](12.75,1.05) circle (2pt);
\filldraw [black] [shift={+(-6,0)}](12.75,1.45) circle (2pt);
\draw (8.1,0) node {.};

\end{tikzpicture}
\end{equation*}

\end{lemma}

\begin{proof}
Begin from the left and absorb the upper most and bottom most crossings into the idempotents using Proposition ~\ref{absorbformulas}.
Once the crossings formed by the $ a $ crossings are removed, continue from the right and absorb the $ b $ crossings in the same way.
\end{proof}

\begin{lemma}
\label{crossinglemma2}
Suppose $ i $ and $ j $ are adjacent nodes.
Then there is an equality of morphisms

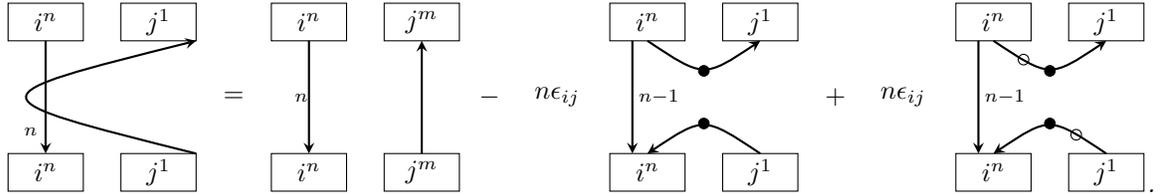
\begin{equation*}
\begin{tikzpicture}[>=stealth]
\draw (0,0) rectangle (1,.5);
\draw (.5,.25) node {$i^{n}$};
\draw (0,2) rectangle (1,2.5);
\draw (.5,2.25) node {$i^{n}$};
\draw (.5,.5) -- (.5,2) [<-][thick];
\draw (1.5,0) rectangle (2.5,.5);
\draw (2,.25) node {$j^{1}$};
\draw (1.5,2) rectangle (2.5,2.5);
\draw (2,2.25) node {$j^{1}$};
\draw (2.5,2) .. controls (-.5,1.25) .. (2.5,.5)[<-][thick];
\draw (.3,.8) node {$\scriptstyle{n}$};

\draw (3,1.25) node {$=$};

\draw (3.5,0) rectangle (4.5,.5);
\draw (4,.25) node {$i^{n}$};
\draw (3.5,2) rectangle (4.5,2.5);
\draw (4,2.25) node {$i^{n}$};
\draw (4,.5) -- (4,2) [<-][thick];
\draw (5,0) rectangle (6,.5);
\draw (5.5,.25) node {$j^{m}$};
\draw (5,2) rectangle (6,2.5);
\draw (5.5,2.25) node {$j^{m}$};
\draw (5.5,.5) -- (5.5,2) [->][thick];
\draw (3.9,1.25) node {$\scriptstyle{n}$};

\draw (6.4,1.25) node {$-$};

\draw (7.3,1.25) node {$n \epsilon_{ij}$};

\draw (8,0) rectangle (9,.5);
\draw (8.5,.25) node {$i^{n}$};
\draw (8,2) rectangle (9,2.5);
\draw (8.5,2.25) node {$i^{n}$};
\draw (8.3,.5) -- (8.3,2) [<-][thick];
\draw (9.5,0) rectangle (10.5,.5);
\draw (10,.25) node {$j^{1}$};
\draw (9.5,2) rectangle (10.5,2.5);
\draw (10,2.25) node {$j^{1}$};
\draw (8.65,1.25) node {$\scriptstyle{n-1}$};
\draw (8.5,.5) .. controls (9.25,1) .. (10,.5)[<-][thick];
\draw (8.5,2) .. controls (9.25,1.5) .. (10,2)[->][thick];
\filldraw [black] (9.25,1.6) circle (2pt);
\filldraw [black] (9.25,.9) circle (2pt);

\draw (11,1.25) node {$+$};

\draw (11.9,1.25) node {$n \epsilon_{ij}$};
\draw (12.6,0) rectangle (13.6,.5);
\draw (13.1,.25) node {$i^{n}$};
\draw (12.6,2) rectangle (13.6,2.5);
\draw (13.1,2.25) node {$i^{n}$};
\draw (12.9,.5) -- (12.9,2) [<-][thick];
\draw (14.1,0) rectangle (15.1,.5);
\draw (14.6,.25) node {$j^{1}$};
\draw (14.1,2) rectangle (15.1,2.5);
\draw (14.6,2.25) node {$j^{1}$};
\draw (13.25,1.25) node {$\scriptstyle{n-1}$};
\draw (13.1,.5) .. controls (13.85,1) .. (14.6,.5)[<-][thick];
\draw (13.1,2) .. controls (13.85,1.5) .. (14.6,2)[->][thick];
\filldraw [black] (13.85,1.6) circle (2pt);
\filldraw [black] (13.85,.9) circle (2pt);

\draw [black] (13.5,1.75) circle (2pt);
\draw [black] (14.2,.75) circle (2pt);
\draw (15.2,0) node {.};

\end{tikzpicture}
\end{equation*}
\end{lemma}

\begin{proof}
The left hand side is equal to
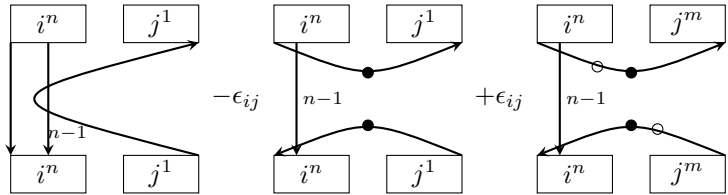
\begin{equation*}
\begin{tikzpicture}[>=stealth]
\draw (0,0) rectangle (1,.5);
\draw (.5,.25) node {$i^{n}$};
\draw (0,2) rectangle (1,2.5);
\draw (.5,2.25) node {$i^{n}$};
\draw (0,.5) -- (0,2) [<-][thick];
\draw (.5,.5) -- (.5,2) [<-][thick];
\draw (1.5,0) rectangle (2.5,.5);
\draw (2,.25) node {$j^{1}$};
\draw (1.5,2) rectangle (2.5,2.5);
\draw (2,2.25) node {$j^{1}$};
\draw (2.5,2) .. controls (-.4,1.25) .. (2.5,.5)[<-][thick];
\draw (.75,.8) node {$\scriptstyle{n-1}$};

\draw (3,1.25) node {$-\epsilon_{ij}$};
\draw (3.5,0) rectangle (4.5,.5);
\draw (4,.25) node {$i^{n}$};
\draw (3.5,2) rectangle (4.5,2.5);
\draw (4,2.25) node {$i^{n}$};
\draw (3.8,.5) -- (3.8,2) [<-][thick];
\draw (5,0) rectangle (6,.5);
\draw (5.5,.25) node {$j^{1}$};
\draw (5,2) rectangle (6,2.5);
\draw (5.5,2.25) node {$j^{1}$};
\draw (3.5,.5) .. controls (4.75,1) .. (6,.5)[<-][thick];
\draw (3.5,2) .. controls (4.75,1.5) .. (6,2)[->][thick];
\draw (4.15,1.25) node {$\scriptstyle{n-1}$};
\filldraw [black] (4.75,1.6) circle (2pt);
\filldraw [black] (4.75,.9) circle (2pt);

\draw (6.5,1.25) node {$+\epsilon_{ij}$};
\draw (7,0) rectangle (8,.5);
\draw (7.5,.25) node {$i^{n}$};
\draw (7,2) rectangle (8,2.5);
\draw (7.5,2.25) node {$i^{n}$};
\draw (7.3,.5) -- (7.3,2) [<-][thick];
\draw (8.5,0) rectangle (9.5,.5);
\draw (9,.25) node {$j^{m}$};
\draw (8.5,2) rectangle (9.5,2.5);
\draw (9,2.25) node {$j^{m}$};
\draw (7,.5) .. controls (8.25,1) .. (9.5,.5)[<-][thick];
\draw (7,2) .. controls (8.25,1.5) .. (9.5,2)[->][thick];
\draw (7.65,1.25) node {$\scriptstyle{n-1}$};
\filldraw [black] (8.25,1.6) circle (2pt);
\filldraw [black] (8.25,.9) circle (2pt);
\draw [black] (7.8,1.68) circle (2pt);
\draw [black] (8.6,.85) circle (2pt);

\end{tikzpicture}
\end{equation*}
by applying the basic double crossing relation to the leftmost of the $ n $ strands.
By Lemma \ref{crossinglemma1} this is equal to

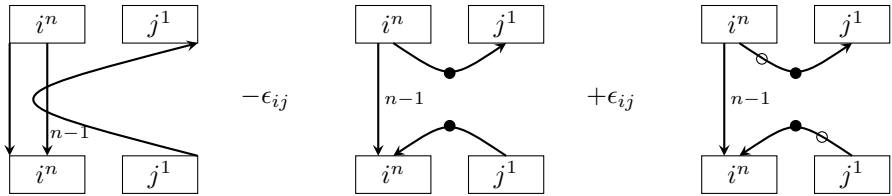
\begin{equation*}
\begin{tikzpicture}[>=stealth]
\draw (0,0) rectangle (1,.5);
\draw (.5,.25) node {$i^{n}$};
\draw (0,2) rectangle (1,2.5);
\draw (.5,2.25) node {$i^{n}$};
\draw (0,.5) -- (0,2) [<-][thick];
\draw (.5,.5) -- (.5,2) [<-][thick];
\draw (1.5,0) rectangle (2.5,.5);
\draw (2,.25) node {$j^{1}$};
\draw (1.5,2) rectangle (2.5,2.5);
\draw (2,2.25) node {$j^{1}$};
\draw (2.5,2) .. controls (-.4,1.25) .. (2.5,.5)[<-][thick];
\draw (.8,.8) node {$\scriptstyle{n-1}$};

\draw (3.4,1.25) node {$-\epsilon_{ij}$};
\draw (4.6,0) rectangle (5.6,.5);
\draw (5.1,.25) node {$i^{n}$};
\draw (4.6,2) rectangle (5.6,2.5);
\draw (5.1,2.25) node {$i^{n}$};
\draw (4.9,.5) -- (4.9,2) [<-][thick];
\draw (6.1,0) rectangle (7.1,.5);
\draw (6.6,.25) node {$j^{1}$};
\draw (6.1,2) rectangle (7.1,2.5);
\draw (6.6,2.25) node {$j^{1}$};
\draw (5.25,1.25) node {$\scriptstyle{n-1}$};
\draw (5.1,.5) .. controls (5.85,1) .. (6.6,.5)[<-][thick];
\draw (5.1,2) .. controls (5.85,1.5) .. (6.6,2)[->][thick];
\filldraw [black] (5.85,1.6) circle (2pt);
\filldraw [black] (5.85,.9) circle (2pt);

\draw (8,1.25) node {$+\epsilon_{ij}$};

\draw (9.2,0) rectangle (10.2,.5);
\draw (9.7,.25) node {$i^{n}$};
\draw (9.2,2) rectangle (10.2,2.5);
\draw (9.7,2.25) node {$i^{n}$};
\draw (9.5,.5) -- (9.5,2) [<-][thick];
\draw (10.7,0) rectangle (11.7,.5);
\draw (11.2,.25) node {$j^{1}$};
\draw (10.7,2) rectangle (11.7,2.5);
\draw (11.2,2.25) node {$j^{1}$};
\draw (9.85,1.25) node {$\scriptstyle{n-1}$};
\draw (9.7,.5) .. controls (10.45,1) .. (11.2,.5)[<-][thick];
\draw (9.7,2) .. controls (10.45,1.5) .. (11.2,2)[->][thick];
\filldraw [black] (10.45,1.6) circle (2pt);
\filldraw [black] (10.45,.9) circle (2pt);

\draw [black] (10,1.79) circle (2pt);
\draw [black] (10.8,.75) circle (2pt);
\draw (11.8,0) node {.};

\end{tikzpicture}
\end{equation*}
We then get the identity by induction on $ n $.

\end{proof}

\begin{lemma}
\label{crossinglemma4}
Suppose $ i $ and $ j $ are adjacent nodes,
Then there is an equality of morphisms

\begin{equation*}
\begin{tikzpicture}[>=stealth]
\draw (0,0) rectangle (1,.5);
\draw (.5,.25) node {$i^{n}$};
\draw (0,2) rectangle (1,2.5);
\draw (.5,2.25) node {$i^{n}$};
\draw (1.5,0) rectangle (2.5,.5);
\draw (2,.25) node {$j^{m}$};
\draw (1.5,2) rectangle (2.5,2.5);
\draw (2,2.25) node {$j^{m}$};
\draw (2.5,2) .. controls (-.5,1.25) .. (2.5,.5)[<-][thick];
\draw (0,2) .. controls (3,1.25) .. (0,.5)[->][thick];

\draw (2,1.25) node {$\scriptstyle{n}$};
\draw (.5,1.25) node {$\scriptstyle{m}$};


\draw (3,1.25) node {$=$};

\draw (3.5,0) rectangle (4.5,.5);
\draw (4,.25) node {$i^{n}$};
\draw (3.5,2) rectangle (4.5,2.5);
\draw (4,2.25) node {$i^{n}$};
\draw (5,0) rectangle (6,.5);
\draw (5.5,.25) node {$j^{m}$};
\draw (5,2) rectangle (6,2.5);
\draw (5.5,2.25) node {$j^{m}$};
\draw (4,.5) .. controls (5.5,1.25) .. (4,2)[<-][thick];
\draw (5.5,.5) .. controls (4,1.25) .. (5.5,2)[->][thick];
\draw (5.75,.5) -- (5.75,2) [->][thick];
\draw (5,1.25) node {$\scriptstyle{n}$};
\draw (4.05,1.25) node {$\scriptstyle{m-1}$};


\draw (7.1,1.25) node {$-$};
\draw (7.7,1.25) node {$n\epsilon_{ij}$};

\draw (8,0) rectangle (9,.5);
\draw (8.5,.25) node {$i^{n}$};
\draw (8,2) rectangle (9,2.5);
\draw (8.5,2.25) node {$i^{n}$};
\draw (8.25,.5) .. controls (9.8,1.25) .. (8.25,2)[<-][thick];
\draw (10.25,.5) .. controls (8.7,1.25) .. (10.25,2)[->][thick];

\draw (9.5,0) rectangle (10.5,.5);
\draw (10,.25) node {$j^{m}$};
\draw (9.5,2) rectangle (10.5,2.5);
\draw (10,2.25) node {$j^{m}$};
\draw (8.75,1.35) node {$\scriptstyle{m-1}$};
\draw (9.8,1.15) node {$\scriptstyle{n-1}$};

\draw (8.5,.5) .. controls (9.25,1) .. (10,.5)[<-][thick];
\draw (8.5,2) .. controls (9.25,1.5) .. (10,2)[->][thick];
\filldraw [black] (9.25,1.6) circle (2pt);
\filldraw [black] (9.25,.9) circle (2pt);


\draw (11.6,1.25) node {$+$};
\draw (12.1,1.25) node {$n\epsilon_{ij}$};

\draw (12.6,0) rectangle (13.6,.5);
\draw (13.1,.25) node {$i^{n}$};
\draw (12.6,2) rectangle (13.6,2.5);
\draw (13.1,2.25) node {$i^{n}$};
\draw (14.1,0) rectangle (15.1,.5);
\draw (14.6,.25) node {$j^{m}$};
\draw (14.1,2) rectangle (15.1,2.5);
\draw (14.6,2.25) node {$j^{m}$};
\draw (13.1,.5) .. controls (13.85,1) .. (14.6,.5)[<-][thick];
\draw (13.1,2) .. controls (13.85,1.5) .. (14.6,2)[->][thick];
\filldraw [black] (13.85,1.6) circle (2pt);
\filldraw [black] (13.85,.9) circle (2pt);

\draw [black] (13.5,1.75) circle (2pt);
\draw [black] (14.2,.75) circle (2pt);
\draw (15.2,0) node {.};

\draw (12.85,.5) .. controls (14.4,1.25) .. (12.85,2)[<-][thick];
\draw (14.85,.5) .. controls (13.3,1.25) .. (14.85,2)[->][thick];

\draw (13.35,1.15) node {$\scriptstyle{m-1}$};
\draw (14.4,1.15) node {$\scriptstyle{n-1}$};

\end{tikzpicture}
\end{equation*}
\end{lemma}

\begin{proof}
Apply the basic crossing relation to the innermost double crossing of the $ n $ and $ m $ strands to get that the left hand side of the lemma is equal to

\begin{equation}
\label{crossinglemma4eq1}
\begin{tikzpicture}[>=stealth]
\draw (3.5,0) rectangle (4.5,.5);
\draw (4,.25) node {$i^{n}$};
\draw (3.5,2) rectangle (4.5,2.5);
\draw (4,2.25) node {$i^{n}$};
\draw (5,0) rectangle (6,.5);
\draw (5.5,.25) node {$j^{m}$};
\draw (5,2) rectangle (6,2.5);
\draw (5.5,2.25) node {$j^{m}$};
\draw (4,.5) .. controls (5.5,1.25) .. (4,2)[<-][thick];
\draw (5.5,.5) .. controls (4,1.25) .. (5.5,2)[->][thick];

\draw (3.5,.5) .. controls (5,1.25) .. (3.5,2)[<-][thick];
\draw (6,.5) .. controls (4.4,1.25) .. (6,2)[->][thick];


\draw (5.45,1.25) node {$\scriptstyle{n-1}$};
\draw (4.0,1.25) node {$\scriptstyle{m-1}$};

\draw (7,1.25) node {$-\epsilon_{ij}$};

\draw (8,0) rectangle (9,.5);
\draw (8.5,.25) node {$i^{n}$};
\draw (8,2) rectangle (9,2.5);
\draw (8.5,2.25) node {$i^{n}$};
\draw (8.25,.5) .. controls (9.8,1.25) .. (8.25,2)[<-][thick];
\draw (10.25,.5) .. controls (8.7,1.25) .. (10.25,2)[->][thick];

\draw (9.5,0) rectangle (10.5,.5);
\draw (10,.25) node {$j^{m}$};
\draw (9.5,2) rectangle (10.5,2.5);
\draw (10,2.25) node {$j^{m}$};
\draw (8.75,1.15) node {$\scriptstyle{m-1}$};
\draw (9.8,1.15) node {$\scriptstyle{n-1}$};

\draw (8,.5) .. controls (9.25,1) .. (10.5,.5)[<-][thick];
\draw (8,2) .. controls (9.25,1.5) .. (10.5,2)[->][thick];
\filldraw [black] (9.25,1.6) circle (2pt);
\filldraw [black] (9.25,.9) circle (2pt);

\draw (11.45,1.25) node {$+\epsilon_{ij}$};

\draw (12.6,0) rectangle (13.6,.5);
\draw (13.1,.25) node {$i^{n}$};
\draw (12.6,2) rectangle (13.6,2.5);
\draw (13.1,2.25) node {$i^{n}$};
\draw (14.1,0) rectangle (15.1,.5);
\draw (14.6,.25) node {$j^{m}$};
\draw (14.1,2) rectangle (15.1,2.5);
\draw (14.6,2.25) node {$j^{m}$};
\draw (12.6,.5) .. controls (13.85,1) .. (15.1,.5)[<-][thick];
\draw (12.6,2) .. controls (13.85,1.5) .. (15.1,2)[->][thick];
\filldraw [black] (13.85,1.6) circle (2pt);
\filldraw [black] (13.85,.9) circle (2pt);

\draw [black] (13.1,1.8) circle (2pt);
\draw [black] (14.6,.7) circle (2pt);
\draw (15.2,0) node {.};

\draw (12.85,.5) .. controls (14.4,1.25) .. (12.85,2)[<-][thick];
\draw (14.85,.5) .. controls (13.3,1.25) .. (14.85,2)[->][thick];

\draw (13.35,1.15) node {$\scriptstyle{m-1}$};
\draw (14.4,1.15) node {$\scriptstyle{n-1}$};

\end{tikzpicture}
\end{equation}
Now apply Lemma \ref{crossinglemma2} to the first term above and Lemma \ref{crossinglemma1} to the next two terms above to get that \eqref{crossinglemma4eq1} is equal to the right hand side of the lemma.

\end{proof}


\begin{lemma}
\label{crossinglemma5}
Let $ i $ and $ j $ be adjacent nodes, $ a+b= n $, and $ d+e=m $.  Then there is an equality
\begin{equation*}
\begin{tikzpicture}[>=stealth]
\draw (0,0) rectangle (1,.5);
\draw (.5,.25) node {$i^{n}$};
\draw (0,2) rectangle (1,2.5);
\draw (.5,2.25) node {$i^{n}$};
\draw (0,.5) -- (0,2) [<-][thick];
\draw (2.5,.5) -- (2.5,2) [->][thick];
\draw (1.5,0) rectangle (2.5,.5);
\draw (2,.25) node {$j^{m}$};
\draw (1.5,2) rectangle (2.5,2.5);
\draw (2,2.25) node {$j^{m}$};
\draw (2,2) .. controls (.5,1.25) .. (2,.5)[<-][thick];
\draw (.5,2) .. controls (2,1.25) .. (.5,.5)[->][thick];

\draw (1.8,1.25) node {$\scriptstyle{b}$};
\draw (.7,1.25) node {$\scriptstyle{d}$};

\draw (.1,1.25) node {$\scriptstyle{a}$};
\draw (2.4,1.25) node {$\scriptstyle{e}$};

\draw (3,1.25) node {$=$};

\draw (4.5,1.25) node {$\sum_{k=0}^{\text{min}(b,d)} \alpha_{b,d,k}$};

\draw (6,0) rectangle (7,.5);
\draw (6.5,.25) node {$i^{n}$};
\draw (6,2) rectangle (7,2.5);
\draw (6.5,2.25) node {$i^{n}$};
\draw (6,.5) -- (6,2) [<-][thick];
\draw (8.5,.5) -- (8.5,2) [->][thick];
\draw (7.5,0) rectangle (8.5,.5);
\draw (8,.25) node {$j^{m}$};
\draw (7.5,2) rectangle (8.5,2.5);
\draw (8,2.25) node {$j^{m}$};
\draw (6.5,.5) .. controls (7.25,1.25) .. (8,.5)[<-][thick];
\draw (6.5,2) .. controls (7.25,1.25) .. (8,2)[->][thick];

\filldraw [black] (7.25,1.4) circle (2pt);
\filldraw [black] (7.25,1.1) circle (2pt);
\draw [black] (7.75,.75) circle (2pt);
\draw [black] (6.75,1.75) circle (2pt);

\draw (7.25,1.7) node {$\scriptstyle{k}$};
\draw (7.25,.85) node {$\scriptstyle{k}$};
\draw (5.4,1.75) node {$\scriptstyle{a+b-k}$};
\draw (8,1.25) node {$\scriptstyle{d+e-k}$};

\draw (9.3,1.25) node {$+$};
\draw (11,1.25) node {$\sum_{k=0}^{\text{min}(b,d)} \beta_{b,d,k}$};

\draw (12.5,0) rectangle (13.5,.5);
\draw (13,.25) node {$i^{n}$};
\draw (12.5,2) rectangle (13.5,2.5);
\draw (13,2.25) node {$i^{n}$};
\draw (12.5,.5) -- (12.5,2) [<-][thick];
\draw (15,.5) -- (15,2) [->][thick];
\draw (14,0) rectangle (15,.5);
\draw (14.5,.25) node {$j^{m}$};
\draw (14,2) rectangle (15,2.5);
\draw (14.5,2.25) node {$j^{m}$};
\draw (13.5,.5) .. controls (13.75,.75) .. (14,.5)[<-][thick];
\draw (13.5,2) .. controls (13.75,1.75) .. (14,2)[->][thick];
\draw (13,.5) .. controls (13.75,1.25) .. (14.5,.5)[<-][thick];
\draw (13,2) .. controls (13.75,1.25) .. (14.5,2)[->][thick];

\filldraw [black] (13.75,1.8) circle (2pt);
\filldraw [black] (13.75,.7) circle (2pt);
\filldraw [black] (13.75,1.4) circle (2pt);
\filldraw [black] (13.75,1.1) circle (2pt);

\draw [black] (13.2,1.8) circle (2pt);
\draw [black] (14.3,.7) circle (2pt);

\draw (13.1,1.55) node {$\scriptstyle{k-1}$};
\draw (13.1,.95) node {$\scriptstyle{k-1}$};
\draw (11.9,1.75) node {$\scriptstyle{a+b-k}$};
\draw (14.4,1.25) node {$\scriptstyle{d+e-k}$};

\end{tikzpicture}
\end{equation*}
where $ \alpha_{b,d,k} = (\epsilon_{ij})^k k! \binom{b}{k} \binom{d}{k}, \beta_{b,d,k} = -(\epsilon_{ij})^k k \binom{b}{k} \binom{d}{k} k!  $.
\end{lemma}

\begin{proof}
This follows by induction using Lemma \ref{crossinglemma4} and noting that $k$ cups or caps must contain a hollow dot on at least $k-1$ strands or else the diagram is zero.
\end{proof}

\begin{lemma}
\label{crossinglemma6}
Let $i$ and $j$ be adjacent nodes.
There are equalities of morphisms
\begin{equation}
\label{crossinglemma6eq1}
\begin{tikzpicture}[>=stealth]
\draw (0,0) rectangle (1,.5);
\draw (.5,.25) node {$j^{s-k}$};
\draw (0,1.5) rectangle (1,2);
\draw (.5,1.75) node {$i^{r}$};
\draw (0,.5) -- (3,1.5) [->][thick];
\draw (3,.5) -- (0,1.5) [<-][thick];
\draw (2,0) rectangle (3,.5);
\draw (2.5,.25) node {$i^{r-k}$};
\draw (2,1.5) rectangle (3,2);
\draw (2.5,1.75) node {$j^s$};

\draw (0,2) -- (3,3) [<-][thick];
\draw (3,2) -- (0,3) [->][thick];
\draw (0,3) rectangle (1,3.5);
\draw (.5,3.25) node {$j^{s-k}$};
\draw (2,3) rectangle (3,3.5);
\draw (2.5,3.25) node {$i^{r-k}$};

\draw (.5,1.5) .. controls (1.5,1) .. (2.5,1.5)[->][thick];
\draw (.5,2) .. controls (1.5,2.5) .. (2.5,2)[<-][thick];
\filldraw [black] (1.5,1.15) circle (2pt);
\filldraw [black] (1.5,2.35) circle (2pt);
\draw (1.65,1.3) node {$\scriptstyle{k}$};
\draw (1.65,2.25) node {$\scriptstyle{k}$};

\draw [black] (2.25,2.1) circle (2pt);
\draw [black] (.8,1.35) circle (2pt);

\draw (3.5,1.75) node {$=$};

\draw (4.5,1.75) node {$\frac{(-1)^k(\epsilon_{ij})^k}{k! \binom{r}{k} \binom{s}{k}}$};

\draw (5,0) rectangle (6,.5);
\draw (5.5,.25) node {$j^{s-k}$};
\draw (7,0) rectangle (8,.5);
\draw (7.5,.25) node {$i^{r-k}$};

\draw (5.5,.5) -- (5.5,3) [->][thick];
\draw (7.5,.5) -- (7.5,3) [<-][thick];

\draw (5,3) rectangle (6,3.5);
\draw (5.5,3.25) node {$j^{s-k}$};
\draw (7,3) rectangle (8,3.5);
\draw (7.5,3.25) node {$i^{r-k}$};

\end{tikzpicture}
\end{equation}

\begin{equation}
\label{crossinglemma6eq2}
\begin{tikzpicture}[>=stealth]
\draw (0,0) rectangle (1,.5);
\draw (.5,.25) node {$j^{s-k}$};
\draw (0,1.5) rectangle (1,2);
\draw (.5,1.75) node {$i^{r}$};
\draw (1,.5) -- (3,1.5) [->][thick];
\draw (2,.5) -- (0,1.5) [<-][thick];
\draw (2,0) rectangle (3,.5);
\draw (2.5,.25) node {$i^{r-k}$};
\draw (2,1.5) rectangle (3,2);
\draw (2.5,1.75) node {$j^s$};

\draw (0,2) -- (2,3) [<-][thick];
\draw (3,2) -- (1,3) [->][thick];
\draw (0,3) rectangle (1,3.5);
\draw (.5,3.25) node {$j^{s-k}$};
\draw (2,3) rectangle (3,3.5);
\draw (2.5,3.25) node {$i^{r-k}$};

\draw (1,1.5) .. controls (1.5,1.25) .. (2,1.5)[->][thick];
\draw (.5,2) .. controls (1.5,2.7) .. (2.5,2)[<-][thick];

\draw (.5,1.5) .. controls (1.5,.8) .. (2.5,1.5)[->][thick];
\draw (1,2) .. controls (1.5,2.25) .. (2,2)[<-][thick];

\filldraw [black] (1.5,1.3) circle (2pt);
\filldraw [black] (1.5,2.5) circle (2pt);

\filldraw [black] (1.5,.95) circle (2pt);
\filldraw [black] (1.5,2.2) circle (2pt);

\draw (1.5,1.12) node {$\scriptstyle{k-1}$};
\draw (1.5,2.35) node {$\scriptstyle{k-1}$};

\draw [black] (1.1,1.1) circle (2pt);
\draw [black] (2.29,2.12) circle (2pt);

\draw (3.4,1.75) node {$=$};

\draw (4.5,1.75) node {$\frac{(-1)^{k-1}(\epsilon_{ij})^k}{k! \binom{r}{k} \binom{s}{k}k}$};

\draw (5,0) rectangle (6,.5);
\draw (5.5,.25) node {$j^{s-k}$};
\draw (7,0) rectangle (8,.5);
\draw (7.5,.25) node {$i^{r-k}$};

\draw (5.5,.5) -- (5.5,3) [->][thick];
\draw (7.5,.5) -- (7.5,3) [<-][thick];

\draw (5,3) rectangle (6,3.5);
\draw (5.5,3.25) node {$j^{s-k}$};
\draw (7,3) rectangle (8,3.5);
\draw (7.5,3.25) node {$i^{r-k}$};
\draw (8.1,0) node {.};

\end{tikzpicture}
\end{equation}

\end{lemma}

\begin{proof}
For the equality in \eqref{crossinglemma6eq1},
first expand the idempotents labeled $i^r$ and $j^s$ into a sum of elements from the symmetric groups $S_r$ and $S_s$ respectively.
The expanded diagram will contain a left hand curl if the term from the idempotent $i^r$ doesn't come from the subgroup $S_{r-k}\times S_k$ and if the term from the idempotent $j^s$ doesn't come from the subgroup $S_k \times S_{s-k}$.
Furthermore for each element in $S_{r-k} \times S_k$ there exists exactly $(s-k)!$ terms in the expansion of $j^s$ which will not produce a left hand curl.

After absorbing crossings into the top and bottom idempotents, the expansion results in $k$ circles each labeled $i$ with a degree two solid dot times the diagram
\begin{equation}
\label{crossinglemma6eq3}
\begin{tikzpicture}[>=stealth]
\draw (-2,3.75) node{$\frac{(-1)^k(\epsilon_{ij})^k}{k! \binom{r}{k} \binom{s}{k}}$};

\draw (0,1) rectangle (1,1.5);
\draw (.5,1.25) node {$j^{s-k}$};
\draw (0,3.5) rectangle (1,4);
\draw (.5,3.75) node {$i^{r-k}$};
\draw (5,1) rectangle (6,1.5);
\draw (5.5,1.25) node {$i^{r-k}$};
\draw (5,3.5) rectangle (6,4);
\draw (5.5,3.75) node {$j^{s-k}$};
\draw (0,6) rectangle (1,6.5);
\draw (.5,6.25) node {$j^{s-k}$};
\draw (5,6) rectangle (6,6.5);
\draw (5.5,6.25) node {$i^{r-k}$};



\draw (.5,1.5) -- (5.5,3.5) [->][thick];
\draw (5.5,1.5) -- (.5,3.5) [<-][thick];

\draw (.5,4) -- (5.5,6) [<-][thick];
\draw (5.5,4) -- (.5,6) [->][thick];
\draw (6.1,1) node{.};
\end{tikzpicture}
\end{equation}
Since each circle is equal to one, \eqref{crossinglemma6eq1} follows.

The identity \eqref{crossinglemma6eq2} is similar except that now left hand curls will be avoided only if the terms in the expansion of the idempotents labeled $i^r$ and $j^s$ come from the subgroups $S_{r-k} \times S_{k-1} \times S_1$
and $S_1 \times S_{k-1} \times S_{s-k}$ respectively.

\end{proof}

\begin{prop}
\label{relationsinK4}
For each pair of nodes $ i,j$ with $ a_{ij}=-1 $,
there is an isomorphism of objects in $\mathcal{K}_{\Gamma}^- $:
\begin{equation*}
\widetilde{Q}_i^{(n)} \widetilde{P}_j^{(m)} \cong  \widetilde{P}_j^{(m)}  \widetilde{Q}_i^{(n)} \oplus \bigoplus_{k \geq 1} (\widetilde{P}_j^{(m-k)} \widetilde{Q}_i^{(n-k)} \oplus \widetilde{P}_j^{(m-k)} \widetilde{Q}_i^{(n-k)} \lbrace 1 \rbrace).
\end{equation*}
\end{prop}

\begin{proof}
Set $ r = \text{min}(m,n) $.  Then let
\begin{equation*}
f =
\begin{pmatrix}
f_0 \\
f_1 \\
f_1'\\
\vdots \\
f_r\\
f_r'
\end{pmatrix}
\colon
\widetilde{Q}_i^{(n)} \widetilde{P}_j^{(m)} \rightarrow  \widetilde{P}_j^{(m)}  \widetilde{Q}_i^{(n)} \oplus \bigoplus_{k \geq 1} (\widetilde{P}_j^{(m-k)} \widetilde{Q}_i^{(n-k)} \lbrace k \rbrace \oplus \widetilde{P}_j^{(m-k)} \widetilde{Q}_i^{(n-k)} \lbrace k-1 \rbrace)
\end{equation*}
where
\begin{equation*}
f_0 \colon \widetilde{Q}_i^{(n)}  \widetilde{P}_j^{(m)}   \rightarrow \widetilde{P}_j^{(m)}  \widetilde{Q}_i^{(n)}
\end{equation*}
and for $ k>0 $,
\begin{equation*}
f_k \colon \widetilde{Q}_i^{(n)}  \widetilde{P}_j^{(m)}   \rightarrow \widetilde{P}_j^{(m-k)}  \widetilde{Q}_i^{(n-k)} \lbrace k \rbrace
\end{equation*}
\begin{equation*}
f_k' \colon \widetilde{Q}_i^{(n)}  \widetilde{P}_j^{(m)}   \rightarrow \widetilde{P}_j^{(m-k)}  \widetilde{Q}_i^{(n-k)} \lbrace k-1 \rbrace
\end{equation*}
where these maps are defined by the diagrams

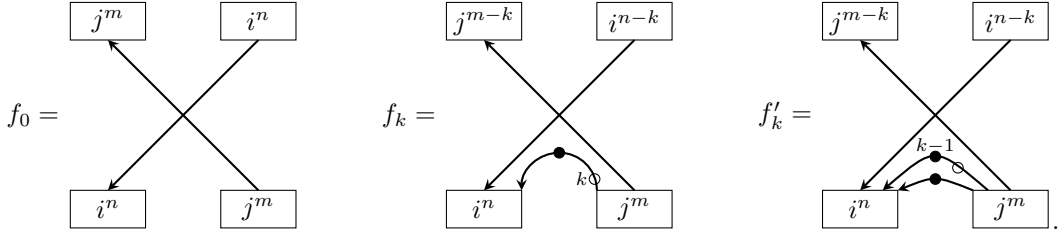
\begin{equation*}
\label{idPkar}
\begin{tikzpicture}[>=stealth]
\draw (-.5,1.5) node {$f_0 =$};
\draw (0,0) rectangle (1,.5);
\draw (.5,.25) node {$i^n$};
\draw (0,2.5) rectangle (1,3);
\draw (.5,2.75) node {$j^m$};
\draw (.5,.5) -- (2.5,2.5) [<-][thick];
\draw (2.5,.5) -- (.5,2.5) [->][thick];
\draw (2,0) rectangle (3,.5);
\draw (2.5,.25) node {$j^m$};
\draw (2,2.5) rectangle (3,3);
\draw (2.5,2.75) node {$i^n$};

\draw (4.5,1.5) node {$f_k =$};
\draw (5,0) rectangle (6,.5);
\draw (5.5,.25) node {$i^n$};
\draw (5,2.5) rectangle (6,3);
\draw (5.5,2.75) node {$j^{m-k}$};
\draw (5.5,.5) -- (7.5,2.5) [<-][thick];
\draw (7.5,.5) -- (5.5,2.5) [->][thick];
\draw (7,0) rectangle (8,.5);
\draw (7.5,.25) node {$j^{m}$};
\draw (7,2.5) rectangle (8,3);
\draw (7.5,2.75) node {$i^{n-k}$};
\draw (6,.5) arc (180:0:0.5cm)[<-] [thick];
\draw [black] (6.97,.65) circle (2pt);
\filldraw [black] (6.5,1) circle (2pt);
\draw (6.8,.65) node{$\scriptstyle{k}$};

\draw (9.5,1.5) node {$f_k' =$};
\draw (10,0) rectangle (11,.5);
\draw (10.5,.25) node {$i^n$};
\draw (10,2.5) rectangle (11,3);
\draw (10.5,2.75) node {$j^{m-k}$};
\draw (10.5,.5) -- (12.5,2.5) [<-][thick];
\draw (12.5,.5) -- (10.5,2.5) [->][thick];
\draw (12,0) rectangle (13,.5);
\draw (12.5,.25) node {$j^{m}$};
\draw (12,2.5) rectangle (13,3);
\draw (12.5,2.75) node {$i^{n-k}$};

\draw (11,.5) .. controls (11.5,.7) .. (12,.5)[<-][thick];
\draw [black] (11.8,.8) circle (2pt);
\filldraw [black] (11.5,.65) circle (2pt);

\draw (10.8,.5) .. controls (11.5,1.1) .. (12.2,.5)[<-][thick];
\filldraw [black] (11.5,.95) circle (2pt);

\draw (11.5,1.1) node{$\scriptstyle{k-1}$};
\draw (13.1,0) node{.};

\end{tikzpicture}
\end{equation*}

Next let
\begin{equation*}
g =
\begin{pmatrix}
g_0 & g_1 & g_1' & \cdots & g_r & g_r'
\end{pmatrix}
\colon
\widetilde{P}_j^{(m)}  \widetilde{Q}_i^{(n)} \oplus \bigoplus_{k \geq 1} (\widetilde{P}_j^{(m-k)} \widetilde{Q}_i^{(n-k)} \lbrace k \rbrace \oplus \widetilde{P}_j^{(m-k)} \widetilde{Q}_i^{(n-k)} \lbrace k-1 \rbrace) \rightarrow
\widetilde{Q}_i^{(n)} \widetilde{P}_j^{(m)}
\end{equation*}
where
\begin{equation*}
g_0 \colon \widetilde{P}_j^{(m)}  \widetilde{Q}_i^{(n)} \rightarrow \widetilde{Q}_i^{(n)}  \widetilde{P}_j^{(m)}
\end{equation*}
and for $ k>0 $,
\begin{equation*}
g_k \colon \widetilde{P}_j^{(m-k)}  \widetilde{Q}_i^{(n-k)} \rightarrow \widetilde{Q}_i^{(n)}  \widetilde{P}_j^{(m)} \lbrace k \rbrace
\end{equation*}
\begin{equation*}
g_k' \colon \widetilde{P}_j^{(m-k)}  \widetilde{Q}_i^{(n-k)} \rightarrow \widetilde{Q}_i^{(n)}  \widetilde{P}_j^{(m)} \lbrace k-1 \rbrace
\end{equation*}
and the maps are given by the diagrams

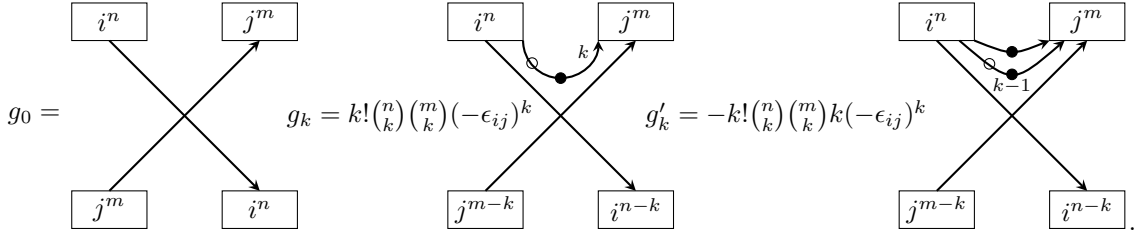
\begin{equation*}
\label{idPkar}
\begin{tikzpicture}[>=stealth]
\draw (-.5,1.5) node {$g_0 =$};
\draw (0,0) rectangle (1,.5);
\draw (.5,2.75) node {$i^n$};
\draw (0,2.5) rectangle (1,3);
\draw (.5,.25) node {$j^m$};
\draw (.5,.5) -- (2.5,2.5) [->][thick];
\draw (2.5,.5) -- (.5,2.5) [<-][thick];
\draw (2,0) rectangle (3,.5);
\draw (2.5,2.75) node {$j^m$};
\draw (2,2.5) rectangle (3,3);
\draw (2.5,.25) node {$i^n$};

\draw (4.5,1.5) node {$g_k =k! \binom{n}{k} \binom{m}{k} (-\epsilon_{ij})^k$};
\draw (5,0) rectangle (6,.5);
\draw (5.5,2.75) node {$i^n$};
\draw (5,2.5) rectangle (6,3);
\draw (5.5,.25) node {$j^{m-k}$};
\draw (5.5,.5) -- (7.5,2.5) [->][thick];
\draw (7.5,.5) -- (5.5,2.5) [<-][thick];
\draw (7,0) rectangle (8,.5);
\draw (7.5,2.75) node {$j^{m}$};
\draw (7,2.5) rectangle (8,3);
\draw (7.5,.25) node {$i^{n-k}$};
\draw (6,2.5) arc (-180:0:0.5cm)[->] [thick];
\filldraw [black] (6.5,2.0) circle (2pt);
\draw [black] (6.12,2.2) circle (2pt);
\draw (6.8,2.35) node{$\scriptstyle{k}$};

\draw (9.5,1.5) node {$g_k' = -k! \binom{n}{k} \binom{m}{k}k (-\epsilon_{ij})^k$};
\draw [shift={+(1,0)}](10,0) rectangle (11,.5);
\draw [shift={+(1,0)}](10.5,2.75) node {$i^n$};
\draw [shift={+(1,0)}](10,2.5) rectangle (11,3);
\draw [shift={+(1,0)}](10.5,.25) node {$j^{m-k}$};
\draw [shift={+(1,0)}](10.5,.5) -- (12.5,2.5) [->][thick];
\draw [shift={+(1,0)}](12.5,.5) -- (10.5,2.5) [<-][thick];
\draw [shift={+(1,0)}](12,0) rectangle (13,.5);
\draw [shift={+(1,0)}](12.5,2.75) node {$j^{m}$};
\draw [shift={+(1,0)}](12,2.5) rectangle (13,3);
\draw [shift={+(1,0)}](12.5,.25) node {$i^{n-k}$};

\draw [shift={+(1,0)}](11,2.5) .. controls (11.5,2.3) .. (12,2.5)[->][thick];
\draw [shift={+(1,0)}][black] (11.2,2.19) circle (2pt);
\filldraw [shift={+(1,0)}][black] (11.5,2.35) circle (2pt);

\draw [shift={+(1,0)}](10.8,2.5) .. controls (11.5,1.9) .. (12.2,2.5)[->][thick];
\filldraw [shift={+(1,0)}][black] (11.5,2.05) circle (2pt);

\draw [shift={+(1,0)}](11.5,1.9) node{$\scriptstyle{k-1}$};
\draw [shift={+(1,0)}](13.1,0) node{.};

\end{tikzpicture}
\end{equation*}

It follows from Lemma \ref{crossinglemma5} that $ gf $ is the identity map.
Lemma \ref{crossinglemma6} and relation ~\eqref{K4} imply that $ f_k g_k $ and $ f_k' g_k' $ are the identity maps for all $ k $.
The maps $ f_k g_l $ and $ f_k' g_l' $ for $ k \neq l $ will be zero since each term will contain a left hand curl after expansion of the middle idempotents.
The maps $ f_k' g_l $ and $ f_k g_l' $ will vanish for all values of $ k $ and $ l $ since there will be a left hand curl or a closed circle with a single hollow dot on it.

\end{proof}

\begin{prop}
\label{relationsinK3}
For a node $ i $ of the Dynkin diagram, $  \widetilde{Q}_i^{(n)} \widetilde{P}_i^{(m)} \cong $
\begin{align*}
\widetilde{P}_i^{(m)} \widetilde{Q}_i^{(n)} \oplus \bigoplus_{k \geq 1} [&\bigoplus_{l=0}^k (\widetilde{P}_i^{(m-k)} \widetilde{Q}_i^{(n-k)} \langle -k+2l \rangle \oplus \widetilde{P}_i^{(m-k)} \widetilde{Q}_i^{(n-k)} \langle -k+2l \rangle \lbrace 1 \rbrace) \oplus\\
&\bigoplus_{l=1}^{k-1} (\widetilde{P}_i^{(m-k)} \widetilde{Q}_i^{(n-k)} \langle -k+2l \rangle \oplus \widetilde{P}_i^{(m-k)} \widetilde{Q}_i^{(n-k)} \langle -k+2l \rangle \lbrace 1 \rbrace)].
\end{align*}
\end{prop}

\begin{proof}
We construct maps back and forth and leave it to the reader to check that they are inverses.
Define
\begin{equation*}
A_{k,l,\diamond,c} \colon \widetilde{Q}_{i}^{(n)} \widetilde{P}_i^{(m)} \rightarrow \widetilde{P}_{i}^{(m-k)} \widetilde{Q}_i^{(n-k)} \langle -k+2l \rangle \lbrace c \rbrace
\end{equation*}

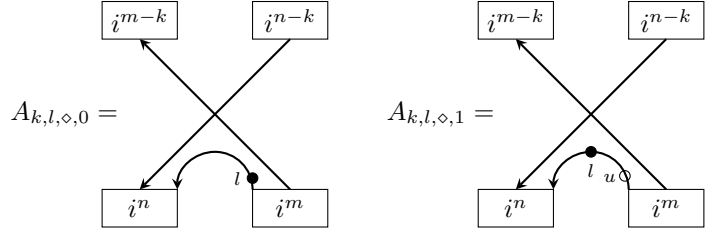
\begin{equation*}
\label{idPkar}
\begin{tikzpicture}[>=stealth]
\draw (4.5,1.5) node {$A_{k,l,\diamond,0} =$};
\draw (5,0) rectangle (6,.5);
\draw (5.5,.25) node {$i^n$};
\draw (5,2.5) rectangle (6,3);
\draw (5.5,2.75) node {$i^{m-k}$};
\draw (5.5,.5) -- (7.5,2.5) [<-][thick];
\draw (7.5,.5) -- (5.5,2.5) [->][thick];
\draw (7,0) rectangle (8,.5);
\draw (7.5,.25) node {$i^{m}$};
\draw (7,2.5) rectangle (8,3);
\draw (7.5,2.75) node {$i^{n-k}$};
\draw (6,.5) arc (180:0:0.5cm)[<-] [thick];
\filldraw [black] (7,.65) circle (2pt);
\draw (6.8,.65) node{$\scriptstyle{l}$};

\draw (9.5,1.5) node {$A_{k,l,\diamond,1} =$};
\draw (10,0) rectangle (11,.5);
\draw (10.5,.25) node {$i^n$};
\draw (10,2.5) rectangle (11,3);
\draw (10.5,2.75) node {$i^{m-k}$};
\draw (10.5,.5) -- (12.5,2.5) [<-][thick];
\draw (12.5,.5) -- (10.5,2.5) [->][thick];
\draw (12,0) rectangle (13,.5);
\draw (12.5,.25) node {$i^{m}$};
\draw (12,2.5) rectangle (13,3);
\draw (12.5,2.75) node {$i^{n-k}$};
\draw (11,.5) arc (180:0:0.5cm)[<-] [thick];
\draw [black] (11.95,.68) circle (2pt);
\draw (11.75,.65) node{$\scriptstyle{u}$};
\filldraw [black] (11.5,1) circle (2pt);
\draw (11.5,.75) node{$\scriptstyle{l}$};
\end{tikzpicture}
\end{equation*}
where the label $ l $ for the solid dot means put one copy of a solid dot on each of the first $ l $ strands counting from bottom to top and the label $ u $ for the hollow dot means put a single hollow dot on the upper most arc.

Define maps for $ 1 \leq l \leq k-1$
\begin{equation*}
A_{k,l,\heartsuit,c} \colon \widetilde{Q}_{i}^{(n)} \widetilde{P}_i^{(m)} \rightarrow \widetilde{P}_{i}^{(m-k)} \widetilde{Q}_i^{(n-k)} \langle -k+2l \rangle \lbrace c \rbrace
\end{equation*}

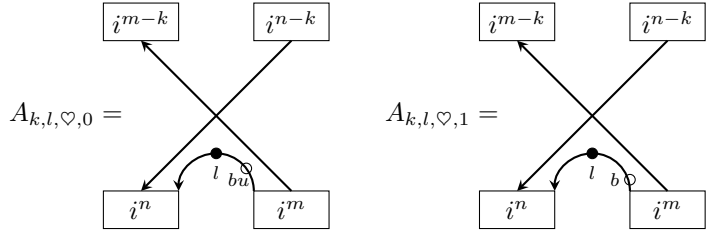
\begin{equation*}
\label{idPkar}
\begin{tikzpicture}[>=stealth]
\draw (4.5,1.5) node {$A_{k,l,\heartsuit,0} =$};
\draw (5,0) rectangle (6,.5);
\draw (5.5,.25) node {$i^n$};
\draw (5,2.5) rectangle (6,3);
\draw (5.5,2.75) node {$i^{m-k}$};
\draw (5.5,.5) -- (7.5,2.5) [<-][thick];
\draw (7.5,.5) -- (5.5,2.5) [->][thick];
\draw (7,0) rectangle (8,.5);
\draw (7.5,.25) node {$i^{m}$};
\draw (7,2.5) rectangle (8,3);
\draw (7.5,2.75) node {$i^{n-k}$};
\draw (6,.5) arc (180:0:0.5cm)[<-] [thick];
\draw [black] (6.9,.8) circle (2pt);
\draw (6.8,.65) node{$\scriptstyle{bu}$};
\filldraw [black] (6.5,1) circle (2pt);
\draw (6.5,.75) node{$\scriptstyle{l}$};

\draw (9.5,1.5) node {$A_{k,l,\heartsuit,1} =$};
\draw (10,0) rectangle (11,.5);
\draw (10.5,.25) node {$i^n$};
\draw (10,2.5) rectangle (11,3);
\draw (10.5,2.75) node {$i^{m-k}$};
\draw (10.5,.5) -- (12.5,2.5) [<-][thick];
\draw (12.5,.5) -- (10.5,2.5) [->][thick];
\draw (12,0) rectangle (13,.5);
\draw (12.5,.25) node {$i^{m}$};
\draw (12,2.5) rectangle (13,3);
\draw (12.5,2.75) node {$i^{n-k}$};
\draw (11,.5) arc (180:0:0.5cm)[<-] [thick];
\draw [black] (12,.65) circle (2pt);
\draw (11.8,.65) node{$\scriptstyle{b}$};
\filldraw [black] (11.5,1) circle (2pt);
\draw (11.5,.75) node{$\scriptstyle{l}$};
\end{tikzpicture}
\end{equation*}
where the label $ l $ for the solid dot means put one copy of a solid dot on each of the first $ l $ caps counting from bottom to top. The label $ b $ for the hollow dot means the bottom cap gets a single hollow dot. The label $ bu $ for the hollow dot means the bottom and top caps each get hollow dots positioned on a vertical line.

Define maps
\begin{equation*}
B_{k,l,\diamond,c} \colon \widetilde{P}_{i}^{(m-k)} \widetilde{Q}_i^{(n-k)} \langle -k+2l \rangle \lbrace c \rbrace \rightarrow \widetilde{Q}_{i}^{(n)} \widetilde{P}_i^{(m)}
\end{equation*}
\begin{equation*}
\label{idPkar}
\begin{tikzpicture}[>=stealth]
\draw [shift={+(-2,0)}](4.2,1.5) node {$B_{k,l,\diamond,0} ={\binom{n}{k}\binom{m}{k}\binom{k}{l}k!}$};
\draw [shift={+(-2,0)}](5,0) rectangle (6,.5);
\draw [shift={+(-2,0)}](5.5,2.75) node {$i^n$};
\draw [shift={+(-2,0)}](5,2.5) rectangle (6,3);
\draw [shift={+(-2,0)}](5.5,.25) node {$i^{m-k}$};
\draw [shift={+(-2,0)}](5.5,.5) -- (7.5,2.5) [->][thick];
\draw [shift={+(-2,0)}](7.5,.5) -- (5.5,2.5) [<-][thick];
\draw [shift={+(-2,0)}](7,0) rectangle (8,.5);
\draw [shift={+(-2,0)}](7.5,2.75) node {$i^{m}$};
\draw [shift={+(-2,0)}](7,2.5) rectangle (8,3);
\draw [shift={+(-2,0)}](7.5,.25) node {$i^{n-k}$};
\draw [shift={+(-2,0)}](6,2.5) arc (-180:0:0.5cm)[->] [thick];
\filldraw [shift={+(-2,0)}][black] (7,2.35) circle (2pt);
\draw [shift={+(-2,0)}](6.5,2.3) node{$\scriptstyle{k-l}$};

\draw [shift={+(-1,0)}](9.5,1.5) node {$B_{k,l,\diamond,1} =-{\binom{n}{k}\binom{m}{k}\binom{k-1}{l}k!k}$};
\draw (10,0) rectangle (11,.5);
\draw (10.5,2.75) node {$i^n$};
\draw (10,2.5) rectangle (11,3);
\draw (10.5,.25) node {$i^{m-k}$};
\draw (10.5,.5) -- (12.5,2.5) [->][thick];
\draw (12.5,.5) -- (10.5,2.5) [<-][thick];
\draw (12,0) rectangle (13,.5);
\draw (12.5,2.75) node {$i^{m}$};
\draw (12,2.5) rectangle (13,3);
\draw (12.5,.25) node {$i^{n-k}$};
\draw (11,2.5) arc (-180:0:0.5cm)[->] [thick];
\filldraw [black] (12,2.35) circle (2pt);
\draw (11.5,2.3) node{$\scriptstyle{k-l}$};
\draw [black] (11.5,2) circle (2pt);
\draw (11.5,1.75) node{$\scriptstyle{b}$};
\end{tikzpicture}
\end{equation*}
where the label $ k-l $ for the solid dot means that each of the first $ k-l $ cups starting from the bottom should carry a solid dot and the label $ b $ for the hollow dot means that the bottom cup should carry a hollow dot.  By convention we define $\binom{k-1}{k}=1$.

Define maps for $ 1 \leq l \leq k-1$
\begin{equation*}
B_{k,l,\heartsuit,c} \colon \widetilde{P}_{i}^{(m-k)} \widetilde{Q}_i^{(n-k)} \langle -k+2l \rangle \lbrace c \rbrace \rightarrow \widetilde{Q}_{i}^{(n)} \widetilde{P}_i^{(m)}
\end{equation*}
\begin{equation*}
\label{idPkar}
\begin{tikzpicture}[>=stealth]
\draw [shift={+(-4,0)}](4.5,1.5) node {$B_{k,l,\heartsuit,0} ={\binom{n}{k}\binom{m}{k}\binom{k-2}{l-1}k!k(k-1)}$};
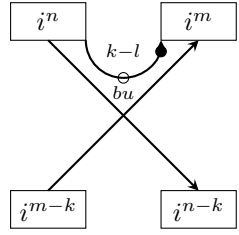
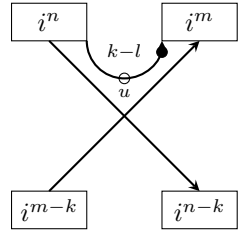
\draw [shift={+(-1.5,0)}](5,0) rectangle (6,.5);
\draw [shift={+(-1.5,0)}](5.5,2.75) node {$i^n$};
\draw [shift={+(-1.5,0)}](5,2.5) rectangle (6,3);
\draw [shift={+(-1.5,0)}](5.5,.25) node {$i^{m-k}$};
\draw [shift={+(-1.5,0)}](5.5,.5) -- (7.5,2.5) [->][thick];
\draw [shift={+(-1.5,0)}](7.5,.5) -- (5.5,2.5) [<-][thick];
\draw [shift={+(-1.5,0)}](7,0) rectangle (8,.5);
\draw [shift={+(-1.5,0)}](7.5,2.75) node {$i^{m}$};
\draw [shift={+(-1.5,0)}](7,2.5) rectangle (8,3);
\draw [shift={+(-1.5,0)}](7.5,.25) node {$i^{n-k}$};
\draw [shift={+(-1.5,0)}](6,2.5) arc (-180:0:0.5cm)[->] [thick];
\filldraw [shift={+(-1.5,0)}][black] (7,2.35) circle (2pt);
\draw [shift={+(-1.5,0)}](6.5,2.35) node{$\scriptstyle{k-l}$};
\draw [shift={+(-1.5,0)}][black] (6.5,2) circle (2pt);
\draw [shift={+(-1.5,0)}](6.5,1.8) node{$\scriptstyle{bu}$};
\end{tikzpicture}
\end{equation*}

\begin{equation*}
\begin{tikzpicture}[>=stealth]
\draw [shift={+(-2.4,0)}](9.5,1.5) node {$B_{k,l,\heartsuit,1} ={\binom{n}{k}\binom{m}{k}\binom{k-1}{l-1}k!k}$};
\draw (10,0) rectangle (11,.5);
\draw (10.5,2.75) node {$i^n$};
\draw (10,2.5) rectangle (11,3);
\draw (10.5,.25) node {$i^{m-k}$};
\draw (10.5,.5) -- (12.5,2.5) [->][thick];
\draw (12.5,.5) -- (10.5,2.5) [<-][thick];
\draw (12,0) rectangle (13,.5);
\draw (12.5,2.75) node {$i^{m}$};
\draw (12,2.5) rectangle (13,3);
\draw (12.5,.25) node {$i^{n-k}$};
\draw (11,2.5) arc (-180:0:0.5cm)[->] [thick];
\filldraw [black] (12,2.35) circle (2pt);
\draw (11.5,2.35) node{$\scriptstyle{k-l}$};
\draw [black] (11.5,2) circle (2pt);
\draw (11.5,1.8) node{$\scriptstyle{u}$};
\end{tikzpicture}
\end{equation*}
where the label $ k-l $ for the solid dot means that each of the first $ k-l $ cups starting from the bottom should carry a solid dot. The label $ u $ for the hollow dot means that the uppermost cup should carry a hollow dot.
The label $ bu $ for the hollow dot means that the uppermost and bottom cups carry hollow dots positioned on a vertical line.

The maps $ A_{k,l,\diamond,c} $ and $ A_{k,l,\heartsuit,c} $ induce the isomorphism in the proposition while the maps $ B_{k,l,\diamond,c} $ and $ B_{k,l,\heartsuit,c} $ induce the inverse isomorphism.
\end{proof}

\begin{prop}
\label{relationsinK5}
For all $ i \neq j $ with $ a_{ij} = 0 $,
there is an isomorphism of objects in $\mathcal{K}_{\Gamma}^- $:
\begin{equation*}
\widetilde{Q}_i^{(n)} \widetilde{P}_j^{(m)} \cong \widetilde{P}_j^{(m)} \widetilde{Q}_i^{(n)}.
\end{equation*}
\end{prop}

\begin{proof}
The proof of this proposition is similar to the proof of Proposition ~\ref{relationsinK4} except that now the last term two terms in the right hand side of ~\eqref{K6} vanish in this case making this proposition easier.
\end{proof}



\section{A representation of $\mathcal{H}_{\Gamma}^-$}
\label{2rep}
\subsection{The category $ C_{n}^{\Gamma} $ and functors $ P(n), Q(n) $}
Let $ C_{n}^{\Gamma} $ be the category of $ \Z \times \Z_2$-graded finite-dimensional left $ B_{n}^{\Gamma} $-modules.

$ B_{n}^{\Gamma} $ is naturally a subalgebra of $ B_{n+1}^{\Gamma} $ where $ b_1 \otimes \cdots \otimes  b_n $ maps to $ b_1 \otimes \cdots \otimes b_n \otimes 1 $, $ s_i $ maps to $ s_i $ and $ c_i $ maps to $ c_i $.
Thus $ B_{n+1}^{\Gamma} $ is a $ (B_{n+1}^{\Gamma}, B_{n}^{\Gamma}) $-bimodule.
Now we define functors:
\begin{equation*}
P(n) \colon C_{n}^{\Gamma} \rightarrow C_{n+1}^{\Gamma} \hspace{.5in} M \mapsto B_{n+1}^{\Gamma} \otimes_{B_{n}^{\Gamma}} M \langle 1 \rangle
\end{equation*}
\begin{equation*}
Q(n) \colon C_{n+1}^{\Gamma} \rightarrow C_{n}^{\Gamma} \hspace{.5in} M \mapsto B_{n+1}^{\Gamma} \otimes_{B_{n+1}^{\Gamma}} M
\end{equation*}
where in the definition of $ Q(n) $, we view the first tensor factor $ B_{n+1}^{\Gamma} $ as a $ (B_{n}^{\Gamma}, B_{n+1}^{\Gamma})$-bimodule.

\subsection{Natural transformations}
\subsubsection{$X(b)$}
Utilizing the inclusion
\begin{equation*}
B_{}^{\Gamma} \mapsto B_{n+1}^{\Gamma} \hspace{.5in} a \mapsto (1 \otimes \cdots \otimes 1 \otimes a)
\end{equation*}
which anti-commutes with the subalgebra $ B_{n}^{\Gamma} $, for a homogeneous element $ b \in B_{}^{\Gamma} $ we define
\begin{equation*}
X(b) \colon B_{n+1}^{\Gamma} \rightarrow B_{n+1}^{\Gamma} \hspace{.5in} x \mapsto (-1)^{|x||b|} xb
\end{equation*}
giving a natural transformation
\begin{equation*}
X(b) \colon P(n) \rightarrow P(n) \langle |b| \rangle.
\end{equation*}

\subsubsection{$X(c_{n+1})$}
The $ (B_{n+1}^{\Gamma}, B_n^{\Gamma})$-bimodule map $ B_{n+1}^{\Gamma} \rightarrow B_{n+1}^{\Gamma} $ where $ x \mapsto (-1)^{||x||} xc_{n+1} $ defines a natural transformation
\begin{equation*}
X(c_{n+1}) \colon P(n) \rightarrow P(n) \lbrace 1 \rbrace.
\end{equation*}

\subsubsection{$Y(b)$}
Recall the inclusion
\begin{equation*}
B_{}^{\Gamma} \mapsto B_{n+1}^{\Gamma} \hspace{.5in} a \mapsto (1 \otimes \cdots \otimes 1 \otimes a).
\end{equation*}
Then for a homogeneous element $ b \in B_{}^{\Gamma} $ we define
\begin{equation*}
Y(b) \colon B_{n+1}^{\Gamma} \rightarrow B_{n+1}^{\Gamma} \hspace{.5in} x \mapsto bx
\end{equation*}
giving a natural transformation
\begin{equation*}
Y(b) \colon Q(n) \rightarrow Q(n) \langle |b| \rangle.
\end{equation*}

\subsubsection{$Y(c_{n+1})$}
The $ (B_n^{\Gamma}, B_{n+1}^{\Gamma})$-bimodule map $ B_{n+1}^{\Gamma} \rightarrow B_{n+1}^{\Gamma} $ where $ x \mapsto c_{n+1} x $ defines a natural transformation
\begin{equation*}
Y(c_{n+1}) \colon Q(n) \rightarrow Q(n) \lbrace 1 \rbrace.
\end{equation*}

\subsubsection{$T$}
The functor $ P(n+1) \circ P(n) \colon C_{n}^{\Gamma} \rightarrow C_{n+2}^{\Gamma} $ is given by tensoring with the
$ (B_{n+2}^{\Gamma}, B_{n}^{\Gamma}) $-bimodule
\begin{equation*}
B_{n+2}^{\Gamma} \langle 1 \rangle \otimes_{B_{n+1}^{\Gamma}} B_{n+1}^{\Gamma} \langle 1 \rangle \cong B_{n+2}^{\Gamma} \langle 2 \rangle.
\end{equation*}
There is a $ (B_{n+2}^{\Gamma}, B_{n}^{\Gamma}) $-bimodule homomorphism $ T \colon B_{n+2}^{\Gamma} \rightarrow B_{n+2}^{\Gamma} $
given by mapping $ x \mapsto x s_{n+1} $.  This gives a natural transformation:
\begin{equation*}
T \colon P(n+1) \circ P(n) \rightarrow P(n+1) \circ P(n).
\end{equation*}

\subsubsection{$T_{--}$}
The functor $ Q(n) \circ Q(n+1) \colon C_{n+2}^{\Gamma} \rightarrow C_{n}^{\Gamma} $ is given by tensoring with the
$ (B_{n}^{\Gamma}, B_{n+2}^{\Gamma}) $-bimodule
\begin{equation*}
B_{n+1}^{\Gamma}  \otimes_{B_{n+1}^{\Gamma}} B_{n+2}^{\Gamma}  \cong B_{n+2}^{\Gamma}.
\end{equation*}
There is a $ (B_{n}^{\Gamma}, B_{n+2}^{\Gamma}) $-bimodule homomorphism $ T \colon B_{n+2}^{\Gamma} \rightarrow B_{n+2}^{\Gamma} $
given by mapping $ x \mapsto s_{n+1} x $.  This gives a natural transformation:
\begin{equation*}
T \colon Q(n) \circ Q(n+1) \rightarrow Q(n) \circ Q(n+1).
\end{equation*}

\subsubsection{$T_{-+}, T_{+-}$}
The functor $ Q(n) P(n) $ is given by tensoring with the
$ (B_{n}^{\Gamma}, B_{n}^{\Gamma}) $-bimodule
\begin{equation*}
B_{n+1}^{\Gamma}  \otimes_{B_{n+1}^{\Gamma}} B_{n+1}^{\Gamma} \langle 1 \rangle \cong B_{n+1}^{\Gamma} \langle 1 \rangle.
\end{equation*}

The functor $ P(n-1) Q(n-1) $ is given by tensoring with
$ (B_{n}^{\Gamma}, B_{n}^{\Gamma}) $-bimodule
\begin{equation*}
B_{n}^{\Gamma} \langle 1 \rangle \otimes_{B_{n-1}^{\Gamma}} B_{n}^{\Gamma}.
\end{equation*}

Now following \cite[Section 15.6]{Klesh} we have a $ (B_{n}^{\Gamma}, B_{n}^{\Gamma}) $-bimodule homomorphism
\begin{equation*}
\phi_{-+} \colon B_{n+1}^{\Gamma} \langle 1 \rangle \rightarrow B_{n}^{\Gamma} \langle 1 \rangle \otimes_{B_{n-1}^{\Gamma}} B_{n}^{\Gamma}.
\end{equation*}
This homomorphism maps an element $ g \in B_{n}^{\Gamma} $ to zero and an element $ g s_n h $ to $ g \otimes h $ where $ g, h \in B_{n}^{\Gamma} $.
This bimodule homomorphism induces a morphism of functors
\begin{equation*}
T_{-+} \colon Q(n) P(n) \rightarrow P(n-1) Q(n-1).
\end{equation*}

There is also a $ (B_{n}^{\Gamma}, B_{n}^{\Gamma}) $-bimodule homomorphism
\begin{equation*}
\phi_{+-} \colon B_{n}^{\Gamma} \langle 1 \rangle \otimes_{B_{n-1}^{\Gamma}} B_{n}^{\Gamma} \rightarrow B_{n+1}^{\Gamma} \langle 1 \rangle.
\end{equation*}
This homomorphism maps an element $ g \otimes h $ to $ g s_n h $.
See \cite[Section 15.6]{Klesh} for more details and compare with \cite[Section 3.3]{Kh} in the untwisted case.
The map $ \phi_{+-} $ gives rise to a morphism of functors
\begin{equation*}
T_{+-} \colon P(n-1) Q(n-1) \rightarrow Q(n) P(n).
\end{equation*}

\subsubsection{$\adj_{QP}^{\Id}$}
The functor $ Q(n) \circ P(n) \colon C_{n}^{\Gamma} \rightarrow C_{n}^{\Gamma} $ is given by tensoring with the
$ (B_{n}^{\Gamma}, B_{n}^{\Gamma}) $-bimodule
\begin{equation*}
B_{n+1}^{\Gamma}  \otimes_{B_{n+1}^{\Gamma}} B_{n+1}^{\Gamma} \langle 1 \rangle \cong B_{n+1}^{\Gamma} \langle 1 \rangle.
\end{equation*}
We define a $ (B_{n}^{\Gamma}, B_{n}^{\Gamma}) $-bimodule map
\begin{equation*}
B_{n+1}^{\Gamma} \rightarrow B_{n}^{\Gamma} \langle -2 \rangle
\end{equation*}
given by
\begin{equation*}
(g_1 \otimes \cdots \otimes g_{n+1}) s_n \mapsto 0 \hspace{.5in} (g_1 \otimes \cdots \otimes g_{n+1}) c_{n+1} \mapsto 0 \hspace{.5in} (g_1 \otimes \cdots \otimes g_n \otimes v_1) \mapsto 0
\end{equation*}
\begin{equation*}
(g_1 \otimes \cdots \otimes g_n \otimes v_2) \mapsto 0 \hspace{.5in} (g_1 \otimes \cdots \otimes g_n \otimes \gamma \omega) \mapsto \delta_{\gamma,e}(g_1 \otimes \cdots \otimes g_n)
\end{equation*}
for $ \gamma \in \Gamma $.
This gives rise to a natural transformation
\begin{equation*}
\adj_{QP}^{\Id} \colon Q(n) \circ P(n) \rightarrow \Id \langle-1 \rangle.
\end{equation*}

\subsubsection{$\adj_{\Id}^{QP}$}
The subalgebra inclusion of $ B_{n}^{\Gamma} $ in $ B_{n+1}^{\Gamma} $ gives rise to a natural transformation
\begin{equation*}
\adj_{\Id}^{QP} \colon \Id \rightarrow Q(n) \circ P(n) \langle -1 \rangle.
\end{equation*}

\subsubsection{$ \adj_{PQ}^{\Id}$}
There is a natural multiplication map
\begin{equation*}
B_{n+1}^{\Gamma} \otimes_{B_{n}^{\Gamma}} B_{n+1}^{\Gamma} \rightarrow B_{n+1}^{\Gamma} \hspace{.5in} a \otimes b \mapsto ab.
\end{equation*}
This gives rise to a natural transformation
\begin{equation*}
\adj_{PQ}^{\Id} \colon P(n) \circ Q(n) \rightarrow \Id \langle 1 \rangle.
\end{equation*}

\subsubsection{$\adj_{\Id}^{PQ}$}
Recall that $ \mathcal{B} $ is a basis of $ B^{\Gamma} $.
For $ b \in \mathcal{B} $, let $ b_l = 1 \otimes \cdots \otimes 1 \otimes b \otimes 1 \otimes \cdots \otimes 1$ be the element of $ B^{\Gamma}_{n+1} $ where $ b $ is in the $ l$th component of
$ B^{\Gamma} \otimes \cdots \otimes B^{\Gamma}$.
\begin{lemma}
There is a $ (B_{n+1}^{\Gamma}, B_{n+1}^{\Gamma}) $-bimodule homomorphism $ B_{n+1}^{\Gamma} \rightarrow B_{n+1}^{\Gamma} \otimes_{B_{n}^{\Gamma}} B_{n+1}^{\Gamma} \langle 2 \rangle $
\begin{equation*}
1 \mapsto \sum_{b \in \mathcal{B}} \sum_{i=1}^{n+1} (s_i \cdots s_n \check{b}_{n+1} \otimes b_{n+1} s_n \cdots s_i + s_i \cdots s_n c_{n+1} \check{b}_{n+1} \otimes b_{n+1} c_{n+1} s_n \cdots s_i).
\end{equation*}
\end{lemma}

\begin{proof}
Consider the element
\begin{equation*}
\rho(i,b_l) = s_i \cdots s_n \check{b}_l \otimes b_l s_n \cdots s_i + s_i \cdots s_n c_{n+1} \check{b}_l \otimes b_l c_{n+1} s_n \cdots s_i.
\end{equation*}
When $ l=n+1 $, this is a term in the summation appearing in the lemma.  We must prove that for any $ \zeta \in B^{\Gamma}_n $ that
\begin{equation*}
\zeta \sum_{b \in \mathcal{B}} \sum_{i=1}^{n+1} \rho(i,b_{n+1}) = \sum_{b \in \mathcal{B}} \sum_{i=1}^{n+1} \rho(i,b_{n+1}) \zeta.
\end{equation*}

Let $ \zeta = s_k $ where $ k < i-1$.  Then it is clear that $ s_k \rho(i,b_{n+1}) = \rho(i,b_{n+1}) s_k $ since $ s_k $ commutes with every transposition appearing in $ \rho(i,b_{n+1}) $
and it does not permute the component that $ b $ appears in.

Now assume that $ k>i $. Then
\begin{eqnarray*}
s_k \rho(i,b_{n+1}) &= &s_k(s_i \cdots s_n \check{b}_{n+1} \otimes b_{n+1} s_n \cdots s_i + s_i \cdots s_n c_{n+1} \check{b}_{n+1} \otimes b_{n+1} c_{n+1} s_n \cdots s_i)\\
&= &s_i \cdots s_k s_{k-1} s_k \cdots s_n \check{b}_{n+1} \otimes b_{n+1} s_n \cdots s_i +\\
& & s_i \cdots s_k s_{k-1} s_k \cdots s_n c_{n+1} \check{b}_{n+1} \otimes b_{n+1} c_{n+1} s_n \cdots s_i\\
&= &s_i \cdots s_{k-1} s_{k} s_{k-1} \cdots s_n \check{b}_{n+1} \otimes b_{n+1} s_n \cdots s_i +\\
& & s_i \cdots s_{k-1} s_{k} s_{k-1} \cdots s_n c_{n+1} \check{b}_{n+1} \otimes b_{n+1} c_{n+1} s_n \cdots s_i\\
&=& s_i \cdots s_n s_{k-1} \check{b}_{n+1} \otimes b_{n+1} s_n \cdots s_i + s_i \cdots s_n s_{k-1} c_{n+1} \check{b}_{n+1} \otimes b_{n+1} c_{n+1} s_n \cdots s_i\\
&=& s_i \cdots s_n  \check{b}_{n+1} \otimes b_{n+1} s_{k-1} s_n \cdots s_i + s_i \cdots s_n s_{k-1} c_{n+1} \check{b}_{n+1} \otimes b_{n+1} c_{n+1} s_{k-1} s_n \cdots s_i\\
&=& \rho(i,b_{n+1}) s_k.
\end{eqnarray*}
Note that in the fifth equality, we used the fact that the transposition $ s_{k-1} $ doesn't permute $ b_{n+1} $ to another component nor does it change the Clifford generator $ c_{n+1} $.
Next consider
\begin{eqnarray*}
s_k(\rho(k,b_{n+1})+\rho(k+1,b_{n+1})) &= &s_k s_k \cdots s_n \check{b}_{n+1} \otimes b_{n+1} s_n \cdots s_k +\\
& &s_k s_{k+1} \cdots s_n \check{b}_{n+1} \otimes b_{n+1} s_n \cdots s_{k+1}+\\
& & s_k s_k \cdots s_n c_{n+1} \check{b}_{n+1} \otimes b_{n+1} c_{n+1} s_n \cdots s_k +\\
& & s_k s_{k+1} \cdots s_n c_{n+1} \check{b}_{n+1} \otimes b_{n+1} c_{n+1} s_n \cdots s_{k+1}\\
&= &(\rho(k,b_{n+1})+\rho(k+1,b_{n+1}))s_k.
\end{eqnarray*}
Thus when $ \zeta $ is in the symmetric group, $ \zeta $ commutes with $ \sum_{b \in \mathcal{B}} \sum_{i=1}^{n+1} \rho(i,b_{n+1}) $ as desired.

Now assume $ k < i$.  Then
\begin{eqnarray*}
c_k \rho(i,b_{n+1}) &= &c_k(s_i \cdots s_n \check{b}_{n+1} \otimes b_{n+1} s_n \cdots s_i + s_i \cdots s_n c_{n+1} \check{b}_{n+1} \otimes b_{n+1} c_{n+1} s_n \cdots s_i)\\
& = &s_i \cdots s_n c_k \check{b}_{n+1} \otimes b_{n+1} s_n \cdots s_i - s_i \cdots s_n c_{n+1} \check{b}_{n+1} c_k \otimes b_{n+1} c_{n+1} s_n \cdots s_i\\
& = &s_i \cdots s_n \check{b}_{n+1} \otimes b_{n+1} s_n \cdots s_i c_k+ s_i \cdots s_n c_{n+1} \check{b}_{n+1} \otimes b_{n+1} c_{n+1} s_n \cdots s_i c_k\\
& = &\rho(i,b_{n+1}) c_k.
\end{eqnarray*}

For $ k =i $ we have
\begin{eqnarray*}
c_i \rho(i,b_{n+1}) &= &c_i(s_i \cdots s_n \check{b}_{n+1} \otimes b_{n+1} s_n \cdots s_i + s_i \cdots s_n c_{n+1} \check{b}_{n+1} \otimes b_{n+1} c_{n+1} s_n \cdots s_i)\\
& = & s_i \cdots s_n \check{b}_{n+1} c_{n+1} \otimes b_{n+1} s_n \cdots s_i + s_i \cdots s_n \check{b}_{n+1} \otimes b_{n+1} c_{n+1} s_n \cdots s_i
\end{eqnarray*}
Similarly we get
\begin{equation*}
\rho(i,b_{n+1}) c_i = s_i \cdots s_n \check{b}_{n+1} c_{n+1} \otimes b_{n+1} s_n \cdots s_i + s_i \cdots s_n \check{b}_{n+1} \otimes b_{n+1} c_{n+1} s_n \cdots s_i.
\end{equation*}

Now for $ k > i $,
\begin{eqnarray*}
c_k \rho(i,b_{n+1}) &= &c_k(s_i \cdots s_n \check{b}_{n+1} \otimes b_{n+1} s_n \cdots s_i + s_i \cdots s_n c_{n+1} \check{b}_{n+1} \otimes b_{n+1} c_{n+1} s_n \cdots s_i)\\
&= & s_i \cdots s_n \check{b}_{n+1} c_{k-1} \otimes b_{n+1} s_n \cdots s_i - s_i \cdots s_n c_{n+1} c_{k-1} \check{b}_{n+1} \otimes b_{n+1} c_{n+1} s_n \cdots s_i\\
&= & s_i \cdots s_n \check{b}_{n+1} \otimes c_{k-1} b_{n+1} s_n \cdots s_i + s_i \cdots s_n c_{n+1} \check{b}_{n+1} \otimes b_{n+1} c_{n+1} c_{k-1} s_n \cdots s_i\\
&= & s_i \cdots s_n \check{b}_{n+1} \otimes b_{n+1} s_n \cdots s_i  c_k+ s_i \cdots s_n c_{n+1} \check{b}_{n+1} \otimes b_{n+1} c_{n+1} s_n \cdots s_i c_k\\
&= & \rho(i,b_{n+1}) c_k.
\end{eqnarray*}

Thus when $ \zeta $ is a Clifford generator, $ \zeta $ commutes with $ \sum_{b \in \mathcal{B}} \sum_{i=1}^{n+1} \rho(i,b_{n+1}) $.

Finally, $ \zeta $ commutes with $ \sum_{b \in \mathcal{B}} \sum_{i=1}^{n+1} \rho(i,b_{n+1}) $ for $ \zeta \in B^{\Gamma} \otimes \cdots \otimes B^{\Gamma} $ by
~\cite[Section 4.3.8]{CL}.

\end{proof}

This induces a natural transformation
\begin{equation*}
\adj_{\Id}^{PQ} \colon \Id \rightarrow P(n) Q(n) \langle 1 \rangle.
\end{equation*}

\subsection{Relations between natural transformations}
Now we introduce some notation for compositions of natural transformations.  Suppose $ \alpha $ is a natural transformation of the functor $ F $ and $ \beta $ is a natural transformation of the functor $ G $.
Furthermore assume that the composition of functors $ FG $ makes sense.  Then we have the horizontal composition of natural transformations
$ \alpha \beta \colon FG \rightarrow FG$.
Assume that $ \gamma $ is another natural transformation of $ F $.  Then we may vertically compose natural transformations to get:
$ \gamma \circ \alpha \colon F \rightarrow F $.
We use the same notation when the functors are given by tensoring with bimodules and the natural transformations are bimodule homomorphisms.
\subsubsection{Equation ~\ref{H1}}
\begin{prop}
\label{repH1}
If $ b \in B^{\Gamma} $, then
\begin{equation*}
T \circ (X(b)\Id) = (\Id X(b)) \circ T \colon P(n+1) P(n) \rightarrow P(n+1) P(n).
\end{equation*}
\end{prop}

\begin{proof}
The functor $ P(n+1) P(n) $ is given by tensoring with the $ (B_{n+2}^{\Gamma}, B_{n}^{\Gamma})$-bimodule $ B_{n+2}^{\Gamma} $.
If $ x \in B_{n+2}^{\Gamma} $, then $ X(b)\Id $ maps $ x $ to $ (-1)^{|x||b|} xb $ where $ b $ is in the last
component of $ B_{1}^{\Gamma} \otimes \cdots \otimes B_{1}^{\Gamma} $.
Then $ T $ maps $ (-1)^{|x||b|} xb $ to $ (-1)^{|x||b|} xb s_{n+1} $ where again $ b $ is in the last component of the tensor product.

For the right hand side of the equation, $ x $ maps to $ x s_{n+1} $ under $ T $.  Then under $ \Id X(b) $, $ x s_{n+1} $ maps to $ (-1)^{|x||b|} x s_{n+1} b $ where $ b $ is in the second to last
component of $ B_{1}^{\Gamma} \otimes \cdots \otimes B_{1}^{\Gamma} $.
In the module $ B_{n+2}^{\Gamma} $, $ (-1)^{|x||b|} x s_{n+1} b = (-1)^{|x||b|} xb s_{n+1} $ where now in the right hand side of the equation $ b $ is in the last component which verifies the proposition.
\end{proof}

\subsubsection{Equation ~\ref{H2}}
\begin{prop}
\label{repH2}
If $ b \in B^{\Gamma} $, then
\begin{equation*}
(X(b) \Id) \circ T  = T \circ (\Id X(b)) \colon P(n+1) P(n) \rightarrow P(n+1) P(n).
\end{equation*}
\end{prop}

\begin{proof}
The proof of this is very similar to that of Proposition ~\ref{repH2}.
\end{proof}

\subsubsection{Equation ~\ref{H3}}
\begin{prop}
\label{repH3}
If $ b \in B^{\Gamma} $, then there are equalities of natural transformations
\begin{enumerate}
\item $ (\adj_{QP}^{\Id}) \circ (Y(b) \Id) = (\adj_{QP}^{\Id}) \circ (\Id X(b)) $
\item $ (\adj_{PQ}^{\Id}) \circ (X(b) \Id) = (\adj_{PQ}^{\Id}) \circ (\Id Y(b)) $.
\end{enumerate}
\end{prop}

\begin{proof}
Both sides of the first equality are composites of maps
\begin{equation*}
Q(n)P(n) \rightarrow Q(n)P(n) \rightarrow \Id
\end{equation*}
corresponding to maps of bimodules
\begin{equation*}
B_{n+1}^{\Gamma} \otimes_{B_{n+1}^{\Gamma}} B_{n+1}^{\Gamma} \rightarrow B_{n+1}^{\Gamma} \otimes_{B_{n+1}^{\Gamma}} B_{n+1}^{\Gamma} \rightarrow B_{n}^{\Gamma}.
\end{equation*}
Let $ g = (g_1 \otimes \cdots \otimes g_{n+1}) \otimes cw $ where $ g_i \in B^{\Gamma} $, $ c \in Cl_{n+1} $ and $ w \in S_{n+1} $.
Let $ b \in B_{}^{\Gamma} $.
Then the left hand side of the equality maps $ g $ under $ Y(b) \Id $ to $ bg $ upon identifying the bimodule associated to $ Q(n) P(n) $ with $ B_{n+1}^{\Gamma} $.
The right hand side maps $ g $ under $ \Id X(b) $ to $ (-1)^{|g||b|} gb $ after making the same identification.

Now we have
\begin{equation*}
bg = (-1)^{|b||g_1 \cdots g_n|} (g_1 \otimes \cdots \otimes g_n \otimes b g_{n+1}) \otimes cw
\end{equation*}
\begin{equation*}
(-1)^{|g||b|} gb = (-1)^{|b||g_1 \cdots g_n| + |b||g_{n+1}|} (g_1 \otimes \cdots \otimes g_n \otimes g_{n+1}) \otimes cwb.
\end{equation*}
If $ w \notin S_n $, then both elements get mapped to zero under $ \adj_{QP}^{\Id} $ so we may assume that $ w \in S_n $.
Then
\begin{equation}
\label{something}
(-1)^{|g||b|} gb = (-1)^{|b||g_1 \cdots g_n| + |b||g_{n+1}|} (g_1 \otimes \cdots \otimes g_n \otimes g_{n+1}b) \otimes cw
\end{equation}
and so applying the map $ \adj_{QP}^{\Id} $ we get
\begin{equation*}
\adj_{QP}^{\Id}((-1)^{|g||b|} gb) = (-1)^{|b||g_1 \cdots g_n| + |b||g_{n+1}|} \tr(g_{n+1} b)(g_1 \otimes \cdots \otimes g_n \otimes 1) \otimes cw .
\end{equation*}
Similarly we get
\begin{equation*}
\adj_{QP}^{\Id}(bg) = (-1)^{|b||g_1 \cdots g_n|} \tr(bg_{n+1}) (g_1 \otimes \cdots \otimes g_n \otimes 1) \otimes cw .
\end{equation*}
The first equality in the proposition follows since $ tr(g_{n+1} b) = (-1)^{|g_{n+1}||b|} tr(b g_{n+1}) $.

Now we prove the second equality.
Both sides are composites of natural transformations
\begin{equation*}
P(n) Q(n) \rightarrow P(n) Q(n) \rightarrow \Id
\end{equation*}
corresponding to maps of bimodules
\begin{equation*}
B_{n}^{\Gamma} \otimes_{B_{n-1}^{\Gamma}} B_{n}^{\Gamma} \rightarrow B_{n}^{\Gamma} \otimes_{B_{n-1}^{\Gamma}} B_{n}^{\Gamma} \rightarrow B_{n}^{\Gamma}.
\end{equation*}
The left hand side acts on an element $ g \otimes h $ as follows:
\begin{equation*}
g \otimes h \mapsto (-1)^{|b||g|} gb \otimes h \mapsto (-1)^{|b||g|} gbh.
\end{equation*}
The right hand side acts on an element $ g \otimes h $ as follows:
\begin{equation*}
g \otimes h \mapsto (-1)^{|b||g|} g \otimes bh \mapsto (-1)^{|b||g|} gbh
\end{equation*}
where the factor $ (-1)^{|b||g|} $ is a consequence of how a tensor products of homomorphisms acts on supermodules.
\end{proof}

\subsubsection{Equation ~\ref{H4}}
\begin{prop}
\label{repH4}
If $ b \in B^{\Gamma} $, then there are equalities of natural transformations
\begin{enumerate}
\item $ (Y(b) \Id) \circ (\adj_{\Id}^{QP}) = (\Id X(b)) \circ (\adj_{\Id}^{QP}) $
\item $ (X(b) \Id) \circ (\adj_{\Id}^{PQ}) = (\Id Y(b)) \circ (\adj_{\Id}^{PQ}) $.
\end{enumerate}
\end{prop}

\begin{proof}
This follows from Propositions ~\ref{repH3} and ~\ref{repisotopy1}.
\end{proof}

\subsubsection{Equation ~\ref{H5}}
\begin{prop}
\label{repH5a}
There is an equality of natural transformations
\begin{equation*}
X(b'b) = (-1)^{|bb'|} X(b) \circ X(b') \colon P(n) \rightarrow P(n).
\end{equation*}
\end{prop}

\begin{proof}
Let $ x $ be an element of the $ (B_{n+1}^{\Gamma}, B_{n}^{\Gamma})$-bimodule $ B_{n+1}^{\Gamma} $.
Then $ X(b'b) \colon x \mapsto (-1)^{|x||b'b|} xb'b $.
On the other hand,
\begin{equation*}
X(b) \circ X(b') \colon x \mapsto (-1)^{|x||b'|} xb' \mapsto (-1)^{|x||b'|} (-1)^{|xb'||b|} xb'b.
\end{equation*}
Therefore the natural transformations are equal.
\end{proof}

\begin{prop}
\label{repH5b}
There is an equality of natural transformations
\begin{equation*}
Y(b'b) =
Y(b') \circ Y(b) \colon Q(n) \rightarrow Q(n).
\end{equation*}
\end{prop}

\begin{proof}
Let $ x $ be an element of the $ (B_{n}^{\Gamma}, B_{n+1}^{\Gamma})$-bimodule $ B_{n+1}^{\Gamma} $.
Then $ Y(b'b) \colon x \mapsto b'b x $.
On the other hand,
\begin{equation*}
Y(b') \circ Y(b) \colon x \mapsto
bx \mapsto
b'b x.
\end{equation*}
Therefore the natural transformations are equal.
\end{proof}

\subsubsection{Equation ~\ref{H6}}
Consider the functor $ P(n+r) \cdots P(n) $.
Let
\begin{equation*}
X_k(b) = \underbrace{\Id \cdots \Id}_{r-k+1} X(b) \underbrace{\Id \cdots \Id}_{k-1}
\end{equation*}

\begin{prop}
\label{repH6}
If $ k \neq l $, then there is an equality:
\begin{equation*}
X_l(b') X_k(b) =  (-1)^{|b||b'|} X_k(b) X_l(b') \colon P(n+r) \cdots P(n) \rightarrow P(n+r) \cdots P(n).
\end{equation*}
\end{prop}

\begin{proof}
Without loss of generality, assume $ k > l $.
The functor $ P(n+r) \cdots P(n) \rightarrow P(n+r) \cdots P(n) $ is given by tensoring with the
$ (B_{n+r+1}^{\Gamma}, B_{n}^{\Gamma})$-bimodule $ B_{n+r+1}^{\Gamma} $.
Let $ x \in B_{n+r+1}^{\Gamma} $.
Then $ X_k(b) $ maps $ x $ to
\begin{equation}
\label{aaa}
(-1)^{|x||b|} [x][(1\otimes \cdots \otimes 1\otimes b \otimes 1 \otimes \cdots \otimes 1)]
\end{equation}
where $ b $ is in component $ (n+k) $ in the $ (n+r+1)$-fold tensor product.
Then $ X_l(b') $ maps the element in ~\eqref{aaa} to:
\begin{equation}
\label{aa}
(-1)^{|x||b|}(-1)^{|b'||xb|} [x][(1 \otimes \cdots \otimes 1 \otimes b \otimes 1 \otimes \cdots \otimes 1)][(1 \otimes \cdots \otimes 1 \otimes b' \otimes 1 \otimes \cdots \otimes 1)]
\end{equation}
where $ b' $ is in component $ (n+l) $ in the $ (n+r+1)$-fold tensor product.

Similarly, $ X_k(b) X_l(b') $ maps $ x $ to
\begin{equation}
\label{ab}
(-1)^{|x||b'|}(-1)^{|b||xb'|} [x][(1 \otimes \cdots \otimes 1 \otimes b' \otimes 1 \otimes \cdots \otimes 1)][(1 \otimes \cdots \otimes 1 \otimes b \otimes 1 \otimes \cdots \otimes 1)]
\end{equation}
where again $ b' $ is in position $ n+l $ and $ b $ is in position $ n+k $.
Note that the powers of $-1$ in ~\eqref{aa} and ~\eqref{ab} are the same.  The last two factors in ~\eqref{ab} commute up to multiplication
by $ (-1)^{|b||b'|} $ which gives the proposition.
\end{proof}

\subsubsection{Equation ~\ref{H7}}
\begin{prop}
There is an equality of natural transformations:
\begin{equation*}
(T \Id) \circ (\Id T) \circ (T \Id) = (\Id T) \circ (T \Id) \circ (\Id T) \colon P(n+2) P(n+1) P(n) \rightarrow P(n+2) P(n+1) P(n).
\end{equation*}
\end{prop}

\begin{proof}
The functor $ P(n+2) P(n+1) P(n) $ is given by tensoring with the $ (B_{n+3}^{\Gamma}, B_{n}^{\Gamma})$-bimodule $ B_{n+3}^{\Gamma} $.
For $ x \in B_{n+3}^{\Gamma} $,
\begin{equation*}
(T \Id) \circ (\Id T) \circ (T \Id) \colon x \mapsto x s_{n+2} \mapsto x s_{n+2} s_{n+1} \mapsto x s_{n+2} s_{n+1} s_{n+2}
\end{equation*}
\begin{equation*}
(\Id T) \circ (T \Id) \circ (\Id T) \colon x \mapsto x s_{n+1} \mapsto x s_{n+1} s_{n+2} \mapsto x s_{n+1} s_{n+2} s_{n+1}.
\end{equation*}
The proposition follows since $ s_{n+2} s_{n+1} s_{n+2} = s_{n+1} s_{n+2} s_{n+1} $ in the Hecke-Clifford algebra.
\end{proof}

\subsubsection{Equation ~\ref{H8}}
\begin{prop}
There is an equality of natural transformations:
\begin{equation*}
T \circ T = \Id \Id \colon P(n+1) P(n) \rightarrow P(n+1) P(n).
\end{equation*}
\end{prop}

\begin{proof}
The functor $ P(n+1) P(n) $ is given by the bimodule $ B_{n+2}^{\Gamma} $.
The natural transformation $ T \circ T $ maps an element $ x $ to $ x s_{n+1} $ and then to $ x s_{n+1} s_{n+1} = x $.
Thus this is the identity map.
\end{proof}

\subsubsection{Equation ~\ref{H9}}
\begin{prop}
There is an equality of natural transformations
\begin{equation*}
T_{-+} \circ T_{+-} = (\Id \Id) \colon P(n) Q(n) \rightarrow P(n) Q(n).
\end{equation*}
\end{prop}

\begin{proof}
By the definition of $ T_{+-} $ and $ T_{-+} $, an element $ g \otimes h $ gets mapped to $ g s_n h $ which gets mapped to $ g \otimes h $ so the composite map is the identity.
\end{proof}

\subsubsection{Equation ~\ref{H10}}
\begin{prop}
There is an equality of natural transformations from $ Q(n) P(n) \rightarrow Q(n) P(n) $:
\begin{align*}
T_{+-} \circ T_{-+} = (\Id \Id) &-  \sum_{b \in \mathcal{B}} (Y(\check{b}) \Id) \circ (\adj_{\Id}^{QP}) \circ (\adj_{QP}^{\Id}) \circ (Y({b}) \Id) \\
&+  \sum_{b \in \mathcal{B}}  (Y(c_{n+1})\Id) \circ (Y(\check{b}) \Id) \circ (\adj_{\Id}^{QP}) \circ (\adj_{QP}^{\Id}) \circ (Y({b}) \Id) \circ (\Id X(c_{n+1}))\\
\end{align*}
\end{prop}

\begin{proof}
The left hand side is a map of functors
\begin{equation*}
Q(n) P(n) \rightarrow P(n-1) Q(n-1) \rightarrow Q(n) P(n)
\end{equation*}
corresponding to a map of $ (B_n^{\Gamma}, B_n^{\Gamma})$-bimodules
\begin{equation*}
B_{n+1}^{\Gamma} \otimes_{B_{n+1}^{\Gamma}} B_{n+1}^{\Gamma} \rightarrow
B_{n}^{\Gamma} \otimes_{B_{n-1}^{\Gamma}} B_{n}^{\Gamma}   \rightarrow
B_{n+1}^{\Gamma} \otimes_{B_{n+1}^{\Gamma}} B_{n+1}^{\Gamma}.
\end{equation*}

Let $ x \otimes 1 \in B_{n+1}^{\Gamma} \otimes_{B_{n+1}^{\Gamma}} B_{n+1}^{\Gamma} $ and let
$ x = (x_1 \otimes \cdots \otimes x_{n+1})wc $ where $ w \in S_{n+1} $ and $ c \in Cl_{n+1} $.
There are three cases to consider.

Case 1: $ w \notin S_n $.  Then let $ x = g s_n h $ where $ g, h \in B_{n}^{\Gamma} $.  Then by definition
$ T_{+-} \circ T_{-+} (g s_n h) = g s_n h $.
Due to the presence of $ s_n $ in $ g s_n h $, the second and third terms in the right hand side map $ g s_n h $ to zero.  Since the first term is the identity, we have equality in this case.

Case 2: $ w \in S_n $ and $ c \in Cl_n $. Then $ T_{-+} $ maps $ x \otimes 1 $ to zero.  The third term in the right hand side will also maps this element to zero because a factor of $ c_{n+1} $ will be introduced.
The second term maps $ x \otimes 1 $ to
\begin{align*}
& - \sum_b (Y(\check{b})\Id) \circ (\adj_{\Id}^{QP}) \circ (\adj_{QP}^{\Id})(b(x_1 \otimes \cdots \otimes x_n \otimes x_{n+1})wc \otimes 1) \\
= & - \sum_b (-1)^{|b||x_1 \cdots x_n|} (Y(\check{b})\Id) \circ (\adj_{\Id}^{QP}) \circ (\adj_{QP}^{\Id})((x_1 \otimes \cdots \otimes x_n \otimes bx_{n+1})wc \otimes 1) \\
= & - \sum_b (-1)^{|b||x_1 \cdots x_n|} \tr(b x_{n+1}) (Y(\check{b})\Id) ((x_1 \otimes \cdots \otimes x_n \otimes 1)wc \otimes 1) \\
= & - \sum_b (-1)^{|b||x_1 \cdots x_n|} (-1)^{|\check{b}||x_1 \cdots x_n|} \tr(b x_{n+1}) ((x_1 \otimes \cdots \otimes x_n \otimes \check{b})wc \otimes 1) \\
= & -x \otimes 1.
\end{align*}
Since the first term on the right hand side is the identity, the right hand side applied to $ x \otimes 1 $ is zero as well.

Case 3: $ w \in S_n $ and $ c \notin Cl_{n} $.  This is similar to case two.  The left hand side applied to $ x \otimes 1 $ is zero again.  Now the second map in the right hand side maps $ x \otimes 1 $ to zero
while the third map sends $ x \otimes 1 $ to $ -x \otimes 1 $.
\end{proof}

\subsubsection{Equation ~\ref{H11}}
\begin{prop}
There is an equality of natural transformations:
\begin{equation*}
(\adj_{QP}^{\Id}) \circ (\Id X(b)) \circ (\adj_{\Id}^{QP}) \colon \Id \rightarrow \Id = \text{tr}(b).
\end{equation*}
\end{prop}

\begin{proof}
This follows directly from the definition of $ \adj_{QP}^{\Id} $.
\end{proof}

\subsubsection{Equation ~\ref{H12}}
\begin{prop}
The natural transformation
$  (\adj_{QP}^{\Id} \Id) \circ (\Id T) \circ (\adj_{\Id}^{QP} \Id) $ of $ P(n) $ is zero.
\end{prop}

\begin{proof}
This natural transformation is a map of bimodules
\begin{equation*}
B_{n+1}^{\Gamma} \rightarrow B_{n+2}^{\Gamma} \otimes_{B_{n+2}^{\Gamma}} B_{n+2}^{\Gamma} \otimes_{B_{n+1}^{\Gamma}} B_{n+1}^{\Gamma} \rightarrow
B_{n+2}^{\Gamma} \otimes_{B_{n+2}^{\Gamma}} B_{n+2}^{\Gamma} \otimes_{B_{n+1}^{\Gamma}} B_{n+1}^{\Gamma}
\rightarrow B_{n+1}^{\Gamma}.
\end{equation*}
It maps the element $ x \in B_{n+1}^{\Gamma} $ as follows:
\begin{equation*}
x \mapsto 1 \otimes 1 \otimes x \mapsto 1 \otimes s_{n+1} \otimes x \mapsto 0
\end{equation*}
since $ \adj_{QP}^{\Id} $ maps $ s_{n+1} $ to zero.
\end{proof}

\subsubsection{Equation ~\ref{H13}}
\begin{prop}
\label{repH13}
There is an equality of natural transformations
\begin{equation*}
T \circ (X(c_{n+2})\Id) = (\Id X(c_{n+1})) \circ T \colon P(n+1) P(n) \rightarrow P(n+1) P(n).
\end{equation*}
\end{prop}

\begin{proof}
The functor $ P(n+1) P(n) $ is given by tensoring with the $ (B_{n+2}^{\Gamma}, B_{n}^{\Gamma})$-bimodule $ B_{n+2}^{\Gamma} $.
If $ x \in B_{n+2}^{\Gamma} $, then $ X(c_{n+2})\Id $ maps $ x $ to $ (-1)^{||x||}xc_{n+2} $.
Then $ T $ maps $ (-1)^{||x||} xc_{n+2} $ to $ (-1)^{||x||} xc_{n+2} s_{n+1} $.

For the right hand side of the equation, $ x $ maps to $ x s_{n+1} $ under $ T $.  Then under $ \Id X(c_{n+1}) $, $ x s_{n+1} $ maps to $ (-1)^{||x||} x s_{n+1} c_{n+1} $.
In the bimodule $ B_{n+2}^{\Gamma} $, $ x s_{n+1} c_{n+1} = x c_{n+2} s_{n+1} $ which verifies the proposition.
\end{proof}

\subsubsection{Equation ~\ref{H14}}
\begin{prop}
\label{repH14}
There is an equality of natural transformations
\begin{equation*}
(X(c_{n+2}) \Id) \circ T  = T \circ (\Id X(c_{n+1})) \colon P(n+1) P(n) \rightarrow P(n+1) P(n).
\end{equation*}
\end{prop}

\begin{proof}
The proof of this is very similar to that of Proposition ~\ref{repH13}.
\end{proof}

\subsubsection{Equation ~\ref{H15}}
\begin{prop}
\label{repH15}
As maps from $ P(n)Q(n) \rightarrow \Id $ and $ Q(n)P(n) \rightarrow \Id$ respectively, there are equalities:
\begin{enumerate}
\item $ (\adj_{PQ}^{\Id}) \circ (X(c_{n+1}) \Id) = (\adj_{PQ}^{\Id}) \circ (\Id Y(c_{n+1})) $
\item $ (\adj_{QP}^{\Id}) \circ (Y(c_{n+1}) \Id) = -(\adj_{QP}^{\Id}) \circ (\Id X(c_{n+1})) $.
\end{enumerate}
\end{prop}

\begin{proof}
The first item is an equality of natural transformations
\begin{equation*}
P(n) Q(n) \rightarrow P(n) Q(n) \rightarrow \Id
\end{equation*}
corresponding to a map of bimodules
\begin{equation*}
B_{n+1}^{\Gamma} \otimes_{B_{n}^{\Gamma}} B_{n+1}^{\Gamma} \rightarrow B_{n+1}^{\Gamma} \otimes_{B_{n}^{\Gamma}} B_{n+1}^{\Gamma} \rightarrow B_{n+1}^{\Gamma}.
\end{equation*}
Let $ g = (g_1 \otimes \cdots \otimes g_{n+1})c_g w_g $ and $ h= (h_1 \otimes \cdots \otimes h_{n+1}) c_h w_h $
where $ c_g, c_h \in Cl_{n+1} $ and $ w_g, w_h \in S_{n+1} $.

The first map on the left hand side ($X(c_{n+1}) \Id$)  takes the element $ g \otimes h $ to
\begin{equation}
\label{cliffcapslide}
(-1)^{||g||} [(g_1 \otimes \cdots \otimes g_{n+1}) c_g w_g c_{n+1}] \otimes [(h_1 \otimes \cdots \otimes h_{n+1})c_h w_h].
\end{equation}
Then $ \adj_{PQ}^{\Id} $ takes the element in \eqref{cliffcapslide} to
\begin{equation*}
(-1)^{||g||}[(g_1 \otimes \cdots \otimes g_{n+1})c_g w_g c_{n+1} c_h][(h_1 \otimes \cdots \otimes h_{n+1})w_h].
\end{equation*}

The first map on the right hand side ($\Id Y(c_{n+1})$) takes the element $ g \otimes h $ to
\begin{equation}
\label{cliffcapslide2}
(-1)^{||g||} [(g_1 \otimes \cdots \otimes g_{n+1})c_g w_g] \otimes [c_{n+1} (h_1 \otimes \cdots \otimes h_{n+1})c_h w_h]
\end{equation}
where now the factor $ (-1)^{||g||} $ comes from the sign convention of tensoring super homomorphisms.
The map $ \adj_{PQ}^{\Id} $ sends the element in \eqref{cliffcapslide2} to
\begin{equation*}
(-1)^{||g||}[(g_1 \otimes \cdots \otimes g_{n+1})c_g w_g c_{n+1} c_h][(h_1 \otimes \cdots \otimes h_{n+1})w_h]
\end{equation*}
which verifies the first equality.

Both sides of the second equality are natural transformations
\begin{equation*}
Q(n) P(n) \rightarrow Q(n) P(n) \rightarrow \Id
\end{equation*}
corresponding to maps of bimodules
\begin{equation*}
B_{n+1}^{\Gamma} \otimes_{B_{n+1}^{\Gamma}} B_{n+1}^{\Gamma} \rightarrow B_{n+1}^{\Gamma} \otimes_{B_{n+1}^{\Gamma}} B_{n+1}^{\Gamma} \rightarrow B_{n}^{\Gamma}.
\end{equation*}
Let $ g = (g_1 \otimes \cdots \otimes g_{n+1})wc $.  Then $ Y(c_{n+1}) \Id $ applied to $ g \otimes 1 $ is
\begin{equation*}
c_{n+1}(g_1 \otimes \cdots \otimes g_{n+1})wc = (g_1 \otimes \cdots \otimes g_{n+1}) c_{n+1} wc.
\end{equation*}

If $ c \in Cl_n $ or if $ w \notin S_n $, then this goes to zero under $ \adj_{QP}^{\Id} $.

If $ c \notin Cl_{n} $ and $ w \in S_n $, then this element gets mapped by $ \adj_{QP}^{\Id} $ to
\begin{equation*}
(-1)^{||g||-1} \tr(g_{n+1})(g_1 \otimes \cdots \otimes g_n \otimes 1) w c c_{n+1}
\end{equation*}
since $ c_{n+1} $ commutes past $ w $ and $ c_{n+1} c = (-1)^{||g||-1} c c_{n+1} $ since $ c $ does contain a factor of $ c_{n+1} $.

On the other hand, $ \Id X(c_{n+1}) $ applied to $ g \otimes 1 $ is
\begin{equation*}
(-1)^{||g||} \tr(g_{n+1})(g_1 \otimes \cdots \otimes g_{n+1})w c c_{n+1}
\end{equation*}
where the factor $ (-1)^{||g||} $ comes from the sign convention of tensor products of super module homomorphisms.
\end{proof}

\subsubsection{Equation ~\ref{H16}}
\begin{prop}
\label{repH16}
As maps from $ \Id \rightarrow Q(n) P(n) $ and $ \Id \rightarrow P(n)Q(n) $ respectively, there are equalites:
\begin{enumerate}
\item $ (Y(c_{n+1}) \Id) \circ (\adj_{\Id}^{QP}) = (\Id X(c_{n+1})) \circ (\adj_{\Id}^{QP}) $
\item $ (X(c_{n+1}) \Id) \circ (\adj_{\Id}^{PQ}) = -(\Id Y(c_{n+1})) \circ (\adj_{\Id}^{PQ}) $.
\end{enumerate}
\end{prop}

\begin{proof}
This follows from Propositions ~\ref{repH15} and ~\ref{repisotopy1}.
\end{proof}



\subsubsection{Equation ~\ref{H17}}
\begin{prop}
\begin{enumerate}
\item The map $ X(c_{n+1}) \circ X(c_{n+1}) \colon P(n) \rightarrow P(n) $ is $ -\Id $.
\item The map $ Y(c_{n+1}) \circ Y(c_{n+1}) \colon Q(n) \rightarrow Q(n) $ is $ \Id $.
\end{enumerate}
\end{prop}

\begin{proof}
These follow easily from the fact that $ c_{n+1}^2 = 1 $ in the Hecke-Clifford algebra.
\end{proof}

\subsubsection{Equation ~\ref{H18}}
\begin{prop}
Let $ b \in B_{}^{\Gamma} $.  Then there is an equality of natural transformations
$ X(b) X(c_{n+1}) = X(c_{n+1}) X(b) \colon P(n) \rightarrow P(n) $.
\end{prop}

\begin{proof}
This follows since the elements $ c_{n+1} $ and $ b $ commute in the algebra $ B_{n+1}^{\Gamma} $.
\end{proof}

\subsubsection{Equation ~\ref{H19}}
Consider the functor $ P(n+r) \cdots P(n) $.
Let
\begin{equation*}
X_k(c_{n+k}) = \underbrace{\Id \cdots \Id}_{r-k+1} X(c_{n+k}) \underbrace{\Id \cdots \Id}_{k-1}
\end{equation*}

\begin{prop}
If $ k \neq l $, then there is an equality:
\begin{equation*}
X_l(c_{n+l}) X_k(c_{n+k}) =  -X_k(c_{n+k}) X_l(c_{n+l}) \colon P(n+r) \cdots P(n) \rightarrow P(n+r) \cdots P(n).
\end{equation*}
\end{prop}

\begin{proof}
This is nearly identical to the proof of Proposition ~\ref{repH6}.
\end{proof}

\subsubsection{Equation ~\ref{H20}}
\begin{prop}
There is an equality of natural transformations:
\begin{equation*}
(\adj_{QP}^{\Id}) \circ (\Id X(c_{n+1}) \circ (\Id X(b)) \circ (\adj_{\Id}^{QP}) \colon \Id \rightarrow \Id = 0
\end{equation*}
\end{prop}

\begin{proof}
This follows directly from the definition of $ \adj_{QP}^{\Id} $ since $ b c_{n+1} $ gets mapped to zero.
\end{proof}

\subsubsection{Isotopies in the plane}
\begin{equation}
\label{isotopy1}
\begin{tikzpicture}
\draw[thick] (0,0) to (0,1) arc(180:0:.5) arc(180:360:.5) to (2,2)[->];
\draw (2.5,1) node {$=$};
\draw [thick] (3,0) -- (3,2)[->];
\draw (3.5,1) node {$=$};
\draw[thick] (4,2)[<-] to (4,1) arc(180:360:.5) arc(180:0:.5) to (6,0);

\draw[thick] (8,0) to (8,1)[<-] arc(180:0:.5) arc(180:360:.5) to (10,2);
\draw (10.5,1) node {$=$};
\draw [thick] (11,0) -- (11,2)[<-];
\draw (11.5,1) node {$=$};
\draw[thick] (12,2) to (12,1) arc(180:360:.5) arc(180:0:.5) to (14,0)[->];
\end{tikzpicture}
\end{equation}

\begin{prop}
\label{repisotopy1}
There are equalities of natural transformations corresponding to the diagrams in ~\eqref{isotopy1}
\begin{enumerate}
\item $ (\adj_{PQ}^{\Id} \Id) \circ (\Id \adj_{\Id}^{QP}) = \Id = (\Id \adj_{QP}^{\Id}) \circ (\adj_{\Id}^{PQ} \Id) $
\item $ (\adj_{QP}^{\Id} \Id) \circ (\Id \adj_{\Id}^{PQ}) = \Id = (\Id \adj_{PQ}^{\Id}) \circ (\adj_{\Id}^{QP} \Id) $.
\end{enumerate}
\end{prop}

\begin{proof}
Since the proofs of both sets of equalities are similar, we only prove the first set.

The map $ (\adj_{PQ}^{\Id} \Id) \circ (\Id \adj_{\Id}^{QP}) $ is a composite of natural transformations
\begin{equation*}
P(n) \rightarrow P(n)Q(n)P(n) \rightarrow P(n)
\end{equation*}
which corresponds to a composite of bimodule maps
\begin{equation*}
B_{n+1}^{\Gamma} \rightarrow B_{n+1}^{\Gamma} \otimes_{B_{n}^{\Gamma}} B_{n+1}^{\Gamma} \otimes_{B_{n+1}^{\Gamma}} B_{n+1}^{\Gamma} \rightarrow B_{n+1}^{\Gamma}.
\end{equation*}
Let $ g \in B_{n+1}^{\Gamma} $.  Then $ \Id \adj_{\Id}^{QP} $ maps $ g $ to $ g \otimes 1 \otimes 1 $.  Then $ \adj_{PQ}^{\Id} \Id $ maps $ g \otimes 1 \otimes1 $ to $ g $ which verifies the first equality.

The map $ (\Id \adj_{QP}^{\Id}) \circ (\adj_{\Id}^{PQ} \Id) $ is a composite of natural transformations
\begin{equation*}
P(n) \rightarrow P(n)Q(n)P(n) \rightarrow P(n)
\end{equation*}
which corresponds to a composite of bimodule maps
\begin{equation*}
B_{n+1}^{\Gamma} \rightarrow B_{n+1}^{\Gamma} \otimes_{B_{n}^{\Gamma}} B_{n+1}^{\Gamma} \otimes_{B_{n+1}^{\Gamma}} B_{n+1}^{\Gamma} \rightarrow B_{n+1}^{\Gamma}.
\end{equation*}

Let $ g = (g_1 \otimes \cdots \otimes g_{n+1})(cw) $ where $ c \in Cl_{n+1} $ and $ w \in S_{n+1} $.
Under the map $ \adj_{\Id}^{PQ} \Id $, the element $ g $ gets mapped to
\begin{equation}
X = \sum_{b \in \mathcal{B}} \sum_{i=1}^{n+1} (s_i \cdots s_n \check{b}_{n+1} \otimes b_{n+1} s_n \cdots s_i \otimes g + s_i \cdots s_n c_{n+1} \check{b}_{n+1} \otimes b_{n+1} c_{n+1} s_n \cdots s_i \otimes g).
\end{equation}
We now consider four cases for $ c $ and $ w $.

Case 1: $ c \in Cl_n $ and $ w \in S_n $.  In this case, the terms
$  s_i \cdots s_n c_{n+1} \check{b}_{n+1} \otimes b_{n+1} c_{n+1} s_n \cdots s_i \otimes g $ get mapped to zero under $ \Id \adj_{QP}^{\Id} $ due to the presence of $ c_{n+1} $ for $ i = n+1 $ and due to the presence of
$ s_n $ for $ i \neq n+1 $.
Similarly, the terms
$ s_i \cdots s_n \check{b}_{n+1} \otimes b_{n+1} s_n \cdots s_i \otimes g $ for $ i \neq n+1 $ get mapped to zero by $ \Id \adj_{QP}^{\Id} $ due to the presence of $ s_n $.
Thus
\begin{align*}
\Id \adj_{QP}^{\Id}(X) &= \Id \adj_{QP}^{\Id}(\sum_b \check{b}_{n+1} \otimes b_{n+1}(g_1 \otimes \cdots \otimes g_{n+1})cw)\\
&= \Id \adj_{QP}^{\Id}(\sum_b (-1)^{|b_{n+1}||g_1 \cdots g_n|}  \check{b}_{n+1} \otimes (g_1 \otimes \cdots \otimes b_{n+1} g_{n+1})cw).
\end{align*}
Let $ g_{n+1} = \gamma^{-1} \check{v} $ where $ \gamma \in \Gamma $ and $ v \in \Lambda^2 V$.  Then all of the terms above are zero unless $ b_{n+1} = v \gamma $ in which case it is equal to
\begin{equation*}
(-1)^{|b_{n+1}||g_1 \cdots g_n|}  \check{b}_{n+1} \otimes (g_1 \otimes \cdots \otimes g_n \otimes 1)cw =
(-1)^{|b_{n+1}||g_1 \cdots g_n|+|\check{b}_{n+1}||g_1 \cdots g_n|} (g_1 \otimes \cdots \otimes g_n \otimes \check{b}_{n+1})cw.
\end{equation*}
Since $ \check{b}_{n+1} = g_{n+1} $ and $ |b_{n+1}|+|\check{b}_{n+1}|=2$, we get
$ \Id \adj_{QP}^{\Id}(X) = g $.

Case 2: $ c = c_{n+1} c' $ where $ c' \in Cl_n $ and $ w \in S_n $.
The terms $ s_i \cdots s_n \check{b}_{n+1} \otimes b_{n+1} s_n \cdots s_i \otimes g $ in $ X $ go to zero under $ \Id \adj_{QP}^{\Id} $ due to the presence of $ c_{n+1} $.
Due to the presence of $ s_n $, the terms
$ s_i \cdots s_n c_{n+1} \check{b}_{n+1} \otimes b_{n+1} c_{n+1}  s_n \cdots s_i \otimes g $ go to zero under $ \Id \adj_{QP}^{\Id} $ unless $ i = n+1 $ so instead of considering the image of $ X $ under $ \Id \adj_{QP}^{\Id} $,
it suffices to consider the image of the element
\begin{equation*}
\sum_b c_{n+1} \check{b}_{n+1} \otimes b_{n+1} c_{n+1} \otimes (g_1 \otimes \cdots \otimes g_{n+1}) c_{n+1} c' w =
\sum_b c_{n+1} \check{b}_{n+1} \otimes b_{n+1}  \otimes (g_1 \otimes \cdots \otimes g_{n+1}) c' w.
\end{equation*}
Now the arguments from case 1 show that this element goes to $ g $ under $ \Id \adj_{QP}^{\Id} $.

Case 3: $ c \in Cl_n $ and $ w \in S_{n+1} $ but $ w \notin S_n  $.
We may write $ w = w^p w_p $ where $ w_p \in S_n $ and $ w^p $ is a shortest length left coset representative in $ S_{n+1} / S_n $.
Then $ w^p \in \lbrace e, s_n, s_{n-1} s_n, \ldots, s_1 \cdots s_n \rbrace $.
Let $ g = (g_1 \otimes \cdots \otimes g_n \otimes g_{n+1}) (wc) $.
Suppose $ w^p = s_i \cdots s_n $.  Then under $ \Id \adj_{QP}^{\Id} $, all terms of $ X $ will vanish except
\begin{align*}
& \sum_b (s_i \cdots s_n \check{b}_{n+1} \otimes b_{n+1} s_n \cdots s_i) \otimes (g_1 \otimes \cdots \otimes g_{n+1})(s_i \cdots s_n w_p c) \\
= & \sum_b (s_i \cdots s_n \check{b}_{n+1} \otimes b_{n+1}) \otimes (w^p)^{-1}.(g_1 \otimes \cdots \otimes g_{n+1}) w_p c\\
= & \sum_b (s_i \cdots s_n \check{b}_{n+1} \otimes b_{n+1}) \otimes (g_1 \otimes \cdots \otimes g_{i-1} \otimes g_{i+1} \otimes \cdots \otimes g_{n+1} \otimes g_i) w_p c\\
= & \sum_b (-1)^{|b_{n+1}||g_1 \cdots g_{i-1} g_{i+1} \cdots g_{n+1}|}
(s_i \cdots s_n \check{b}_{n+1}) \otimes (g_1 \otimes \cdots \otimes g_{i-1} \otimes g_{i+1} \otimes \cdots \otimes g_{n+1} \otimes b_{n+1} g_i) w_p c\\
\end{align*}
Let $ g_i = \gamma^{-1} \check{v} $ for $ \gamma \in \Gamma $.  Then under $ \Id \adj_{QP}^{\Id} $, the terms go to zero unless
$ b_{n+1} = v \gamma $ in which case the term goes to
\begin{align*}
& (-1)^{|v||g_1 \cdots g_{i-1} g_{i+1} \cdots g_{n+1}|}
(s_i \cdots s_n \check{(v \gamma)}) \otimes (g_1 \otimes \cdots \otimes g_{i-1} \otimes g_{i+1} \otimes \cdots \otimes g_{n+1} \otimes 1) w_p c\\
= & (s_i \cdots s_n) \otimes (g_1 \otimes \cdots \otimes g_{i-1} \otimes g_{i+1} \otimes \cdots \otimes g_{n+1} \otimes \check{(v \gamma)}) w_p c\\
= & g
\end{align*}
where the last equality follows after commuting $ s_i \cdots s_n $ until it is next to $ w_p c $.

Case 4: $ c = c_{n+1} c' $ where $ c' \in Cl_n $ and $ w \in S_{n+1} $ but $ w \notin S_n $.
Recall the elements $ w_p $ and $ w^p $ from case 3.
Once again, most of the terms in $ X $ get mapped to zero by $ \Id \adj_{QP}^{\Id} $ due to the presence of $ s_n $ so instead of $ X $, it suffices to consider the element
\begin{align*}
& \sum_b s_i \cdots s_n \check{b}_{n+1} \otimes b_{n+1}(g_1 \otimes \cdots \otimes g_{i-1} \otimes g_{i+1} \otimes \cdots \otimes g_{n+1} \otimes g_i) w_p c_{n+1} c' + \\
& \sum_b s_i \cdots s_n c_{n+1} \check{b}_{n+1} \otimes b_{n+1} c_{n+1}(g_1 \otimes \cdots \otimes g_{i-1} \otimes g_{i+1} \otimes \cdots \otimes g_{n+1} \otimes g_i) w_p c_{n+1} c' .
\end{align*}
Due to the presence of $ c_{n+1} $ in the first term, it goes to zero under $ \Id \adj_{QP}^{\Id} $.  The second summation goes to $ g $ as in case 3.
\end{proof}

 \begin{equation}
 \label{isotopy4}
\begin{tikzpicture}
   \draw[->,thick] (0.5,1) .. controls (0.5,0) and (1.5,0) .. (1.5,1);
  \draw[<-,thick] (1,1) to (0.5,0);
  \draw (2.25,0.5) node {=};
  \draw[->,thick] (3,1) .. controls (3,0) and (4,0) .. (4,1);
  \draw[<-,thick] (3.5,1) to (4,0);
   \draw[->,thick] (6.5,1) .. controls (6.5,0) and (7.5,0) .. (7.5,1);
  \draw[->,thick] (7,1) to (6.5,0);
  \draw (8.25,0.5) node {=};
  \draw[->,thick] (9,1) .. controls (9,0) and (10,0) .. (10,1);
  \draw[->,thick] (9.5,1) to (10,0);
\end{tikzpicture}
\end{equation}

\begin{prop}
\label{repisotopy4}
There are equalities of natural transformations corresponding to the diagrams in ~\eqref{isotopy4}
\begin{enumerate}
\item $ (T_{+-} \Id) \circ (\Id \adj_{\Id}^{QP}) = (\Id T) \circ (\adj_{\Id}^{QP} \Id) $
\item $ (T_{--} \Id) \circ (\Id \adj_{\Id}^{QP}) = (\Id T_{+-}) \circ (\adj_{\Id}^{QP} \Id) $.
\end{enumerate}
\end{prop}

\begin{proof}
The left hand side of the first equality is a composition of bimodule homomorphisms:
\begin{equation*}
B_{n+1}^{\Gamma} \rightarrow B_{n+1}^{\Gamma} \otimes_{B_{n}^{\Gamma}} B_{n+1}^{\Gamma} \otimes_{B_{n+1}^{\Gamma}} B_{n+1}^{\Gamma} \rightarrow
B_{n+2}^{\Gamma} \otimes_{B_{n+2}^{\Gamma}} B_{n+2}^{\Gamma} \otimes_{B_{n+1}^{\Gamma}} B_{n+1}^{\Gamma}.
\end{equation*}
This composition maps an element $ g $ as follows:
\begin{equation*}
g \mapsto g \otimes 1 \otimes 1 \mapsto g s_n \otimes 1 \otimes 1.
\end{equation*}
The right hand side of the first equality is a composition of maps:
\begin{equation*}
B_{n+1}^{\Gamma} \rightarrow B_{n+2}^{\Gamma} \otimes_{B_{n+2}^{\Gamma}} B_{n+2}^{\Gamma} \otimes_{B_{n+1}^{\Gamma}} B_{n+1}^{\Gamma} \rightarrow
B_{n+2}^{\Gamma} \otimes_{B_{n+2}^{\Gamma}} B_{n+2}^{\Gamma} \otimes_{B_{n+1}^{\Gamma}} B_{n+1}^{\Gamma}.
\end{equation*}
This composition of morphisms maps an element $ g $ as follows:
\begin{equation*}
g \mapsto 1 \otimes 1 \otimes g \mapsto 1 \otimes 1 \otimes g s_n = g s_n \otimes 1 \otimes 1.
\end{equation*}
This proves the first equality.

The left hand side of the second equality is a composition of bimodule homomorphisms:
\begin{equation*}
B_{n+1}^{\Gamma} \rightarrow B_{n+1}^{\Gamma} \otimes_{B_{n+1}^{\Gamma}} B_{n+2}^{\Gamma} \otimes_{B_{n+2}^{\Gamma}} B_{n+2}^{\Gamma} \rightarrow
B_{n+1}^{\Gamma} \otimes_{B_{n+1}^{\Gamma}} B_{n+2}^{\Gamma} \otimes_{B_{n+2}^{\Gamma}} B_{n+2}^{\Gamma}.
\end{equation*}
This composition maps an element $ g $ as follows:
\begin{equation*}
g \mapsto g \otimes 1 \otimes 1 = 1 \otimes g \otimes 1 \mapsto 1 \otimes s_{n+1}g \otimes 1 = 1 \otimes 1 \otimes s_{n+1}g.
\end{equation*}
The right hand side of the second equality is a composition of bimodule homomorphisms:
\begin{equation*}
B_{n+1}^{\Gamma} \rightarrow B_{n+1}^{\Gamma} \otimes_{B_{n+1}^{\Gamma}} B_{n+1}^{\Gamma} \otimes_{B_{n}^{\Gamma}} B_{n+1}^{\Gamma} \rightarrow
B_{n+1}^{\Gamma} \otimes_{B_{n+1}^{\Gamma}} B_{n+2}^{\Gamma} \otimes_{B_{n+2}^{\Gamma}} B_{n+2}^{\Gamma}.
\end{equation*}
This composition maps an element $ g $ as follows:
\begin{equation*}
g \mapsto 1 \otimes 1 \otimes g \mapsto 1 \otimes 1 \otimes s_{n+1}g.
\end{equation*}
This verifies the second equality.
\end{proof}

\begin{equation}
\label{isotopy5}
\begin{tikzpicture}
  \draw[<-,thick] (0.5,0) .. controls (0.5,1) and (1.5,1) .. (1.5,0);
  \draw[<-,thick] (1.5,1) to (1,0);
  \draw (2.25,0.5) node {=};
  \draw[<-,thick] (3,0) .. controls (3,1) and (4,1) .. (4,0);
  \draw[<-,thick] (3,1) to (3.5,0);
  \draw[<-,thick] (6.5,0) .. controls (6.5,1) and (7.5,1) .. (7.5,0);
  \draw[->,thick] (7.5,1) to (7,0);
  \draw (8.25,0.5) node {=};
  \draw[<-,thick] (9,0) .. controls (9,1) and (10,1) .. (10,0);
  \draw[->,thick] (9,1) to (9.5,0);
\end{tikzpicture}
 \end{equation}

\begin{prop}
\label{repisotopy5}
There are equalities of natural transformations corresponding to the diagrams in ~\eqref{isotopy5}
\begin{enumerate}
\item $ (\adj_{QP}^{\Id} \Id) \circ (\Id T_{}) = (\Id \adj_{QP}^{\Id}) \circ (T_{-+} \Id) $
\item $ (\adj_{QP}^{\Id} \Id) \circ (\Id T_{-+}) = (\Id \adj_{QP}^{\Id}) \circ (T_{--} \Id) $.
\end{enumerate}
\end{prop}

\begin{proof}
We begin with the first equality.
The left hand side is a composite of natural transformations
\begin{equation*}
Q(n+1) P(n+1) P(n) \rightarrow Q(n+1) P(n+1) P(n) \rightarrow P(n)
\end{equation*}
corresponding to a composition of bimodules homomorphisms
\begin{equation*}
B_{n+2}^{\Gamma} \otimes_{B_{n+2}^{\Gamma}} B_{n+2}^{\Gamma} \otimes_{B_{n+1}^{\Gamma}} B_{n+1}^{\Gamma} \rightarrow
B_{n+2}^{\Gamma} \otimes_{B_{n+2}^{\Gamma}} B_{n+2}^{\Gamma} \otimes_{B_{n+1}^{\Gamma}} B_{n+1}^{\Gamma} \rightarrow
B_{n+1}^{\Gamma}.
\end{equation*}
The right hand side is a composite of natural transformations
\begin{equation*}
Q(n+1) P(n+1) P(n) \rightarrow P(n) Q(n) P(n) \rightarrow P(n)
\end{equation*}
corresponding to a composition of bimodules homomorphisms
\begin{equation*}
B_{n+2}^{\Gamma} \otimes_{B_{n+2}^{\Gamma}} B_{n+2}^{\Gamma} \otimes_{B_{n+1}^{\Gamma}} B_{n+1}^{\Gamma} \rightarrow
B_{n+1}^{\Gamma} \otimes_{B_{n}^{\Gamma}} B_{n+1}^{\Gamma} \otimes_{B_{n+1}^{\Gamma}} B_{n+1}^{\Gamma} \rightarrow
B_{n+1}^{\Gamma}.
\end{equation*}

Case 1: $ g \in B_{n+1}^{\Gamma} $.
Then $ T_{-+} \Id $ sends $ g \otimes 1 \otimes 1  $ to zero so the right hand side is zero.
For the left hand side we get,
\begin{equation*}
(\Id T)(g \otimes 1 \otimes 1) = (\Id T)(1 \otimes g \otimes 1) = 1 \otimes g s_{n+1} \otimes 1.
\end{equation*}
Now we must compute $ \adj_{QP}^{\Id} \Id(1 \otimes g s_{n+1} \otimes 1) $.
Since $ g $ is assumed to be in $ B_{n+1}^{\Gamma} $, the term $ s_{n+1} $ does not appear in $ g $ so $ s_{n+1} $ does appear
in $ g s_{n+1} $.  Thus $ \adj_{QP}^{\Id} \Id(1 \otimes g s_{n+1} \otimes 1) = 0 $.

Case 2: $ g \notin B_{n+1}^{\Gamma} $.  Then $ g = x s_{n+1} y $ where $ x, y \in B_{n+1}^{\Gamma} $.
Let $ x = (x_1 \otimes \cdots \otimes x_{n+1} \otimes 1)c_x w_x $ and $ y = (y_1 \otimes \cdots \otimes y_{n+1} \otimes 1) c_y w_y $
where $ c_x, c_y \in Cl_{n+1} $ and $ w_x, w_y \in S_{n+1} $.

Subcase 2a: $ g \notin B_{n+1}^{\Gamma} $, $ w_y \in S_n $, and $ c_y \in Cl_n $.
For the right hand side we compute,
\begin{equation*}
(T_{-+} \Id)(g \otimes 1 \otimes 1) = x \otimes y \otimes 1 = x \otimes (y_1 \otimes \cdots \otimes y_{n+1} \otimes 1)c_y w_y \otimes 1.
\end{equation*}
Since $ w_y \in S_n $ and $ c_y \in Cl_n $,
\begin{equation*}
(\Id \adj_{QP}^{\Id})(x \otimes (y_1 \otimes \cdots \otimes y_{n+1} \otimes 1)c_y w_y \otimes 1) = \tr(y_{n+1})(x(y_1 \otimes \cdots \otimes y_n \otimes 1 \otimes 1)c_y w_y).
\end{equation*}

For the left hand side we see that
\begin{align*}
(\adj_{QP}^{\Id}) \circ (\Id T)(1 \otimes g \otimes 1) &= (\adj_{QP}^{\Id} \Id)(1 \otimes g s_{n+1} \otimes 1) \\
&= \adj_{QP}^{\Id} \Id (1 \otimes x s_{n+1}(y_1 \otimes \cdots \otimes y_{n+1} \otimes 1)c_y w_y s_{n+1} \otimes 1).
\end{align*}
Since $ c_y \in Cl_n $ and $ w_y \in S_n $, we may bring $ s_{n+1} $ to the left of these two elements and then past $ (y_1 \otimes \cdots \otimes y_{n+1} \otimes 1) $ to get
that the above is equal to
\begin{equation*}
\adj_{QP}^{\Id} \Id(1 \otimes x(y_1 \otimes \cdots \otimes y_n \otimes 1 \otimes y_{n+1})c_y w_y \otimes 1) =
\tr(y_{n+1})x(y_1 \otimes \cdots \otimes y_n \otimes 1 \otimes 1)c_y w_y
\end{equation*}
which verifies the equality in this case.

Subcase 2b: $ g \notin B_{n+1}^{\Gamma} $, $ w_y \in S_n $, $ c_y \in Cl_{n+1} $, but $ c_y \notin Cl_n $.
For the left hand side we compute
\begin{align*}
(\adj_{QP}^{\Id} \Id) \circ (\Id T)(g \otimes 1 \otimes 1) &= (\adj_{QP}^{\Id} \Id)(1 \otimes gs_{n+1} \otimes 1) \\
&= \adj_{QP}^{\Id} \Id(1 \otimes x s_{n+1}(y_1 \otimes \cdots \otimes y_{n+1} \otimes 1)c_y w_y s_{n+1} \otimes 1)\\
&= \adj_{QP}^{\Id} \Id(1 \otimes x(y_1 \otimes \cdots \otimes 1 \otimes y_{n+1})(s_{n+1}.c_y) w_y \otimes 1).
\end{align*}
Since $ c_y \notin Cl_n $, the term $ c_{n+2} $ must occur in $ s_{n+1}.c_y $.
Since $ x \in B_{n+1}^{\Gamma} $,
$ c_{n+2} $ must occur in the expression
$ 1 \otimes x(y_1 \otimes \cdots \otimes 1 \otimes y_{n+1})(s_{n+1}.c_y) w_y \otimes 1 $
so $ \adj_{QP}^{\Id} \Id $ must send it to zero.

For the right hand side,
\begin{align*}
(\Id \adj_{QP}^{\Id}) \circ (T_{-+} \circ \Id)(1 \otimes g \otimes 1) &= \Id \adj_{QP}^{\Id}(x \otimes y \otimes 1)\\
&= \Id \adj_{QP}^{\Id}(x \otimes (y_1 \otimes \cdots \otimes y_{n+1} \otimes 1)c_y w_y \otimes 1) \\
&= 0
\end{align*}
since $ c_{n+1} $ must occur in
$ x \otimes (y_1 \otimes \cdots \otimes y_{n+1} \otimes 1)c_y w_y \otimes 1 $.

Subcase 2c: $ g \notin B_{n+1}^{\Gamma} $, $ w_y \notin S_n $.
For the left hand side we compute
\begin{align*}
(\adj_{QP}^{\Id} \Id) \circ (\Id T)(1 \otimes g \otimes 1) &= \adj_{QP}^{\Id} \Id(1 \otimes g s_{n+1} \otimes 1)\\
&= (\adj_{QP}^{\Id} \Id)(1 \otimes x s_{n+1} y s_{n+1} \otimes 1)\\
&=(\adj_{QP}^{\Id} \Id)(1 \otimes x s_{n+1}(y_1 \otimes \cdots \otimes y_{n+1} \otimes 1)c_y w_y s_{n+1} \otimes 1)\\
&= 0
\end{align*}
since the two occurrences of $ s_{n+1} $ do not cancel because $ w_y \notin S_n $.

For the right hand side,
\begin{align*}
(\Id \adj_{QP}^{\Id}) \circ (T_{-+} \Id)(g \otimes 1 \otimes 1) &= \Id \adj_{QP}^{\Id}(x \otimes y \otimes 1) \\
&= \Id \adj_{QP}^{\Id}(x \otimes (y_1 \otimes \cdots \otimes y_{n+1} \otimes 1)c_y w_y \otimes 1)\\
&= 0
\end{align*}
since $ s_{n+1} $ must occur in
$ x \otimes (y_1 \otimes \cdots \otimes y_{n+1} \otimes 1)c_y w_y \otimes 1 $
because of the assumption that $ w_y \notin S_n $.
This finishes the first equality.

Now for the second equality.   The left hand side is a composite of natural transformations
\begin{equation*}
Q(n-1) Q(n) P(n) \rightarrow Q(n-1) P(n-1) Q(n-1) \rightarrow Q(n-1)
\end{equation*}
corresponding to a composition of bimodule homomorphisms
\begin{equation*}
B_{n}^{\Gamma} \otimes_{B_{n}^{\Gamma}} B_{n+1}^{\Gamma} \otimes_{B_{n+1}^{\Gamma}} B_{n+1}^{\Gamma} \rightarrow
B_{n}^{\Gamma} \otimes_{B_{n}^{\Gamma}} B_{n}^{\Gamma} \otimes_{B_{n-1}^{\Gamma}} B_{n}^{\Gamma} \rightarrow
B_{n}^{\Gamma}.
\end{equation*}
The right hand side is a composite of natural transformations
\begin{equation*}
Q(n-1) Q(n) P(n) \rightarrow Q(n-1) Q(n) P(n) \rightarrow Q(n-1)
\end{equation*}
corresponding to a composition of bimodule homomorphisms
\begin{equation*}
B_{n}^{\Gamma} \otimes_{B_{n}^{\Gamma}} B_{n+1}^{\Gamma} \otimes_{B_{n+1}^{\Gamma}} B_{n+1}^{\Gamma} \rightarrow
B_{n}^{\Gamma} \otimes_{B_{n}^{\Gamma}} B_{n+1}^{\Gamma} \otimes_{B_{n+1}^{\Gamma}} B_{n+1}^{\Gamma} \rightarrow
B_{n}^{\Gamma}.
\end{equation*}
Let $ 1 \otimes 1 \otimes g \in B_{n}^{\Gamma} \otimes_{B_{n}^{\Gamma}} B_{n+1}^{\Gamma} \otimes_{B_{n+1}^{\Gamma}} B_{n+1}^{\Gamma} \rightarrow $ where
$ g = (g_1 \otimes \cdots \otimes g_n \otimes g_{n+1})cw $ where $ w \in S_{n+1} $ and $ c \in Cl_{n+1} $.

Case 1: $ g \in B_n^{\Gamma} $.  Then the left hand side applied to $ 1 \otimes 1 \otimes g $ is easily seen to be zero.
For the right hand side,
\begin{equation*}
T_{--} \Id (1 \otimes 1 \otimes g) = 1 \otimes s_n \otimes g = 1 \otimes 1 \otimes s_n g.
\end{equation*}
Since $ g \in B_n^{\Gamma} $, it is clear that $ s_n $ appears in $ s_n g $ so $ 1 \otimes 1 \otimes s_n g $ goes to zero under $ \Id \adj_{QP}^{\Id} $ which verifies the equality.

Case 2: $ g \notin B_n^{\Gamma} $.  Then $ g = x s_n y $ where $ x,y \in B_n^{\Gamma} $.
Let $ x = (x_1 \otimes \cdots \otimes x_n \otimes 1)c_x w_x $ and $ y =(y_1 \otimes \cdots \otimes y_n \otimes 1)c_y w_y $ where
$ c_x, c_y \in Cl_n $ and $ w_x, w_y \in S_n $.

Subcase 2a: $ g \notin B_n^{\Gamma} $, $ w_x \in S_{n-1}$, and $ c_x \in Cl_{n-1} $.
We consider the map on the left hand side.
First, $ \Id T_{-+}(1 \otimes 1 \otimes g) = 1 \otimes x \otimes y $.
Then
\begin{equation*}
\adj_{QP}^{\Id}(1 \otimes x \otimes y) = \tr(x_n)(x_1 \otimes \cdots \otimes x_{n-1} \otimes 1 \otimes 1)c_x w_x \otimes y
\end{equation*}
since the term $ s_{n-1} $ does not occur in $ w_x $ by assumption.

Now for the right hand side we have:
\begin{align*}
(T_{--} \Id)(1 \otimes 1 \otimes g) &= 1 \otimes s_n \otimes g\\
&=1 \otimes 1 \otimes s_n x s_n y\\
&=1 \otimes 1 \otimes s_n(x_1 \otimes \cdots \otimes x_n \otimes 1)c_x w_x s_n y\\
&=1 \otimes 1 \otimes (x_1 \otimes \cdots \otimes x_{n-1} \otimes 1 \otimes x_n)s_n c_x w_x s_n y\\
&=1 \otimes 1 \otimes (x_1 \otimes \cdots \otimes x_{n-1} \otimes 1 \otimes x_n)(s_n. c_x) w_x y\\
&=1 \otimes 1 \otimes (x_1 \otimes \cdots \otimes x_{n-1} \otimes 1 \otimes x_n) c_x w_x y
\end{align*}
since $ c_x \in Cl_{n-1} $.  Since $ y_{n+1} =1 $ by assumption, the element above gets mapped by $ \Id \adj_{QP}^{\Id} $ to
$ \tr(x_n)(x_1 \otimes \cdots \otimes x_{n-1} \otimes 1 \otimes 1)c_x w_x \otimes y  $
which verifies the equality.

Subcase 2b: $ g \notin B_n^{\Gamma} $, $ w_x \in S_{n-1} $, $ c_x \notin Cl_{n-1} $, but $ c_x \in Cl_n $.
For the left hand side,
\begin{equation*}
\adj_{QP}^{\Id} \Id(1 \otimes (x_1 \otimes \cdots \otimes x_n)c_x w_x \otimes y) = 0
\end{equation*}
since $ c_x \in Cl_n $ and $ c_x \notin Cl_{n-1} $.

For the right hand side,
\begin{equation*}
(\Id \adj_{QP}^{\Id}) \circ (T_{--} \Id)(1 \otimes 1 \otimes g) = (\Id \adj_{QP}^{\Id})(1 \otimes 1 \otimes (x_1 \otimes \cdots \otimes x_{n-1} \otimes 1 \otimes x_n)(s_n . c_x)w_x y).
\end{equation*}
Since the term $ c_n $ must occur in $ c_x $ by assumption, the term $ c_{n+1} $ must occur in $ s_n . c_x $.   Thus the expression above must be zero which verifies the equality.

Subcase 2c: $ g \notin B_{n}^{\Gamma} $, $ w_x \notin S_{n-1} $.
For the left hand side,
$ (\Id T_{-+})(1 \otimes 1 \otimes g) = 1 \otimes x \otimes y $.
Since $ w_x \notin S_{n-1} $, $ 1 \otimes x \otimes y $ must go to zero under $ \adj_{QP}^{\Id} \Id $.

For the right hand side,
\begin{align*}
T_{--} \Id(1 \otimes 1 \otimes g) &= 1 \otimes 1 \otimes s_n x s_n y \\
&= 1 \otimes 1 \otimes s_n(x_1 \otimes \cdots \otimes x_n \otimes 1)c_x w_x s_n(y_1 \otimes \cdots y_n \otimes 1) c_y w_y\\
&= 1 \otimes 1 \otimes (x_1 \otimes \cdots \otimes 1 \otimes x_n)s_n c_x w_x s_n(y_1 \otimes \cdots y_n \otimes 1) c_y w_y\\
&= 1 \otimes 1 \otimes (x_1 \otimes \cdots \otimes 1 \otimes x_n)(s_n .c_x) s_n w_x s_n(y_1 \otimes \cdots y_n \otimes 1) c_y w_y\\
&= 1 \otimes 1 \otimes (x_1 \otimes \cdots \otimes 1 \otimes x_n)(s_n .c_x) (s_n w_x s_n.(y_1 \otimes \cdots y_n \otimes 1)) s_n w_x s_n w_y (w_y . c_y).\\
\end{align*}
This goes to zero under $ \Id \adj_{QP}^{\Id} $ since $ s_n w_x s_n w_y $ will go to zero because the term $ s_n $ must be present since $ w_x \notin S_{n-1} $ so the two occurrences of $ s_n $ will not cancel.
\end{proof}

\begin{equation}
\label{isotopy2}
\begin{tikzpicture}
   \draw[<-,thick] (0.5,1) .. controls (0.5,0) and (1.5,0) .. (1.5,1);
  \draw[<-,thick] (1,1) to (0.5,0);
  \draw (2.25,0.5) node {=};
  \draw[<-,thick] (3,1) .. controls (3,0) and (4,0) .. (4,1);
  \draw[<-,thick] (3.5,1) to (4,0);
   \draw[<-,thick] (6.5,1) .. controls (6.5,0) and (7.5,0) .. (7.5,1);
  \draw[->,thick] (7,1) to (6.5,0);
  \draw (8.25,0.5) node {=};
  \draw[<-,thick] (9,1) .. controls (9,0) and (10,0) .. (10,1);
  \draw[->,thick] (9.5,1) to (10,0);
\end{tikzpicture}
\end{equation}

\begin{prop}
\label{repisotopy2}
There are equalities of natural transformations corresponding to the diagrams in ~\eqref{isotopy2}
\begin{enumerate}
\item $ (T \Id) \circ (\Id \adj_{\Id}^{PQ}) = (\Id T_{-+}) \circ (\adj_{\Id}^{PQ} \Id) $
\item $ (T_{-+} \Id) \circ (\Id \adj_{\Id}^{PQ}) = (\Id T_{--}) \circ (\adj_{\Id}^{PQ} \Id) $.
\end{enumerate}
\end{prop}

\begin{proof}
The first equation follows easily from Proposition ~\ref{repisotopy1} and the first part of Proposition ~\ref{repisotopy5}.
The second equation follows from Proposition ~\ref{repisotopy1} and the second part of Proposition ~\ref{repisotopy5}.
\end{proof}

\begin{equation}
\label{isotopy3}
\begin{tikzpicture}
  \draw[->,thick] (0.5,0) .. controls (0.5,1) and (1.5,1) .. (1.5,0);
  \draw[<-,thick] (1.5,1) to (1,0);
  \draw (2.25,0.5) node {=};
  \draw[->,thick] (3,0) .. controls (3,1) and (4,1) .. (4,0);
  \draw[<-,thick] (3,1) to (3.5,0);
  \draw[->,thick] (6.5,0) .. controls (6.5,1) and (7.5,1) .. (7.5,0);
  \draw[->,thick] (7.5,1) to (7,0);
  \draw (8.25,0.5) node {=};
  \draw[->,thick] (9,0) .. controls (9,1) and (10,1) .. (10,0);
  \draw[->,thick] (9,1) to (9.5,0);
\end{tikzpicture}
 \end{equation}

\begin{prop}
\label{repisotopy3}
There are equalities of natural transformations corresponding to the diagrams in ~\eqref{isotopy3}
\begin{enumerate}
\item $ (\adj_{PQ}^{\Id} \Id) \circ (\Id T_{+-}) = (\Id \adj_{PQ}^{\Id}) \circ (T \Id) $
\item $ (\adj_{PQ}^{\Id} \Id) \circ (\Id T_{--}) = (\Id \adj_{PQ}^{\Id}) \circ (T_{+-} \Id) $.
\end{enumerate}
\end{prop}

\begin{proof}
The first equality follows easily from Proposition ~\ref{repisotopy1} and the first part of Proposition ~\ref{repisotopy4}.
The second equality follows from Proposition ~\ref{repisotopy1} and the second part of Proposition ~\ref{repisotopy4}.
\end{proof}

\begin{equation}
\label{isotopy6}
\begin{tikzpicture}[>=stealth]
\draw (0,0) .. controls (1,1) .. (1,2)[->][thick];
\draw (1,0) .. controls (0,1) .. (0,2)[->][thick];
\draw (2,0) .. controls (2,1) .. (3,2)[->][thick];
\draw (3,0) .. controls (3,1) .. (2,2)[->][thick];
\draw (1.5,1) node{$ \cdots$};

\draw (4,1) node{$ = $};
\draw (5,0) .. controls (5,1) .. (6,2)[->][thick];
\draw (6,0) .. controls (6,1) .. (5,2) [->][thick];
\draw (8,0) .. controls (7,1) .. (7,2) [->][thick];
\draw (7,0) .. controls (8,1) .. (8,2) [->][thick];
\draw (6.5,1) node{$ \cdots$};
\end{tikzpicture}
\end{equation}

\begin{prop}
\label{repisotopy6}
There is an equality of natural transformations corresponding to the diagrams in ~\eqref{isotopy6}.  That is, if $ |k-l|>1 $ and $ 1 \leq k,l \leq r $, then as maps from
$ P(n+r) \circ \cdots \circ P(n) $ to itself, there is an equality:
\begin{equation*}
(\underbrace{\Id \cdots \Id}_{l-1} T \Id \cdots \Id) \circ (\underbrace{\Id \cdots \Id}_{k-1} T \Id \cdots \Id) = (\underbrace{\Id \cdots \Id}_{k-1} T \Id \cdots \Id) \circ (\underbrace{\Id \cdots \Id}_{l-1} T \Id \cdots \Id).
\end{equation*}
\end{prop}

\begin{proof}
This follows from the fact that elements $ s_k $ and $ s_l $ commute in the Hecke-Clifford algebra if $ |k-l| > 1 $.
\end{proof}

\subsubsection{A monoidal functor}
\begin{theorem}
There is a monoidal functor $ F \colon \mathcal{H}_{\Gamma}^- \rightarrow \bigoplus_{n \geq 0} C_n^{\Gamma} $.
\end{theorem}

\begin{proof}
This follows from all the relations checked above.
\end{proof}

\section{Main result}
\label{mainresult}
We have already proved that relations in the twisted Heisenberg algebra lift to isomorphisms in the diagrammatic category $ \mathcal{H}_{\Gamma}^{-} $.  We still must show that the Grothendieck group of this category is isomorphic to the Heisenberg algebra.  We follow \cite[Section 7]{CL} as most of the proofs go through nearly identically.

Let $ \mathcal{F}^{\Gamma} $ be the Fock space associated to $ \Gamma $ which up to isomorphism is the unique irreducible module for $ \mathfrak{h}_{\Gamma}^{-} $.  Then there is a sequence of maps
\begin{equation}
\label{sequenceofmaps}
\mathfrak{h}_{q,\Gamma}^- \xrightarrow{\pi} K_0(\mathcal{K}_{\Gamma}^{-}) \xrightarrow{K_0(\eta)} K_0(\mathcal{H}_{\Gamma}^{-}) \rightarrow \text{End}(\bigoplus_{n \geq 0} K_0(\mathcal{C}_n^{\Gamma})) \cong \text{End}(\mathcal{F}^{\Gamma}).
\end{equation}

The next lemma is a consequence of the degree conventions and relations in the diagrammatic category.
\begin{lemma} ~\cite[Lemma 3]{CL}
\label{degreelemma}
There are no negative degree endomorphisms of $ \prod_i \widetilde{P}_i^{m_i} \prod_i \widetilde{Q}_i^{n_i} $ while the algebra of degree zero endomorphisms is isomorphic to
$ \otimes_i \mathbb{S}_{m_i} \otimes \otimes_i \mathbb{S}_{n_i} $.
\end{lemma}

The categorified Heisenberg relations along with Lemma ~\ref{degreelemma} give the next proposition.
\begin{prop} ~\cite[Proposition 5]{CL}
\label{degreeprop}
Any object in $ \mathcal{K}_{\Gamma}^{-} $ decomposes uniquely into a finite direct sum of indecomposable objects which are of the form
$ \prod_i \widetilde{P}_i^{\lambda_i} \prod_i \widetilde{Q}_i^{\mu_i} $ where each $ \lambda_i $ and $ \mu_i $ is a strict partition.
\end{prop}

\begin{lemma} ~\cite[Lemma 4]{CL}
\label{inductionrestrictionlemma}
In the category $ \mathcal{K}_{\Gamma}^{-} $, there is an isomorphism of objects
\begin{equation*}
\widetilde{P}_i^{\lambda} \widetilde{P}_i \cong \bigoplus_{\mu} \widetilde{P}_i^{\mu}
\end{equation*}
where the direct sum is over all strict partitions $ \mu $ such that $ |\mu| = |\lambda|+1 $.
\end{lemma}

\begin{proof}
This follows from well know induction and restriction rules for irreducible representations of the Hecke-Clifford algebra.  See for example \cite[Proposition 2]{N2}.
\end{proof}

\begin{corollary} ~\cite[Corollary 1]{CL}
\label{catgenerators}
The elements $ \lbrace [\widetilde{P}_i^{(m)}], [\widetilde{Q}_j^{(n)}] | \forall m,n \geq 0, i,j \rbrace $ generate $ K_0(\mathcal{K}_{\Gamma}^{-}) $ as an algebra.
\end{corollary}
This corollary follows from Proposition ~\ref{degreeprop} and repeated use of Lemma ~\ref{inductionrestrictionlemma}.

We now come to the main theorem whose proof is the same as the proof of the main result in \cite{CL} but we repeat it for the sake of completeness.
\begin{theorem}
The map $ \phi \colon \mathfrak{h}_{q,\Gamma}^- \rightarrow K_0( \mathcal{H}_{\Gamma}^{-}) $ is an isomorphism.
\end{theorem}

\begin{proof}
The action of $ \mathfrak{h}_{\Gamma}^- $ on the Fock space $ \mathcal{F}^{\Gamma} $ is faithful which implies that $ \pi $ in \eqref{sequenceofmaps} is injective.
By Corollary ~\ref{catgenerators}, the classes $ [\widetilde{P}_i^{(m)}] $ and $ [\widetilde{Q}_j^{(n)}] $ generate $ K_0(\mathcal{K}_{\Gamma}^{-}) $.
Since $ \pi(p_i^{(m)}) = [\widetilde{P}_i^{(m)}] $ and $ \pi(q_j^{(n)}) = [\widetilde{Q}_j^{(n)}] $, the map $ \pi $ is surjective which shows that it is an isomorphism.

By an argument similar to that of the proof of Proposition ~\ref{degreeprop} (see ~\cite[Proposition 5]{CL}), $ K_0(\eta) $ is surjective.  Again, since the composite of maps in \eqref{sequenceofmaps} is injective,
$ K_0(\eta) $ must be injective as well which proves it's an isomorphism and implies the theorem.
\end{proof}


\section{References}


\def\refname{}

\end{document}